\documentclass[12pt]{article}
\usepackage{amsmath}
\usepackage{amsxtra}
\usepackage{amscd}
\usepackage{amsthm}
\usepackage{thmtools}
\usepackage{amsfonts}
\usepackage{amssymb}
\usepackage{eucal}
\usepackage[all]{xy}
\usepackage{graphicx}
\usepackage{pifont}
\usepackage{comment}
\usepackage{tikz-cd}
\usepackage{enumitem}
\usepackage{latexsym}
\usepackage{framed}
\usepackage{fullpage}
\usepackage{hyperref}
    \hypersetup{colorlinks=true,citecolor=blue,urlcolor =black,linkbordercolor={1 0 0}}

\usepackage[noabbrev,capitalize]{cleveref}

\allowdisplaybreaks[1]

\theoremstyle{plain}
\newtheorem{theorem}{Theorem}[section]
\newtheorem{corollary}[theorem]{Corollary}
\newtheorem{lemma}[theorem]{Lemma}
\newtheorem{proposition}[theorem]{Proposition}

\newtheorem{conjecture}{Conjecture}

\theoremstyle{definition}
\newtheorem{definition}[theorem]{Definition}

\newtheorem{remark}[theorem]{Remark}

\newcommand\nc{\newcommand}
\nc\on{\operatorname}
\nc\renc{\renewcommand}
\nc\ssec{\subsection}
\nc\sssec{\subsubsection}
\nc{\mf}{\mathfrak}
\nc\N{{\mathbf N}}
\nc\Z{{\mathbf Z}}
\nc\Q{{\mathbf Q}}
\nc\R{{\mathbf R}}
\nc\C{{\mathbf C}}
\nc\T{{\mathbf T}}
\nc\D{{\mathbf D}}
\nc{\G}{{\mathbf G}}
\nc\mO{{\mathcal O}}
\nc{\mc}{\mathcal }
\nc{\bb}{\mathbf }
\nc{\HH}{\mathrm{H}}
\nc{\eps}{\varepsilon}
\nc\F{{\mathbf F}}
\nc{\tG}{\widetilde{G}}
\nc{\tg}{\widetilde{\mf{g}}}
\nc{\Art}{\on{Art}}
\nc{\HC}{\on{HC}}

\nc{\Hom}{\on{Hom}}
\nc{\Aut}{\on{Aut}}
\nc{\End}{\on{End}}
\nc{\Spec}{\on{Spec}}
\nc{\Reg}{\on{Reg}}
\nc{\Specm}{\on{Specm}}
\nc{\Spl}{\on{Spl}}
\nc{\Gal}{\on{Gal}}
\nc{\Cl}{\on{Cl}}
\nc{\GL}{\on{GL}}
\nc{\SL}{\on{SL}}
\nc{\Frob}{\on{Frob}}
\nc{\Mod}{\on{Mod}}
\nc{\Ind}{\on{Ind}}
\nc{\Ord}{\on{Ord}}
\nc{\Ext}{\on{Ext}}
\nc{\BM}{\on{BM}}

\nc{\ip}[1]{\langle #1 \rangle}
\nc{\ips}[1]{\langle #1,#1 \rangle}

\nc{\cis}{\cos\theta + i \sin\theta}
\nc{\Real}{\on{Re}}
\nc{\Imag}{\on{Im}}
\nc{\Res}{\on{Res}}

\nc\A{{\mathbf A}}
\nc\BP{{\mathbf P}}
\nc{\Proj}{\on{Proj}}

\nc\fa{{\mathfrak a}}
\nc\fp{{\mathfrak p}}
\nc\fq{{\mathfrak q}}
\nc\fm{{\mathfrak m}}
\nc\pt{\mathrm{pt}}
\nc{\bd}{\mathbf{d}}
\nc{\disc}{\on{disc}}
\nc{\Tr}{\on{Tr}}

\nc\ol{\overline}
\nc{\ul}{\underline}
\nc\wt{\widetilde}
\nc{\wh}{\widehat}
\nc{\one}{{\mathbf{1}}}

\nc{\id}{\mathrm{id}}
\nc{\uHom}{\ul\Hom}
\nc{\tHom}{\ul\uHom}
\nc{\la}{\on{la}}

\nc\E{{\mathbf E}}
\nc\CC{{\mathcal C}}

\nc\abs[1]{| #1 |}

\title{Infinitesimal characters for the completed cohomology of $\GL_n$ over CM fields}

\author{Jelena Ivan\v{c}i\'c and Vaughan McDonald}

\date{}

\begin{document}

\maketitle


\begin{abstract}
    Let $p$ be a prime, and let $F$ be a CM field containing an imaginary quadratic field in which $p$ splits. We show that the locally analytic vectors of Hecke eigenspaces in the ($p$-adic) completed cohomology of $\GL_n/F$ localized at a non--Eisenstein decomposed generic maximal ideal admit infinitesimal characters determined by the Hodge--Tate--Sen weights of the corresponding Galois representations, thus confirming a conjecture of Dospinescu--Pa\v{s}k\={u}nas--Schraen in this case.
\end{abstract}

\tableofcontents

\section{Introduction}
Let $p$ be a prime number and $L/\Q_p$ a finite extension.
For a connected reductive group $G/\Q$ and a compact open level subgroup $K^p \subset G(\A_{\Q}^{p, \infty})$, the completed cohomology groups defined by Emerton \cite{emerton2006interpolation} $\wt{\HH}^i := \wt{\HH}^i(K^p, L)$ have proven to be a fruitful setting for studying the $p$-adic Langlands program. These are $p$-adic Banach spaces with a continuous action of both the group $G(\Q_p)$ and a Hecke algebra $\T^S(K^p)$, and contain rich information about the relation between ($p$-adic) automorphic forms and Galois representations. One recent version of this interplay is the study of infinitesimal characters, as systematically initiated by Dospinescu--Pa\v{s}k\={u}nas--Schraen in \cite{DPS25}. Let us denote by $\wt{\HH}^{i, \la}$ the space of locally analytic vectors for $G(\mathbf Q_p)$ in the completed cohomology $\wt{\HH}^{i}$ -- this has an action of the center of the universal enveloping algebra $Z(\mf{g})_{\mathbf Q_p}$ of $\mf{g}= \on{Lie}(G)$. This action captures infinitesimal characters seen on the automorphic side, which by the Langlands program should be related to $p$-adic Hodge theoretic properties of the corresponding local Galois representations. In \cite{DPS25}, this is formulated as follows: given a Galois representation $\rho: \Gal_F \to {}^LG_f(\overline{\mathbf{Q}}_p)$ ($L$--parameter), the authors construct a character $\zeta_{\rho}^C:Z(\mf{g})_{\mathbf Q_p} \to \overline{\mathbf{Q}}_p$ (using the Sen operator of $\rho$ and a key shift from the $L$-group to the $C$-group) which encodes (up to this shift) the Hodge--Tate--Sen weights of $\rho$. The authors of \textit{loc.\,cit.} then conjecture that the action of $Z(\mf{g})_{\mathbf Q_p}$ on $\wt{\HH}^{i, \on{la}}$ (up to a localization or passing to a Hecke eigenspace) is given by $\zeta^{C}_{\rho}$ for a suitable choice of a $L$--parameter and a shift to the $C$--group, both choices being governed by the global Langlands program.

More precisely, given a continuous ring homomorphism $x: \T^S(K^p)[1/p] \to \overline{\mathbf Q}_p$ with kernel $\mf{m}_x$, the global Langlands program predicts one can construct a Galois representation $\rho_x: \Gal_{\Q,S} \to {}^LG_f(\overline{\mathbf Q}_p)$ such that, at unramified primes $\ell\notin S$, the conjugacy class of $\rho_x(\Frob_{\ell})$ is related to the restriction of $x$ to the local Hecke algebra at $\ell$ under the Satake isomorphism.\footnote{See \cite[9.13]{DPS25} for a formulation for $C$--parameters, or \cite[Conjecture 3.2.1]{buzzard2014conjectural} (interpreted for any Hecke eigensystem and not just automorphic representations) for a formulation for $L$--parameters.} The recipe of \cite{DPS25} produces from $\rho_x$ a map $\zeta_{\rho_x}^C: Z(\mf{g})_{\mathbf Q_p} \to \overline{\mathbf Q}_p$. Their conjecture then says the following:
\begin{conjecture}[{\cite[Conjecture 9.34]{DPS25}}]
\label{DPSconjecture}
Let $x: \T^S(K^p)[1/p] \to \overline{\mathbf Q}_p$ be a continuous ring homomorphism with the corresponding Galois representation $\rho_x$ and $i \ge 0$ an integer. Then $Z(\mf{g})_{\mathbf Q_p}$ acts on $(\wt{\HH}^i[\mf{m}_x])^{\la}\otimes_{\T^S(K^p), x}\ol{\Q}_p$ via $\zeta_{\rho_x}^C$.
\end{conjecture}
In order to emphasize that the ``infinitesimal character'' $\zeta_{\rho}^C$ comes from Hodge-theoretic properties of the Galois side, we propose the following terminology.\footnote{We thank Peter Scholze for suggesting this.} 
\begin{definition}
    Let $A$ be a $p$-adic Banach space and $\rho: \Gal_F \rightarrow {}^LG_f(A)$ a continuous Galois representation. 
    We call the character $\zeta^C_{\rho}$ attached to $\rho$ by \cite{DPS25} a \textbf{generalized Hodge-Tate-Sen character} of $\rho$. We use the same terminology for characters $\zeta^C_{\theta}$ attached to Galois pseudocharacters $\theta$ as in \cite{PasQua25} (which is a generalization of \cite{DPS25}).
\end{definition}
Our main theorem confirms Conjecture \ref{DPSconjecture} (localized at a suitable maximal ideal of the Hecke algebra) when $G = \on{Res}_{F/\Q}\GL_n$ and $F$ is a CM field.
\begin{theorem}[Theorem \ref{theorem: main theorem v2}]
\label{theorem: Main Theorem}
    Suppose $F$ is a CM field which contains an imaginary quadratic field $F_0$ in which $p$ splits.\footnote{The result is slightly stronger: it suffices to assume that every place $\overline{v} \mid p$ of the maximal totally real subfield $F^+ \subset F$ splits in $F$, instead of $p$ splitting in an imaginary quadratic subfield. For further details on the assumptions, see Remark \ref{remark: discussion on ass.}} Let $\mf{m}\subset \T^S(K^p)$ be a maximal ideal which is non--Eisenstein and decomposed generic. Then, for any $i \ge 0$, the action of $Z(\mf{g})_{\Q_p}$ on $\wt{\HH}_{\mf{m}}^{i,\la}$ is given by the expected generalized Hodge-Tate-Sen character.
\end{theorem}
See Theorem \ref{theorem: main theorem v2} for a precise statement and Section \ref{section: Setup}, following \eqref{eq: Scholze mod p Galois representation}, for definitions of non--Eisenstein and decomposed generic. In Remark \ref{remark: discussion on ass.} below, we briefly discuss the assumptions on $F$ and $\mf{m}$ that we impose.
\begin{remark}
    We confirm (cases of) Conjecture \ref{DPSconjecture} by specializing Theorem \ref{theorem: Main Theorem} with respect to any continuous character $x: \T^S(K^p)_{\mf{m}}[1/p]\rightarrow \overline{\Q}_p$ (see Remark \ref{Remark: actually proving the main theorem}).
\end{remark}
\begin{remark}
    In their follow-up paper \cite{DPS23}, Dospinescu--Pa\v{s}k\={u}nas--Schraen proved that if $V$ is an admissible Banach representation of $G(\Q_p)$ such that $V^{\la}$ admits an infinitesimal character, then the Gelfand--Kirillov dimension of $V$ is strictly less than $2\dim \on{Fl}_{G_{\ol{\Q}_p}}$, where $\on{Fl}_{G_{\ol{\Q}_p}}$ denotes the flag variety (see Theorem 6.1 of \textit{loc.\,cit.}). In our setting, we then get that if $G = \Res_{F/\Q}\GL_n$ and $V^i = \wt{\HH}_{\mf{m}}^i[\mf{m}_x]$ then this result and Theorem \ref{theorem: Main Theorem} imply that the Gelfand--Kirillov dimension of $V^i$ is at most $[F:\Q] n(n-1) - 1$, as a representation of $G(\Q_p) = \prod_{\nu\mid p}\GL_n(F_{\nu})$.
\end{remark}

\cite[Sections 9.6--9.10]{DPS25} previously proved Conjecture \ref{DPSconjecture} in some cases when $G$ is related to a Shimura variety, the cohomological degree $i$ is the middle degree $d$, and $\mf{m} \subset \T^S(K^p)$ is such that $\HH_{\mf{m}}^j = 0$ for $j \neq d$.
As mentioned in \cite[Remark 9.36]{DPS25}, the key assumption for their work to hold is that $\wt{\HH}_{i, \mf{m}}$, the completed homology group localized at $\mf{m}$, is a projective module over the Iwasawa algebra $L[[K_p]]:=L \otimes_{\mathbf Z_p}\mathbf Z_p[[K_p]]$ for $K_p \subset G(\Q_p)$ a compact open subgroup (in particular, the defect $l_0=0$).\footnote{The suggestion made in \cite[Remark 9.36]{DPS25} to use Theorems 8.5 and 8.6 in \textit{loc.\,cit.} on patched complexes for $\Res_{F/\mathbf{Q}}\on{PGL}_2$ relies on a still conjectural local--global compatibility \cite[Conjecture 5.1.12]{GeeNewton}. Our argument does not (directly) require Theorems 8.5 and 8.6 and works for $\Res_{F/\mathbf Q}\GL_n$ for all $n\geq 2$.} In particular, any setting where $G$ admits a generic torsion vanishing result \`{a} la Caraiani--Scholze will be amenable to the methods of \textit{loc.\,cit.} in many cases.
But crucially, this projectivity property does \textbf{not} hold for $G = \Res_{F/\Q}\GL_n$ in any degree (see the proof of \cite[Theorem 3.4]{CalEm09}). As such, our result seems to be the first of its kind when the defect $l_0$ is bigger than zero.

Even though the methods of \cite{DPS25} cannot be applied directly to our case, we are able to apply these methods indirectly, by first proving a similar statement for the completed cohomology of unitary Shimura varieties and then deducing the statement for $G=\Res_{F/\mathbf Q} \GL_n$. We note that the work of \cite{PasQua25}, which generalized\footnote{The construction from \cite{PasQua25} agrees with the construction from \cite{DPS25} whenever a Galois pseudocharacter arises from a Galois representation, even for general coefficients and not just $\overline{\Q}_p$. The authors of \textit{loc.\,cit.} will add this remark in a new version of their paper.}  the construction made in \cite{DPS25} from Galois representations to Galois pseudocharacters, was crucial for us to be able to carry out this argument. Let us now sketch our strategy for proving Theorem \ref{theorem: Main Theorem}.

\subsection{On the method of proof}
To work around the obstacles mentioned above (in particular $l_0>0$), we follow the lead of recent work on the Langlands program over CM fields and relate the completed cohomology groups $\wt{\HH}_{\mf{m}}^{i, \la}$ to the completed cohomology of a rank $2n$ unitary group $\wt{G}/F^+$, which contains $\Res_{F/F^+}\GL_n$ as a Levi subgroup. Such a strategy has been exploited previously to prove properties about automorphic Galois representations for $\GL_n/F$ (see among others \cite{HLTT16,Sch15, 10AuthorPaper, CN2023, AC24, Hevesi23}). Most such papers proceed by looking at cohomology with \textit{torsion} coefficients, with the exception of second-named author's previous work \cite{mcdonald2025eigenvarieties}. Roughly speaking, all of these previous works realize the cohomology of $\GL_n/F$ (in \textit{any degree}\footnote{For a general $\GL_n/F$ it is still not provably known for which degrees $\wt{\HH}_{\mf{m}}^i \neq 0$.}) as a subquotient of the cohomology of $\tG$ in middle degree, which generally is tightly connected to classical automorphic forms. This process is referred to as ``degree shifting'', and is often filled with technical obstacles.

In this work, we still prove that a certain middle degree cohomology of the unitary group $\wt{\HH}_{\tG}^d := \wt{\HH}^d(\wt{K}^p, L)_{\wt{i}(\wt{\mf{m}})}$ is tightly controlled by automorphic forms (see the proof of Theorem \ref{unitary inf action}), and thus has the infinitesimal action given by the desired generalized Hodge-Tate-Sen character. In this case, the completed homology is not projective over $L[[\wt{K}_p]]$, but is still \textit{torsion-free}, and thus we can run the machinery of \cite{DPS25}. Technically, the maximal ideal $\wt{\mf{m}}$ we work with is \textit{Eisenstein}, so we only have a Galois pseudocharacter attached to the localized Hecke algebra. For this reason, we must use the recent work of Pa\v{s}k\={u}nas--Quast \cite{PasQua25}, where they construct generalized Hodge-Tate-Sen characters $\wt{\zeta}^{C}$ for Galois pseudocharacters.

On the other hand, in our setting the degree shifting comes somewhat more easily, as Poincar\'{e} duality for completed (co)homology directly relates $\wt{\HH}_{\tG}^{d,\wt{G}-\on{la}}$ to $\wt{\HH}_c^{i+1}(\wt{K}^p, L)_{\wt{\mf{m}}}^{\wt{G}-\on{la}}$ for any $i \le d-1$, which by our assumptions on the maximal ideals $\wt{\mf{m}}$ is directly related to boundary cohomology in degree $i$. Then a suitable invariant subspace of the boundary cohomology in degree $i$ contains the desired $\wt{\HH}_{\mf{m}}^{i, \on{la}}$, and the information of the infinitesimal action on $\wt{\HH}_{\tG}^{d, \wt{G}-\on{la}}$ passes through these steps.

Two more steps after this are needed: the above process does not pin down the entire action of the center $Z(\Res_{F/\Q}\mf{gl}_n)_{\mathbf{Q}_p}$, but instead a subalgebra coming from the image of the unnormalized Harish-Chandra map (with respect to the opposite parabolic) $\HC^{\overline{P}}: Z(\Res_{F^+/\mathbf Q}\wt{\mf{g}})_{\mathbf{Q}_p} \to Z(\Res_{F/\Q}\mf{gl}_n)_{\mathbf Q_p}$ where $\wt{\mf{g}}:=\on{Lie}(\tG)$. First, we need to show that the constructions of the generalized Hodge-Tate-Sen characters $\wt{\zeta}^C$ for $\tG$ and $\zeta^C$ for $G$ are compatible under $\HC^{\overline{P}}$. Second, we need to address the problem of $\HC^{\overline{P}}$ not being surjective. We overcome this by twisting the maximal ideal $\mf{m}$ by a suitable character $\chi$, which also twists the image of the map $\HC^{\overline{P}}$, and we show that $\HC^{\overline{P}}$ and its twist actually generate all of $Z(\Res_{F/\Q}\mf{gl}_n)_{\mathbf Q_p}$. Intuitively, the cohomology of $\tG$ only knows about the Galois representation $\rho \oplus \rho^{c,\vee}(1-2n)$, but we want to pin down the weights of $\rho$. By also considering $\chi\rho \oplus \chi^{c,\vee}\rho^{c,\vee}(1-2n)$, the weights of $\rho$ are pinned down.

\begin{remark}\label{remark: discussion on ass.}
        Let us briefly comment on the assumptions that we make in Theorem \ref{theorem: Main Theorem}. 
    \begin{enumerate}[label=(\arabic*)]
        \item Just as in the paper \cite{10AuthorPaper} (among others), the assumption that \textbf{$F$ contains an imaginary quadratic field} is needed for some automorphic inputs from \cite{Shin14} (more precisely, the local--global compatibility at $p$ for classical points in completed cohomology for $\wt{G}$), so that the work is independent of the twisted weighted fundamental lemma. \textit{There is work in progress by Connor Halleck-Dub\'e on deducing the twisted weighted fundamental lemma from the results they prove in their PhD thesis; such a result could be used to remove the assumption that $F$ contains an imaginary quadratic field.}
        \item The assumption that \textbf{$p$ splits in $F_0$} implies that every place $\overline{v} \mid p$ of $F^+$, the maximal totally real subfield of $F$, splits in $F$. This is later used to note that the unitary group $\wt{G}$ is split after base change to $F^+_{\overline{v}}$ for all $\overline{v}\mid p$. Therefore, \textbf{we can replace this assumption with: every place $\overline{v} \mid p$ of $F^+$ splits in $F$}. 
        \item We use the assumption that \textbf{$\mf{m}$ is non--Eisenstein} to establish a clean relationship between $\wt{\HH}^{i, \on{la}}_{\mf{m}}$, the cohomology of interest, and the boundary cohomology $\wt{\HH}^i_{\partial}(\wt{K}^p,L)_{\wt{\mf{m}}}^{\wt{G}-\on{la}}$ (i.e. that the former is a subquotient of the latter).
        \item  The assumption that \textbf{$\mf{m}$ is decomposed generic} is only used for the vanishing results of Caraiani--Scholze \cite{CS24} (as generalized in \cite{KoshikawaGeneric} to remove the assumption that $F^+ \neq \Q$). These vanishing results are used in two ways. Firstly, we use them to compare $\wt{\HH}^{i+1}_c(\wt{K}^p,L)_{\wt{\mf{m}}}^{\tG-\on{la}}$ and $\wt{\HH}^{i}_{\partial}(\wt{K}^p,L)_{\wt{\mf{m}}}^{\tG-\on{la}}$, hence $\wt{\HH}^{i+1}_c(\wt{K}^p,L)_{\wt{\mf{m}}}^{\tG-\on{la}}$ and $\wt{\HH}^{i, \on{la}}_{\mf{m}}$ by the previous point. 

        Secondly, we use the vanishing results to prove that the infinitesimal action on $\wt{\HH}^{*}_c(\wt{K}^p,L)_{\wt{\mf{m}}}^{\tG-\on{la}}$ is given by a generalized Hodge-Tate-Sen character, as constructed in \cite{DPS25} (and generalized in \cite{PasQua25}). Here, we need a density statement concerning (suitable) locally algebraic vectors and the decomposed generic assumption on $\mf{m}$ provides this density, as under this assumption we can relate the locally analytic vectors of the localized completed cohomology to $\mathcal{C}^{\on{la}}(\wt{K}_p,L)^{\oplus m}$ for some integer $m$.
    \end{enumerate}
\end{remark}
\subsection{Overview of sections}
We review the contents briefly. In Section \ref{section: Setup} we collect preliminaries on Hecke algebras, completed cohomology of $\GL_n/F$ and $\tG$, and relate the infinitesimal action on the cohomology of $\tG$ to the action on the cohomology of $\GL_n/F$ via the Harish-Chandra map. In Section \ref{section: Infinitesimal action for unitary group} we prove Conjecture \ref{DPSconjecture} for the cohomology of $\tG$ in middle degree and (by Poincar\'{e} duality) for the compactly supported cohomology in \textit{all degrees}, when localized at a certain Eisenstein maximal ideal $\wt{\mf{m}}$, with the key inputs coming from the vanishing result of \cite{CS24} and the work \cite{PasQua25}. In Section \ref{section: Compatibility between rhoTilde and rho} we prove that the expected character $\zeta^C: Z(\on{Res}_{F/\Q}\mf{gl}_n)_{\mathbf{Q}_p} \to \T^S(K^p)^{\on{rig}}_{\mf{m}}$ for $\wt{\HH}^{\ast, \on{la}}_{\mf{m}}$ and (a version of) the character $\wt{\zeta}^C_{\on{BM}}: Z(\Res_{F^+/\mathbf Q}\wt{\mf{g}})_{\mathbf Q_p} \to \wt{\T}^S_{\on{BM}}(\wt{K}^p)^{\on{rig}}_{\wt{\mf{m}}}$ for $\wt{\HH}_c^{*}(\wt{K}^p, L)^{\wt{G}-\on{la}}_{\wt{\mf{m}}}$ are compatible under the Harish-Chandra map (see Theorem \ref{theorem: Harish Chandra + Inf Char compatibility}). Lastly, in Section \ref{section: Twisting and proof of the main theorem} we do the twisting argument to exhaust $Z(\on{Res}_{F/\Q}\mf{gl}_n)_{\mathbf Q_p}$ and finish the proof of Theorem \ref{theorem: Main Theorem}.

\subsection{Acknowledgments}
First, we heartily thank our advisors, Peter Scholze and Richard Taylor, for their continuous advice and support.
We thank Vytautas Pa\v{s}k\={u}nas for putting us in contact at the $p$-adic Langlands program conference in Luminy, which has led to the current work; we also thank the organizers of that conference. We are especially grateful to Lue Pan for an inspiring discussion, which motivated a simplification of the original intended strategy. Furthermore, we thank Lambert A'Campo, Arthur-C\'esar le Bras, Juan Esteban Rodr\'iguez Camargo, Toby Gee, Dongryul Kim, James Newton, and Vytautas Pa\v{s}k\={u}nas for helpful conversations on many different aspects related to this work. We thank Peter Scholze and Vytautas Pa\v{s}k\={u}nas for their comments on an earlier draft. The first--named author thanks the Max Planck Institute for Mathematics for its hospitality.

\subsection{Notations and conventions}
Let $p$ be a prime number, which we will fix throughout.
Let $F/\Q$ be a number field. We let $S_p(F)$ denote the set of places of $F$ dividing $p$. We let $F_p:= F\otimes_{\Q}\Q_p$, and $\mO_{F_{p}}:= \mO_F\otimes_{\Z}\Z_p$ with $\mO_F$ the ring of integers of $F$. Let $S$ be a finite set of finite places of $F$. We denote by $\Gal_{F,S}$ the Galois group of the maximal extension of $F$ which is unramified outside of $S$. If $F$ is an imaginary CM field with $F^+\subset F$ its maximal totally real subfield, we denote by $c \in \Gal(F/F^+)$ the complex conjugation. For each prime place $\nu$ of $F$, we let $\nu^c$ denote the place $c(\nu)$.

Let $G$ be a connected reductive group over $F$. Since the constructions that we will consider are functorial with respect to the coefficient field, we may work with a coefficient field $L/\mathbf Q_p$ which is both a finite extension and large enough so that there are $[F:\mathbf Q]$ embeddings $F \hookrightarrow L$ and so that the reductive group $G_L$ splits.\footnote{This will be useful for Theorem \ref{unitary inf action} and in Section \ref{section: Twisting and proof of the main theorem}.}

Let $S$ be a set of places of $F$ containing the infinite places. For $K^S \subset G(\A_F^{S})$ a compact open subgroup, $\mc{H}(G^S, K^S)$ denotes the set of compactly supported $K^S$-biinvariant functions $f: G(\A_F^{S}) \to \Z$. Moreover, for a general compact open subgroup $K \subset G(\A_F^{\infty})$, we set $K^S := \prod_{\nu \notin S}K_{\nu} \subset G(\A_F^{S})$, and $K_S := \prod_{\nu \in S\setminus S_{\infty}}K_{\nu} \subset \prod_{\nu \in S\setminus S_{\infty}}G(F_{\nu})$.

We shall use $\mf{g}$ to denote the Lie algebra of $G$ and denote by $Z(\mf{g}):= Z(U((\Res_{F/\mathbf Q}\mf{g})_{\mathbf Q_p}))$ the center of the universal enveloping algebra for $(\Res_{F/\mathbf Q}\mf{g})_{\Q_p}$. Similarly we shall write $U(\mf{g}):=U((\Res_{F/\mathbf Q}\mf{g})_{\mathbf Q_p})$.  Note that $(\Res_{F/\Q} \mf{g})_{\Q_p}= \oplus_{\nu \mid p} \Res_{F_\nu/\Q_p} \mf{g}_{F_\nu}$, $Z(\mf{g})= \otimes_{\nu \mid p} Z(\Res_{F_\nu/\Q_p} \mf{g}_{F_\nu})$ and similarly for $U(\mf{g})$. We refer the reader to \cite[Lemma 4.14]{DPS25} and discussions around it for results on the behaviour of these algebras under a change of the base field.

If $P \subset G$ is a parabolic subgroup of Levi decomposition $P = MN$, and $\mf{p} = \mf{m}\oplus \mf{n}$ is its Lie algebra, we denote by $\HC^P: Z(\mf{g}) \to Z(\mf{m})$ the \textit{unnormalized} Harish-Chandra homomorphism with respect to $\mf{p}$ (see \cite[\textsection 1.3]{Emerton_Jacquet_I}).

Let $K_p \subset G(F_p)$ be a compact open subgroup. We denote by $\mathbf Z_p[[K_p]]:=\varprojlim_{N \subset K_p} \mathbf Z_p[K_p/N]$ the Iwasawa algebra for $K_p$, the limit taken over normal open subgroups $N \subset K_p$, and denote by $L[[K_p]]:=L \otimes_{\mathbf Z_p} \mathbf Z_p[[K_p]]$ its rationalization.\footnote{We warn the reader about the abuse of notation here -- this is not the same as $\varprojlim_{N \subset K_p}  L[K_p/N]$.} We view $K_p$ as a locally $\mathbf Q_p$--analytic manifold and let $D(K_p) = D(K_p,L)$ denote the distribution algebra (see e.g. \cite[Section 2]{schneider2002locally}). For any locally $\Q_p$--analytic group $H$ we let $\mc{C}^{\rm{la}}(H, L) \subset \mc{C}^{\on{cts}}(H, L)$ be the spaces of locally analytic functions and continuous functions, respectively. For continuous functions, we may also take torsion or integral coefficients, or coefficients in an $L$-Banach space. For locally analytic functions, we may take coefficients in an LB space of compact type.
For $H\subset G$ locally $\Q_p$-analytic groups such that $G/H$ is compact and $(\pi,V)$ an $L$-Banach representation of $H$, we let $\on{cts}-\Ind_{H}^G\pi$ denote the continuous induction, which by definition is
$\on{cts}-\Ind_H^{G}(\pi):=\bigl\{f\in \mc{C}^{\on{cts}}(G, V) \ : \ f(hg)=h.f(g) \ \forall h \in H  \bigr\}$. We can make similar definitions for smooth induction and locally analytic induction, denoted by $\on{sm}-\Ind_H^G$ and $\on{la}-\Ind_H^G$, respectively. If $V$ is a locally convex topological $L$-vector space, we let $V'$ denote its strong dual.

We let $\chi_{\on{cyc}}: \Gal_F \to \Z_p^{\times}$ denote the $p$-adic cyclotomic character. In order to match \cite{DPS25} and \cite{PasQua25}, we adopt the convention that the Hodge-Tate-Sen weight $\on{wt}_{\nu, \tau}(\chi_{\on{cyc}}) = 1$ for all $\nu \mid p$ and $\tau: F_{\nu} \hookrightarrow L$. For each finite place $\nu$ of $F$ we let $\on{Frob}_{\nu}$ denote a \textit{geometric} Frobenius element of $\Gal_{F}$. We note that for all $\nu \nmid p$, $\chi_{\on{cyc}}(\on{Frob}_{\nu})=q_{\nu}^{-1}$ where $q_{\nu}$ is the size of the residue field corresponding to $\nu$. 

Let $K/\mathbf Q_p$ be a finite extension. We use the geometric normalization of class field theory so that the Artin map $\Art_K: K^{\times} \xrightarrow{\sim} W_K^{\rm{ab}}$ sends uniformizers to geometric Frobenius elements. We note that $\chi_{\on{cyc}}\circ \Art_K : K^{\times}\rightarrow \Z_p^{\times}$ sends $x$ to $N_{K/\Q_p}(x) \cdot |N_{K/\Q_p}(x)|_p $ where $N_{K/\Q_p}: K^{\times} \rightarrow \Q_p^{\times}$ is the norm map and $|\cdot |_p$ is the standard absolute value on $\Q_p$.

Let $R$ be a ring and $f(x) \in R[x]$ a polynomial of degree $n$ with an invertible constant coefficient $a_0 \in R^{\times}$. We denote by $f^{\vee}(X):=a_0^{-1}X^nf(X^{-1})$. We use $D(R)$ to denote the derived category of $R$--modules.
\newpage

\section{Setup}
\label{section: Setup}
Let $F$ be a CM field which contains an imaginary quadratic field $F_0$ in which $p$ splits. Let $F^+$ be the maximal totally real subfield of $F$. Let $\wt{G}/F^+$ be a quasi-split unitary group, $P=G \ltimes U$ the Siegel parabolic, $G\subset P$ the standard Levi subgroup and $U$ the unipotent radical (recall $G \simeq \Res_{F/F^+}\GL_n$ by \cite[Lemma 5.1]{NT16}). We also denote by $T$ and $B$ the subgroups of diagonal and upper triangular matrices of $\wt{G}$ as usual. See \cite[2.1.1 + 2.2.1]{10AuthorPaper} for precise details on these algebraic groups and locally symmetric spaces that we attach to them. As in \textit{loc.\,cit.} we use the notation $X_{\wt{K}}^{\wt{G}}$ and $X_K^G$ to denote the locally symmetric spaces for $\wt{G}$ and $G$ respectively with level defined by compact open subgroups $\wt{K}\subset \wt{G}(\A_{F^+}^{\infty})$ and $K\subset G(\A_{F^+}^{\infty})$. To simplify, we consider only \textit{good} subgroups $\wt{K}$, i.e. \textit{neat} compact open subgroups of the form $\prod_{\nu\nmid \infty}K_{\nu}$. Furthermore, for such $\wt{K}$, we let $K_P:=\wt{K} \cap P(\A_{F^+}^{\infty})$, $K_U:=\wt{K} \cap U(\A_{F^+}^{\infty})$ and $K_G:=\on{im}(K_P)$ under the natural projection $P(\A_{F^+}^\infty)\twoheadrightarrow G(\A_{F^+}^\infty)$. We say $\wt{K}$ is \textit{decomposed well} with respect to $P=GU$ if $K_P=K_G \ltimes K_U$; equivalently if $K_G=\wt{K} \cap G(\A_{F^+}^{\infty})$. From now on we shall consider only $\wt{K}$ which are decomposed well with respect to $P$ and we let $K=\wt{K}\cap G(\A_{F^+}^\infty)$.\footnote{Any good compact open $K \subset G(\A_{F^+}^{\infty})$ is of this form for $\wt{K}$ which is decomposed well with respect to $P$. To ensure neatness of $\wt{K}_{\ol{\nu}}$, note intersecting with inverse images of finite groups $P(\mO_{F^+}/\ell^n)$ preserves $K_{\ol{\nu}}$.} We shall also consider the opposite parabolic and the opposite unipotent subgroups $\overline{P}=G\ltimes\overline{U} \subset \wt{G}$.

Recall that the group scheme $\wt{G}$ is naturally defined over $\mO_{F^+}$ where all the subgroups introduced above are defined.
We shall denote by $\overline{U}_0:=\overline{U}({\mO_{F_p^+}})$.  
Let $d:=n^2[F^+:\mathbf Q]$ which also equals $\frac{1}{2}\dim_{\mathbf R} X^{\wt{G}}_{\wt{K}}=\dim_{\mathbf R}X^G_K+1$. 

The relationship between locally symmetric spaces for $X^G_{K}$ and $X^{\wt{G}}_{\wt{K}}$ appears through the boundary term $\partial X^{\wt{G}}_{\wt{K}}$. In particular it allows us to realize (completed) cohomology for $G$ as a piece of the (completed) cohomology associated to the boundary, which relates to a (completed) cohomology associated to $\wt{G}$ via the usual long exact sequence that we shall study in the next subsections. This will require localizing at certain maximal ideals for Hecke algebras that we now introduce first. 
\subsection{The big Hecke algebras}\label{subsection: big Hecke}
We first discuss Hecke algebras for $\tG$. Let $\wt{K}^p \subset \tG(\A_{F^+}^{p, \infty})$ be a compact open subgroup, and choose a compact open $K_0 \subset \tG(\mO_{F_p^+})$ such that $\wt{K}^pK_0$ is good and decomposed well with respect to $P$. 

Let $S$ be a finite set of places of $F$ containing all the infinite places, all primes above $p$, all primes where $F$ ramifies, and all primes where a fixed tame level $\wt{K}^p$ is not hyperspecial. Moreover, we assume that $S$ is stable under complex conjugation $c$, i.e. $S=S^c$. Let $\overline{S}$ denote the set of primes of $F^+$ that lie below the primes in $S$. Let $L$ be any finite extension of $\mathbf Q_p$, $\mO_L$ its valuation ring and $\varpi_L$ a uniformizer.

We have the abstract Hecke algebra\footnote{\label{footnote: functorial}This and the subsequent definitions of Hecke algebras depend on the extension $L/\Q_p$. We note that there are, functorial in $L$, natural maps between corresponding Hecke algebras respecting their actions on cohomologies. Since all the constructions we will consider will be naturally functorial in $L$, we will omit $L$ from the notation of the Hecke algebra and it should be clear from the context for which field they are defined.} (which is a commutative ring)
$$\wt{\T}^S_{\on{abs}} = \biggl(\bigotimes_{\ol{\nu} \notin \overline{S}} \mathcal{H}(\wt{G}(F^+_{\ol{\nu}}), \wt{K}_{\ol{\nu}}) \biggr) \otimes_{\mathbf Z} \mO_L
$$ 
For each finite place $\ol{\nu} \notin \overline{S}$ and a place $\nu \mid \ol{\nu}$, for $1\leq i \leq 2n$ we let $\wt{T}_{\nu,i} \in \mathcal{H}(\wt{G}(F^+_{\ol{\nu}}), \wt{K}_{\ol{\nu}}) \otimes_{\mathbf Z}\mathbf Z[q_{\nu}^{-1}]$ be the element $T_{G,\nu,i}$ defined in \cite[Proposition--Definition 5.2]{NT16}. We shall consider polynomials
\[
\wt{P}_{\nu}(X) = X^{2n}-\wt{T}_{\nu,1}X^{2n-1}+\ldots +(-1)^jq_{\nu}^{j(j-1)/2}\wt{T}_{\nu,j}X^{2n-j}+\ldots +
q_{\nu}^{n(2n-1)}\wt{T}_{\nu,2n}
\]
as in \cite[(5.2)]{NT16}. We consider the subalgebra $\wt{\T}^S \subset \wt{\T}^S_{\on{abs}}$ generated by all the coefficients of $\wt{P}_{\nu}(X)$ for all $\nu \mid \ol{\nu} \notin \overline{S}$ (note $q_{\nu}$ is invertible in $\mO_L$ for each $\nu \notin S$).

For each normal open subgroup $\wt{K}_p\subset K_0$ and $m \geq 1$, let us denote by
$$
\wt{\T}^S(\wt{K}_p \wt{K}^p, m):= \on{im}\biggl(\wt{\T}^S \rightarrow \on{End}_{D(\mO_L/\varpi_L^m[K_0/\wt{K}_p])}\bigl(C_{\bullet}(\wt{K}_p\wt{K}^p,m)\bigr)\biggr),
$$
where $C_{\bullet}(\wt{K}_p\wt{K}^p,m)$ is a complex calculating the homology for $X^{\wt{G}}_{\wt{K}_p\wt{K}^p}$ with coefficients $\mO_L/\varpi_L^m$. This is naturally a complex of right modules for $K_0/\wt{K}_p$ (see \cite[Section 2.1]{GeeNewton} or \cite[Section 6.2]{khare2017potential}).
In particular, these complexes come equipped with natural and explicit $\wt{\T}^S\times K_0/\wt{K}_p$--action. We have similar complexes for Borel--Moore and boundary homology as in \cite[Section 5]{caraiani2023vanishing}, and we denote by $\wt{\T}^S_{\on{BM}}(\wt{K}_p\wt{K}^p, m)$ and $\wt{\T}^S_{\partial}(\wt{K}_p\wt{K}^p,m)$ the corresponding images of the Hecke algebra.
\begin{remark}[Relationship between homology and cohomology]\label{coh vs hom}
There are similar complexes $C^{\bullet}(\wt{K}_p\wt{K}^p, m)$ calculating the cohomology (see for example \cite[Section 6.2]{khare2017potential}). As in \cite[Lemma 6.3]{khare2017potential} we have an equality (termwise as complexes)
\begin{equation}\label{eq: duality}
    C^{\bullet}(\wt{K}_p\wt{K}^p, m) = \on{Hom}_{\mO_L/\varpi_L^m}(C_{\bullet}(\wt{K}_p\wt{K}^p, m), \mO_L/\varpi_L^m)
\end{equation}
equivariant for the $\wt{\T}^S \times K_0/\wt{K}_p$--actions on both sides. Here $\wt{\T}^S$ is a commutative algebra and it acts on the RHS of (\ref{eq: duality}) by its natural action on $C_{\bullet}(\wt{K}_p\wt{K}^p, m)$ (which still produces a left action on the RHS). On the other hand, $K_0/\wt{K}_p$ acts on the RHS by acting on the right on $C_{\bullet}(\wt{K}_p\wt{K}^p, m)$.
For $\wt{K}_p$ sufficiently small, the complex $C_{\bullet}(\wt{K}_p\wt{K}^p, m)$ is a complex of free $\mO_L/\varpi_L^m$--modules, so (\ref{eq: duality}) can be seen as a quasi--isomorphism in $D(\mO_L/\varpi_L^m)$:
\begin{equation}
\label{eq: cohomology homology on derived level}
C^{\bullet}(\wt{K}_p\wt{K}^p, m) \simeq R\on{Hom}_{\mO_L/\varpi_L^m}(C_{\bullet}(\wt{K}_p\wt{K}^p, m), \mO_L/\varpi_L^m) \ .
\end{equation}
Furthermore, the complex $C_{\bullet}(\wt{K}_p\wt{K}^p, m)$ is perfect in $D(\mO_L/\varpi_L^m[K_0/\wt{K}_p])$. Using the perfect duality \eqref{eq: cohomology homology on derived level} together with  \cite[\href{https://stacks.math.columbia.edu/tag/07VI}{Tag 07VI}, \href{https://stacks.math.columbia.edu/tag/0A60}{Tag 0A60}]{stacks-project} and the equivariance of \eqref{eq: duality}, one checks that $\on{im}(\wt{\T}^S \rightarrow \on{End}_{D(\mathcal{O}_L/\varpi_L^m[K_0/\wt{K}_p])}(C^{\bullet}(\wt{K}_p\wt{K}^p, m)))$ is naturally identified with $\wt{\T}^S(\wt{K}_p\wt{K}^p,m)$. A similar statement is true for Borel--Moore homology and compactly supported cohomology (see \cite[Section 5]{caraiani2023vanishing}).
\end{remark}
Let us use $?$ to denote $\emptyset, \on{BM}$ or $\partial$. We recall that $\wt{\T}^S_{?}(\wt{K}_p\wt{K}^p,m)$ is a finite $\mathcal{O}_L/\varpi_L^m$--algebra which we equip with the discrete topology. The \textit{big Hecke algebra} is defined to be the ring
\begin{equation}
\label{equation: big Hecke algebra for tG}
    \wt{\T}^S_{?}(\wt{K}^p):=\varprojlim_{m,\wt{K}_p} \wt{\T}^S_{?}(\wt{K}_p\wt{K}^p,m)
\end{equation}
equipped with the projective limit topology. Here, the inverse limit maps come from the functorial maps
$$\on{End}_{D(\Lambda_2)}(C_{\bullet}^?(K_2\wt{K}^p,{m_2})) \rightarrow \on{End}_{D(\Lambda_1)}(C^{?}_{\bullet}(K_2\wt{K}^p,{m_2}) \otimes_{\Lambda_2}\Lambda_1)
$$
where $\Lambda_2=\mathcal{O}_L/\varpi_L^{m_2}[K_0/K_2]$, $\Lambda_1=\mathcal{O}_L/\varpi_L^{m_1}[K_0/K_1]$, $m_2 \geq m_1$ and $K_2 \subset K_1$, together with the natural isomorphism
$C^{?}_{\bullet}(K_2\wt{K}^p,{m_2}) \otimes_{\Lambda_2}\Lambda_1 \simeq C^{?}_{\bullet}(K_1\wt{K}^p,{m_1})$.

Next, we note that $\wt{\T}^S_{\rm{BM}}(\wt{K}^p), \ \wt{\T}^S(\wt{K}^p)$ and $ \wt{\T}^S_{\partial}(\wt{K}^p)$ act on completed compactly supported, usual, boundary cohomology, respectively, and similarly for homology (see the beginning of \ref{subsection: intro to completed} for a definition). Moreover, there are natural morphisms $\wt{\T}^S\rightarrow \wt{\T}^S_?(\wt{K}^p)$ and the long exact sequences \eqref{equation: ExcisionDistinguishedTriangleCompletedCohomology} considered in next subsection \ref{subsection: intro to completed} are equivariant for the $\wt{\T}^S$--action. 
Recall (see \cite[Lemma 2.1.14]{GeeNewton}) that the profinite $\mathcal{O}_L$--algebra $\wt{\T}^S_{?}(\wt{K}^p)$ is complete and semilocal. It has finitely many (open) maximal ideals $\mf{m}_1,\ldots,\mf{m}_r$ and $\wt{\T}^S_?(\wt{K}^p) = \prod_{i=1}^r \wt{\T}^S_?(\wt{K}^p)_{\mf{m}_i}$ with the localization $\wt{\T}^S_?(\wt{K}^p)_{\mf{m}_i}$ being an $\mf{m}_i$--adically complete and separated local ring with residue field a finite extension of $k=\mathcal{O}_L/\varpi_L$: these maximal ideals also come from a finite level Hecke algebra and so are related to maximal ideals of $\wt{\T}^S$. We record this relationship in the next lemma.
\begin{lemma}
    Suppose $\mf{m} \subset \wt{\T}^S_?(\wt{K}^p)$ is an open maximal ideal and let $\mf{m}' \subset \wt{\T}^S$ denote its preimage. Then 
    $$\wt{\HH}^*_?(\wt{K}^p, \mO_L)_{\mf{m}} = \varprojlim_n \wt{\HH}^*_{?}(\wt{K}^p, \mO_L/\varpi_L^n)_{\mf{m}'} \ .$$
    Furthermore, if $\mf{m}' \subset \wt{\T}^S$ is a maximal ideal such that for some $\wt{K}_p$ sufficiently small (contained in a pro-$p$ group) and some $n$ we have that $\wt{\HH}^*_?(\wt{K}_p\wt{K}^p, \mO_L/\varpi_L^n)_{\mf{m}'}\neq 0$, then there exists an open maximal ideal $\mf{m} \subset \wt{\T}^S_?(\wt{K}^p)$ such that the preimage of $\mf{m}$ is $\mf{m}'$.
\end{lemma}
\begin{proof}
    Let $\mf{m}_1, \ldots, \mf{m}_l$ be the open maximal ideals of $\wt{\T}^S_{?}(\wt{K}^p)$ and let $e_1,\ldots,e_l$ be idempotent elements corresponding to the localizations. Label the maximal ideals so that $\mf{m}_1 = \mf{m}$. Then it is not hard to see that for any $\wt{\T}^S_?(\wt{K}^p)$--module $V$, $V_{\mf{m}} \simeq e_1 \cdot V \subset V$ and that $e_1 \cdot V=\{v \in V \ : \ e_1\cdot v=v, \ e_i\cdot v=0  \ \ \ \ \forall i \neq 1 \}$. This condition is stable under limits, so $\wt{\HH}^*_?(\wt{K}^p, \mO_L)_{\mf{m}}= \varprojlim_n \wt{\HH}^*_?(\wt{K}^p, \mO_L/\varpi_L^n)_{\mf{m}}$. To finish the proof of the first claim, we note that $\wt{\HH}^*_?(\wt{K}^p, \mO_L/\varpi_L^n)_{\mf{m}}= \wt{\HH}^*_?(\wt{K}^p, \mO_L/\varpi_L^n)_{\mf{m}'}$ since two localizations agree when viewed as subalgebras of the endomorphism algebra of any (sufficiently small) finite level.
    The second statement follows as in the proof of \cite[Lemma 2.1.14]{GeeNewton}.
\end{proof}
\begin{remark}\label{remark: loc at max}
    We use the previous lemma as follows. Let $\wt{\mf{m}}\subset \wt{\T}^S_{\partial}(\wt{K}^p)$ be an open maximal ideal and, by abuse of notation, denote by $\wt{\mf{m}}$ its preimage in $\wt{\T}^S$. We may localize long exact sequences for completed cohomology with torsion coefficients coming from equation (\ref{equation: ExcisionDistinguishedTriangleCompletedCohomology}) in the next section at $\wt{\mf{m}}$ and take inverse limits. This stays exact since each term satisfies the Mittag--Leffler condition. Each non--zero term of the localized sequence $\varprojlim_n\wt{\HH}^i_{?}(\wt{K}^p, \mO_L/\varpi_L^n)_{\wt{\mf{m}}}$ actually equals $\wt{\HH}^i_{?}(\wt{K}^p, \mO_L)_{\wt{\mf{m}}'}$ for some $\wt{\mf{m}}' \subset \wt{\T}^S_?(\wt{K}^p)$ maximal open ideal, which we will also denote by $\wt{\mf{m}}$. Therefore it makes sense to say that we have a long--exact sequence coming from (\ref{equation: ExcisionDistinguishedTriangleCompletedCohomology}) for completed cohomology with $\mathcal{O}_L$ or $L$--coefficients and localized at $\wt{\mf{m}}$. On each non-zero term we have an action of the appropriate $\wt{\T}^S_?(\wt{K}^p)_{\wt{\mf{m}}}$.
\end{remark}
Let us also record that, by Poincar\'e duality that we shall recall in the next subsection, there is a natural morphism $\wt{i}: \wt{\T}^S(\wt{K}^p)\rightarrow \wt{\T}^S_{\rm{BM}}(\wt{K}^p)$ extending the morphism $\wt{i}: \wt{\T}^S \rightarrow \wt{\T}^S$ given on double cosets by $[\wt{K}_{\ol{\nu}}g\wt{K}_{\ol{\nu}}] \mapsto [\wt{K}_{\ol{\nu}}g^{-1}\wt{K}_{\ol{\nu}}]$. On the level of the polynomials that we are interested in, recall that $\wt{i}(\wt{P}_{\nu}(X))=\wt{P}_{\nu^c}(X)$ (see e.g. \cite[2.2.19]{10AuthorPaper}).

Now we address Hecke algebras for $G$. Let $\T^S:=\otimes'_{\nu \notin S} \mathcal{H}(\GL_n(F_{\nu}),\GL_n(\mathcal{O}_{F_{\nu}})) \otimes_{\mathbf Z} \mathcal{O}_L$, and note that this algebra has generators given by $T_{\nu,i}=[\GL_n(\mathcal{O}_{F_{\nu}})\on{diag}(\varpi_{\nu},\ldots,\varpi_{\nu},1,\ldots,1)\GL_n(\mathcal{O}_{F_{\nu}})]$ where $\varpi_{\nu}$ appears $i$ times and $i=0,\ldots,n$, together with $T_{\nu,n}^{-1}$, for all $\nu \notin S$. We define
\begin{equation}
\label{equation: big Hecke algebra for G}
\T^S(K^p):=\on{cl-im}\left(\T^S \rightarrow \prod_{i \geq 0,\ K_p,\ m} \End(\HH^i(X^G_{K_pK^p},\mathcal{O}_L/\varpi_{L}^m))\right),
\end{equation}
where the codomain is endowed with the profinite topology and $\on{cl-im}$ denotes the closure of the image. As for $\wt{\T}^S(\wt{K}^p)$ we note that $\T^S(K^p)$ is a profinite semilocal algebra with finitely many open maximal ideals, each of which comes from a finite level, and each localization is adically complete and separated.
For each $\nu \notin S$ the polynomial 
$$P_{\nu}(X):=X^n-T_{\nu,1}X^{n-1}+\cdots+ (-1)^{i}q_{\nu}^{i(i-1)/2}T_{\nu,i}X^{n-i}+\cdots+(-1)^nq_{\nu}^{n(n-1)/2}T_{\nu,n}
$$
will play an important role later.

The Hecke algebras for $\wt{G}$ and $G$ are related via the \textit{unnormalized Satake} morphism (see \cite[2.2.3-2.2.4]{NT16}),  which we denote by $\mc{S}: \wt{\T}^S \rightarrow \T^S$. Let $\mf{m}\subset \T^S(K^p)$ be an open maximal ideal. Then there is a Galois representation (cf. \cite[Theorem 1.0.3]{Sch15}) 
\begin{equation}
\label{eq: Scholze mod p Galois representation}
\ol{\rho}_{\mf{m}}:\Gal_{F,S}\rightarrow \GL_n(\T^S(K^p)/\mf{m}) \ .
\end{equation}
The maximal ideal $\mf{m}$ is called \textit{non--Eisenstein} if $\ol{\rho}_{\mf{m}}$ is absolutely irreducible. Moreover, we say $\mf{m}$ is \textit{decomposed generic} (in the sense of Caraiani--Scholze) if there is a prime $\ell$ which splits completely in $F$ and such that for all $\nu \mid \ell$ of $F$, $\ol{\rho}_{\mf{m}}|_{\Gal_{F_{\nu}}}$ is unramified with eigenvalues $\alpha_1,\dots, \alpha_n$ satisfying $\alpha_i/\alpha_j \neq \ell$ for all $i \neq j$. 

For the rest of this article we assume that $\mf{m}$ is non--Eisenstein and decomposed generic. Let $\wt{\mf{m}} \subset \wt{\T}^S$ be the inverse image of $\mf{m}$ under $\mc{S}$. By \cite[Theorem 2.4.4, Theorem 5.4.1]{10AuthorPaper}, $\wt{\mf{m}}$ corresponds to a maximal open ideal of $\wt{\T}^S_{\partial}(\wt{K}^p)$, which we also denote by $\wt{\mf{m}}$, and the Satake morphism descends to a morphism $\mc{S}: \wt{\T}^S_{\partial}(\wt{K}^p)_{\wt{\mf{m}}} \rightarrow \T^S(K^p)_{\mf{m}}$. 

We now argue that under our assumptions on $\mf{m}$, the Satake morphism also descends to a morphism $\mc{S}: \wt{\T}^S_{\rm{BM}}(\wt{K}^p)_{\wt{\mf{m}}} \rightarrow \T^S(K^p)_{\mf{m}}$. Indeed, let $\wt{K}=\wt{K}_p\wt{K}^p$ which decomposes well with respect to $P$ and let $K=\wt{K}\cap G(\mathbf A_{F^+}^{\infty})$. Note  $\HH^i(X_K^G, \mathcal{O}_L/\varpi_{L}^m)=0$ for all $i> d-1$.\footnote{Note $\dim X_K^G=d-1$.} For $i\leq d-1$, we have a $\wt{\T}^S_{\wt{\mf{m}}}$--equivariant embedding $\HH^i_{\partial}(X_{\wt{K}}^{\wt{G}}, \mathcal{O}_L/\varpi_{L}^m)_{\wt{\mf{m}}}\hookrightarrow \HH^{i+1}_c(X_{\wt{K}}^{\wt{G}}, \mathcal{O}_L/\varpi_{L}^m)_{\wt{\mf{m}}}$ since $\HH^j(X_{\wt{K}}^{\wt{G}}, \mathcal{O}_L/\varpi_{L}^m)_{\wt{\mf{m}}}=0$ for $j\leq d-1$ by \cite[Theorem 1.1]{CS24} and \cite[Theorem 1.3]{KoshikawaGeneric}. By \cite[proof of Theorem 2.4.4]{10AuthorPaper}, we have a natural map extending $\mc{S}$:
\begin{equation}
\label{eq: 10AP 2.4.4 on cohomology in degree i}    
\on{im}(\wt{\T}^S_{\wt{\mf{m}}} \rightarrow \End(\HH^i_{\partial}(X_{\wt{K}}^{\wt{G}}, \mO_L/\varpi_{L}^m)_{\wt{\mf{m}}})) \rightarrow \on{im}(\T^S_{\mf{m}} \rightarrow \End(\HH^i(X_{K}^{G}, \mO_L/\varpi_{L}^m)_{\mf{m}})) \ .
\end{equation}
Doing this for all $\wt{K}$ and $m$, and restricting $\wt{\T}^S_{\rm{BM}}(\wt{K}^p)_{\wt{\mf{m}}}$ to its actions on the respective cohomology groups, we get a natural, continuous by construction, map $\mc{S}: \wt{\T}^S_{\rm{BM}}(\wt{K}^p)_{\wt{\mf{m}}}\rightarrow \T^S(K^p)_{\mf{m}}$ extending $\mc{S}:\wt{\T}^S\rightarrow \T^S$.
In the next subsections we will relate the locally analytic vectors of compactly supported completed cohomology localized at $\wt{\mf{m}}$ for $\wt{G}$ to locally analytic vectors of completed cohomology localized at $\mf{m}$ for $G$, and $\mc{S}$ is precisely the relation between the two Hecke actions.

Finally, for later use, we recall from \cite[Proposition--Definition 5.3]{NT16} that for all $\nu \notin S$
$$\mc{S}(\wt{P}_{\nu}(X))=P_{\nu}(X)q_{\nu}^{n(2n-1)}P_{\nu^c}^{\vee}(q_{\nu}^{1-2n}X) \ .
$$

\begin{remark}
    We briefly comment on our choices of Hecke algebras for $\wt{G}$ and $G$. Firstly, we have chosen to use a ``derived'' version of the Hecke algebras for $\wt{G}$ since these are more compatible with (Poincar\'e) dualities between different types of cohomology/homology complexes. On the other hand, one might wonder why we use \eqref{equation: big Hecke algebra for G} for the definition of $\T^S(K^p)$, instead of the more derived definition as in \eqref{equation: big Hecke algebra for tG}. The main reason is that with \eqref{equation: big Hecke algebra for G} it is easier to show that $\mc{S}$ descends to a morphism between Hecke algebras, and it also does not affect the application to the main result, since the action of any choice of the Hecke algebra would factor through \eqref{equation: big Hecke algebra for G} anyway. A very similar choice of Hecke algebra is also employed by \cite{DPS25} and \cite{PasQua25}.
\end{remark}
\subsection{Completed cohomology and locally analytic vectors}
\label{subsection: intro to completed}
We are interested in Hecke eigensystems appearing in completed cohomology groups. We now collect some general preliminaries about completed cohomology. Although the results in this subsection are stated for $\tG$, they all hold for general connected reductive groups over any number field. As before, our coefficients come from a finite extension $L/\Q_p$, with valuation ring $\mathcal{O}_L$ and uniformizer $\varpi_L$. Let us denote by 
\begin{equation}
\label{equation: Definition of completed cohomology}
R\Gamma(\wt{K}^p, \mO_L):=\varprojlim_n \varinjlim_{\wt{K}'_p} R\Gamma(X^{\wt{G}}_{\wt{K}'_p\wt{K}^p}, \mO_L/\varpi_L^n) = \varprojlim_n R\Gamma(\wt{K}^p, \mO_L/\varpi_L^n),
\end{equation}
and let $R\Gamma(\wt{K}^p, L):=R\Gamma(\wt{K}^p, \mO_L) \otimes_{\mO_L} L$. We define complexes for boundary and compactly supported cohomology similarly. We denote the cohomology of these complexes by $\wt{\HH}^i_?(\wt{K}^p, L):=  \HH^i(R\Gamma_?(\wt{K}^p, L))$ (similarly for $\mO_L$ and $\mO_L/\varpi_L^n$ coefficients). We use similar notation for the completed cohomology for $G$ instead of $\wt{G}$. We note that these complexes come equipped with an action of $\wt{\T}^S_?(\wt{K}^p) \times \wt{G}(F^+_p)$ (or $\T^S(K^p)\times \GL_n(F_p)$ for $G$).  We use similar definitions and notations for completed (usual and Borel--Moore) homology.

For coefficients $\Lambda = \mO_L/\varpi_L^n, \mO_L, L$, the complex $R\Gamma(\wt{K}^p, \Lambda)$ is related to complexes $R\Gamma_c(\wt{K}^p, \Lambda)$ and $R\Gamma_{\partial}(\wt{K}^p,\Lambda)$ via the usual $\wt{\T}^S\times \wt{G}(F^+_p)$-equivariant distinguished triangle
\begin{equation}
\label{equation: ExcisionDistinguishedTriangleCompletedCohomology}
 R\Gamma_c(\wt{K}^p, \Lambda) \rightarrow R\Gamma(\wt{K}^p, \Lambda) \rightarrow R\Gamma_\partial(\wt{K}^p, \Lambda) \xrightarrow{[1]}
\end{equation}
As in the previous subsection, let $\wt{i}: \wt{\T}^S \rightarrow \wt{\T}^S$ be given by sending cosets $[\wt{K}^Sg\wt{K}^S]\rightarrow [\wt{K}^Sg^{-1}\wt{K}^S]$. We now recall various (Poincar\'e) duality statements. 
\begin{proposition}[Poincar\'e and homology--cohomology dualities]\label{dualities}
    \begin{enumerate}[label=(\arabic*)]
        \item Let $R$ be a Noetherian commutative ring and $A$ a module for $R[\wt{K}_S]$ which is finite free as an $R$--module and let $B:=\Hom_R(A,R)$. There is a natural homotopy equivalence
        \[
    R\Hom_R(R\Gamma_c(X^{\wt{G}}_{\wt{K}}, \underline{A}),R) \simeq R\Gamma(X^{\wt{G}}_{\wt{K}}, \underline{B})[2d]
        \]
        so that $\wt{\T}^S$-action on the LHS corresponds to the $\wt{\T}^S$-action on the RHS pre--composed with $\wt{i}$.
        \item Let $\wt{\HH}_i:= \wt{\HH}_i(\wt{K}^p,L)$ and $\wt{\HH}_i^{\rm{BM}} := \wt{\HH}_i^{\rm{BM}}(\wt{K}^p,L)$ denote completed homology and completed Borel--Moore homology, respectively. Abbreviate cohomology groups similarly. Then there are natural $\wt{\T}^S_?(\wt{K}^p) \times \wt{G}(F_p^+)$--equivariant isomorphisms (here $\wt{G}(F_p^+)$ acts naturally on the right on homology, and on the left on cohomology)
        $$\Hom^{cts}_L(\wt{\HH}^i, L) \cong \wt{\HH}_i, \ \ \Hom_L^{cts}(\wt{\HH}^i_c, L) \cong \wt{\HH}_i^{\rm{BM}}
        $$
        Furthermore $\wt{\HH}^i$ and $\wt{\HH}^i_c$ are admissible Banach representations of $\wt{G}(F_p^+)$ and $\wt{\HH}_i$ and $\wt{\HH}_i^{\rm{BM}}$ are finitely--generated (naturally right) $L[[\wt{K}_p]]$--modules.
        \item We have a Poincar\'e duality on the level of completed homology given by a spectral sequence
        $$E_2^{i,j}:= \on{Ext}_{L[[\wt{K}_p]]}^i(\wt{\HH}_j, L[[\wt{K}_p]]) \implies \wt{\HH}^{\rm{BM}}_{2d-i-j}
        $$
        which is equivariant for $\wt{\T}^S\times L[[\wt{K}_p]]$--action as follows: the action of $\wt{\T}^S$ on the LHS equals the action on the RHS pre--composed with $\wt{i}$, as usual. The right action of $L[[\wt{K}_p]]$ on the RHS equals the right action of $L[[\wt{K}_p]]$ on the LHS coming from the right action on the Ext--factor $L[[\wt{K}_p]]$ (here $\on{Ext}$ is taken with respect to left actions, and $\wt{\HH}_j$ is viewed as a left module by inverting the right action). 
    \end{enumerate}
\end{proposition}
\begin{proof}
    For the first part see \cite[Corollary 5.3.2 + (5.1.3)]{caraiani2023vanishing}. For the second part, the first statements follow from tensoring up to $L$ the statements of \cite[Theorem 1.1 (3)]{CalEmSurvey}. The admissibility comes from \cite[Theorem 2.2.11]{emerton2006interpolation}, and the statements about homology follow from dualizing. The third part follows from \cite[Theorem 5.1.6, Lemma 5.3.1, Corollary 5.3.2]{caraiani2023vanishing}. Indeed, we have that  $$C_{\mathbf A, \bullet }^{\rm{BM}}(\wt{K}, L[[\wt{K}_p]])[2d] \simeq \Hom_{L[[\wt{K}_p]]}(F_{\bullet} \otimes_{\mathbf Z[\wt{K}]} L[[\wt{K}_p]], L[[\wt{K}_p]])\simeq\Hom_{L[[\wt{K}_p]]}(C_{\mathbf A, \bullet}(\wt{K}, L[[\wt{K}_p]]), L[[\wt{K}_p]]),$$
    where $L[[\wt{K}_p]]$ is viewed as a left/right module naturally over itself and both $C_{\mathbf{A}, \bullet}^{?}(\wt{K}, -)$  and $F_{\bullet}$ are as in \textit{loc.\,cit.} (see also \cite[Proposition 6.3.5]{JNWE24} for compatibility of actions). 
\end{proof}
\begin{remark}
    The Poincar\'e duality at each finite level with torsion coefficients yields an extension of a morphism $\wt{i}: \wt{\T}^S \rightarrow \wt{\T}^S$ to a continuous morphism $\wt{i}: \wt{\T}^S(\wt{K}^p) \rightarrow \wt{\T}^S_{\on{BM}}(\wt{K}^p)$. Furthermore, the Hecke-equivariance in Proposition \ref{dualities} (3) extends to the compatibility of actions of $\wt{\T}^S(\wt{K}^p)$ on the LHS and $\wt{\T}^S_{\on{BM}}(\wt{K}^p)$ on the RHS given by the map $\wt{i}: \wt{\T}^S(\wt{K}^p) \rightarrow \wt{\T}^S_{\on{BM}}(\wt{K}^p)$.
\end{remark}
The key for us is that we can use Poincar\'{e} duality on the dual of locally analytic completed cohomology, just by tensoring with distribution algebras. First, for $V$ an $L$-Banach representation of $\tG(F_p^+)$, we let $V^{\tG-\la}$ denote its locally analytic vectors. Similarly, for $G(F_p^+)$-representations $W$ we will use the notation $W^{G-\la}$, to distinguish between the two. When there is no risk of confusion, we simply use the notation $(-)^{\on{la}}$.

\begin{proposition}\label{prop:poincare for locally analytic}
    \begin{enumerate}[label=(\arabic*)]
        \item $(\wt{\HH}_{\wt{\mf{m}}}^{i,\wt{G}-\on{la}})' \simeq (\wt{\HH}_i)_{\wt{\mf{m}}}\otimes_{L[[\wt{K}_p]]}D(\wt{K}_p)$ for any $\wt{\mf{m}}\subset \wt{\T}^S(\wt{K}^p)$ an open maximal ideal. Similar statements hold for $\wt{\HH}^i_c$, $\wt{\HH}_i^{\rm{BM}}$ and any maximal ideal $\wt{\mf{m}}\subset \wt{\T}^S_{\on{BM}}(\wt{K}^p)$.
        \item  There is also a Poincar\'{e} duality spectral sequence
        \[
        E_2^{i, j} = \Ext_{D(\wt{K}_p)}^i(\wt{\HH}_j\otimes_{L[[\wt{K}_p]]}D(\wt{K}_p), D(\wt{K}_p)) \implies \wt{\HH}_{2d-i-j}^{\rm{BM}} \otimes_{L[[\wt{K}_p]]} D(\wt{K}_p),
        \]
       with similar equivariance properties as Proposition \ref{dualities} (3).
    \end{enumerate}
\end{proposition}
\begin{proof}
    The first part is \cite[Theorem 7.1]{ST03} (here we view homologies as right $D(\wt{K}_p)$--modules and in \textit{loc.\,cit.} they view them as left modules using the anti--involution $D(\wt{K}_p) \rightarrow D(\wt{K}_p)$ induced by $g \mapsto g^{-1}$: note this extends the involution $\iota: Z(\wt{\mf{g}})\rightarrow Z(\wt{\mf{g}}) $, with natural inclusions $Z(\wt{\mf{g}})\subset U(\wt{\mf{g}}) \subset D(\wt{K}_p)$ as explained in \cite[p. 9-12]{schneider2002locally}).

    The second part follows from Proposition \ref{dualities} together with the facts that $L[[\wt{K}_p]] \hookrightarrow D(\wt{K}_p)$ is faithfully flat (see \cite[Theorem 5.2]{ST03}), $L[[\wt{K}_p]]$ is Noetherian, and $\wt{\HH}_j$ and $L[[\wt{K}_p]]$ are finitely presented $L[[\wt{K}_p]]$--modules, and thus the natural right $D(\wt{K}_p)$--module map $$\Ext_{L[[\wt{K}_p]]}^i(\wt{\HH}_j, L[[\wt{K}_p]])\otimes_{L[[\wt{K}_p]]}D(\wt{K}_p) \to \Ext_{D(\wt{K}_p)}^i(\wt{\HH}_j\otimes_{L[[\wt{K}_p]]}D(\wt{K}_p), D(\wt{K}_p))$$
    is an isomorphism.
\end{proof}
In the next subsection, we relate the locally analytic vectors of compactly supported completed cohomology for $\wt{G}$ (localized suitably) to the locally analytic vectors of completed cohomology for $G$ (localized suitably) and keep track of how the corresponding infinitesimal actions of $Z(\wt{\mf{g}})$ and $Z(\mf{g})$ relate.
\subsection{Completed cohomology: from unitary group to $\GL_n$}
\label{section: facts about completed cohomology of unitary groups}
We now recall some facts we need that are specific to the unitary group $\tG$. Let $\mf{m}\subset \T^S(K^p)$ be a non--Eisenstein, decomposed generic maximal ideal and let $\wt{\mf{m}}:=\mc{S}^{\ast}(\mf{m})\subset \wt{\T}^S$ be the corresponding maximal Eisenstein ideal. We localize (\ref{equation: ExcisionDistinguishedTriangleCompletedCohomology}) with respect to $\wt{\mf{m}}$ and record the following facts from the literature. Let $L/\mathbf Q_p$ be a finite extension as usual.
\begin{theorem}\label{vanishing}
    \begin{enumerate}[label=(\arabic*)]
        \item If $\wt{\HH}^i(\wt{K}^p,L)_{\wt{\mf{m}}}\neq 0$ then $i=d$; the same is true after localizing at $\wt{i}(\wt{\mf{m}})$.
        \item Let $H\subset \tG(F_p^+)$ be a compact open pro-$p$ subgroup. Then, $R\Gamma_c(\wt{K}^p, L)_{\wt{\mf{m}}}$ is quasi--isomorphic to a complex of injective admissible $H$-representations concentrated in degrees $[0,d]$ -- furthermore, each term is of the form $\mc{C}^{\rm{cts}}(H,L)^{\oplus m}$ for some $m$; the same is true when localized at $\wt{i}(\wt{\mf{m}})$.
        \item The natural map $\wt{\HH}^d_c(\wt{K}^p, L) \rightarrow \wt{\HH}^d(\wt{K}^p, L)$ is surjective.
        \item Let $\Lambda=\mathcal{O}_L/\varpi_L^n$ for some $n\geq 1$. Then $\on{sm}-\Ind_{P(F^+_p)}^{\wt{G}(F^+_p)} R\Gamma(K^p, \Lambda)_{\mf{m}}$ is a $\wt{\T}^S\times \wt{G}(F^+_p)$-equivariant direct summand of $R\Gamma_\partial(\wt{K}^p, \Lambda)_{\wt{\mf{m}}}$. Here, $\wt{\T}^S$ acts on $\on{sm}-\Ind_{P(F^+_p)}^{\wt{G}(F^+_p)} R\Gamma(K^p, \Lambda)_{\mf{m}}$ by composing $\mc{S}$ with the natural action of $\T^S$ and $P(F^+_p)$ acts on $R\Gamma(K^p, \Lambda)$ by inflating the action of $G(F_p^+)$.
    \end{enumerate}
\end{theorem}
\begin{proof}
\begin{enumerate}[label=(\arabic*)]
    \item Firstly, by \cite[Theorem 2.6.2, Lemma 4.6.2]{CS24} we have that $\wt{\HH}^i=0$ for $i>d$. From \cite{CS24} and \cite{KoshikawaGeneric} we have the vanishing of $\wt{\HH}^i_{\wt{\mf{m}}}$ for $i<d$ and the vanishing of $\wt{\HH}^i_c(\wt{K}_p\wt{K}^p,\mO_L/\varpi_L)_{\wt{\mf{m}}}$ for $i>d$ for all $\wt{K}_p$ small enough, which by Poincar\'e duality gives us the vanishing of $\wt{\HH}^i(\wt{K}_p\wt{K}^p,\mO_L/\varpi_L)_{\wt{i}(\wt{\mf{m}})}$ for all $i<d$ and therefore vanishing of $\wt{\HH}^i(\wt{K}^p)_{\wt{i}(\wt{\mf{m}})}$ for $i \neq d$. This proves the first claim.
    \item The second claim follows from computation as in \cite[Proposition 5.4]{mcdonald2025eigenvarieties} and the vanishing results (for both $\wt{\mf{m}}$ and $\wt{i}(\wt{\mf{m}})$) as in the previous part.
    \item The third claim follows from \cite[Theorem 1.4, Remark 4.6, Corollary 4.7]{HJ23}  -- here we note that their results also imply the statement for our conventions of locally symmetric spaces (defined as in \cite[Section 3.1]{NT16}).
    \item The fourth claim is the content of \cite[Theorem 5.4.1]{10AuthorPaper}.
\end{enumerate}
\end{proof}
We adapt the previous results for our purposes in the next corollary.
\begin{corollary} 
\label{corollary: MiddleDegreeOnCompletedLevel}
    \begin{enumerate}[label=(\arabic*)]
        \item For any $\wt{K}_p\subset \wt{G}(F_p^+)$ a compact open pro-$p$ subgroup, there exists an integer $m\geq 0$ such that there is a surjection $\mc{C}^{\rm{cts}}(\wt{K}_p,L)^{\oplus m} \twoheadrightarrow \wt{\HH}^d(\wt{K}^p, L)_{?}$ of admissible $\wt{K}_p$-representations, where $?$ denotes $\wt{\mf{m}}$ or $\wt{i}(\wt{\mf{m}})$.
        \item  For all $i\geq 0$, the continuous induction $\on{cts}-\Ind_{P(F_p^+)}^{\wt{G}(F_p^+)} \wt{\HH}^i(K^p, L)_{\mf{m}}$ $\wt{\T}^S \times \wt{G}(F_p^+)$--equivariantly embeds into $\wt{\HH}^i_{\partial}(\wt{K}^p, L)_{\wt{\mf{m}}}$. In particular, we have a continuous $\wt{\T}^S\times \wt{G}(F_p^+)$--equivariant embedding $\on{la}-\Ind_{P(F_p^+)}^{\wt{G}(F_p^+)} \wt{\HH}^i(K^p,L)^{G-\on{la}}_{\mf{m}} \hookrightarrow \wt{\HH}^{i}_{\partial}(\wt{K}^p,L)_{\wt{\mf{m}}}^{\wt{G}-\on{la}}$.
    \end{enumerate}
\end{corollary}
\begin{proof}
    The first part follows from the previous Theorem \ref{vanishing} (2) + (3).
    The second part follows from Theorem \ref{vanishing} (4), more precisely the proof of \cite[Theorem 5.4.1]{10AuthorPaper}, and taking a limit over $n$ (more specifically, see \cite[Proposition 7.1]{FuDerived}). We note that continuity of the embedding in the second part is automatic, since both spaces are admissible $\tG(F_p^+)$-representations (see e.g. \cite[Proposition 3.1]{schneider2002banach}). The last part follows by taking $\wt{G}(F_p^+)$--locally analytic vectors.
    \end{proof}
Next, we move from the boundary to compactly supported cohomology.
\begin{theorem}
\label{theorem: G-la in G tilde la}
    There is a (natural) $\wt{\T}^S_{\wt{\mf{m}}}\times U(\mf{g})$--equivariant embedding $\wt{\HH}^i(K^p, L)_{\mf{m}}^{G-\on{la}} \hookrightarrow \wt{\HH}^{i+1}_c(\wt{K}^p,L)_{\wt{\mf{m}}}^{\wt{G}-\on{la}}$ for all $i=0,\ldots,d-1$, where the $\wt{\T}^S_{\wt{\mf{m}}}$-action on the source factors through $\mc{S}: \wt{\T}^S_{\wt{\mf{m}}} \to \T^S_{\mf{m}}$. Furthermore, the Lie algebra $\ol{\mf{n}}:=\on{Lie}(\overline{U}(F_p^+))$ annihilates $\wt{\HH}^i(K^p, L)_{\mf{m}}^{G-\on{la}}$ inside of $\wt{\HH}^{i+1}_c(\wt{K}^p,L)_{\wt{\mf{m}}}^{\wt{G}-\on{la}}$.
\end{theorem}
\begin{proof}
     Localizing the distinguished triangle \eqref{equation: ExcisionDistinguishedTriangleCompletedCohomology} at $\wt{\mf{m}}$ and passing to the long exact sequence, Theorem \ref{vanishing} (1) gives $\wt{\HH}_{\partial}^{i}(\wt{K}^p, L)_{\wt{\mf{m}}} \hookrightarrow \wt{\HH}_c^{i+1}(\wt{K}^p, L)_{\wt{\mf{m}}}$ for $i \leq d-1$, so it suffices to produce an embedding $\wt{\HH}^{i}(K^p, L)_{\mf{m}}^{G-\la} \hookrightarrow \wt{\HH}_{\partial}^i(\wt{K}^p, L)_{\wt{\mf{m}}}^{\wt{G}-\la}$ for all $i \le d-1$. Using Corollary \ref{corollary: MiddleDegreeOnCompletedLevel}, we just need to produce an embedding $\wt{\HH}^{i}(K^p, L)_{\mf{m}}^{G-\la} \hookrightarrow$  $\la-\Ind_{P(F_p^+)}^{\wt{G}(F_p^+)} (\wt{\HH}^i(K^p, L)_{\mf{m}}^{G-\la})$. Set $V = \wt{\HH}^{i}(K^p, L)_{\mf{m}}^{G-\la}$ and consider the subspace \[
    (\la-\Ind_{P(F_p^+)}^{\wt{G}(F_p^+)}V)(P(F_p^+)\ol{U}_0)\hookrightarrow \la-\Ind_{P(F_p^+)}^{\wt{G}(F_p^+)}V
    \] of locally analytic functions whose nonzero germs are contained in $P(F_p^+)\ol{U}_0$. By \cite[Lemma 2.3.3]{Emerton_Jacquet_II} we have a $\ol{P}_0$-equivariant isomorphism $(\la-\Ind_{P(F_p^+)}^{\wt{G}(F_p^+)}V)(P(F_p^+)\ol{U}_0)\simeq \mc{C}^{\la}(\ol{U}_0,V)$, via the natural map $f \mapsto f|_{\ol{U}_0}$. By passing through the inverse of this isomorphism, the subspace of constant functions $V \simeq (\mc{C}^{\la}(\ol{U}_0,V))^{\ol{U}_0} \hookrightarrow \mc{C}^{\la}(\ol{U}_0,V)$ gives the desired embedding. Since all the maps are $\ol{P}_0$-equivariant, we also have that $\wt{\HH}^i(K^p, L)_{\mf{m}}^{G-\on{la}} = V \hookrightarrow (\la-\Ind_{P(F_p^+)}^{\wt{G}(F_p^+)}V)^{\ol{U}_0}\hookrightarrow (\wt{\HH}_c^{i+1}(\wt{K}^p, L)_{\wt{\mf{m}}}^{\wt{G}-\la})^{\ol{U}_{0}}$, all inclusions being $\wt{\T}^S_{\wt{\mf{m}}} \times U(\mf{g})$--equivariant. The last statement is clear from the construction of the embedding.
\end{proof}
\begin{remark}
    Keeping track of all the compatibilities of Hecke actions on finite levels, as in the end of \cref{subsection: big Hecke}, we note that the $\wt{\T}^S_{\wt{\mf{m}}}$--equivariance in the theorem above upgrades naturally to $\wt{\T}^S_{\on{BM}}(\wt{K}^p)_{\wt{\mf{m}}}$--equivariance, where $\wt{\T}^S_{\on{BM}}(\wt{K}^p)_{\wt{\mf{m}}}$ acts on $\wt{\HH}^i(K^p,L)_{\mf{m}}^{G-\on{la}}$ via $\mc{S}: \wt{\T}^S_{\on{BM}}(\wt{K}^p)_{\wt{\mf{m}}} \rightarrow \T^S(K^p)_{\mf{m}}$.
\end{remark}
In the next section, we show that the infinitesimal action on $\wt{\HH}^*_c(\wt{K}^p,L)_{\wt{\mf{m}}}^{\wt{G}-\on{la}}$ is given by the expected generalized Hodge-Tate-Sen character (that we recall).

\section{Infinitesimal action for $\wt{G}$}
\label{section: Infinitesimal action for unitary group}

In order to use results from \cite{DPS25} and \cite{PasQua25} to construct the appropriate generalized Hodge-Tate-Sen characters, we need to start with the appropriate (continuous) Galois representations or (continuous) Lafforgue pseudocharacters valued in the desired Hecke algebras. This is the content of \cref{subsection: existence of Galois representations} and, later, we will use it in \cref{subsection: inf act on unitary} to prove that $Z(\wt{\mf{g}})$ acts on completed cohomology groups via the appropriate generalized Hodge-Tate-Sen character.
We refer the reader to \cite[Definition 4.1]{PasQua25} for the notion of Lafforgue's pseudocharacters. We note that, by the main theorem of \cite{emerson2018comparison}, this notion is equivalent to the notion of a (continuous) determinant, as defined in \cite[Sections 1.5, 2.30]{Che14}. For the most part we will use the notion of determinants explicitly and implicitly view them as pseudocharacters as well. 

For convenience we briefly recall the definition of a determinant.
Let $A$ be a commutative ring and $M,N$ two $A$--modules. If $\mc{C}_A$ denotes the category of commutative $A$-algebras, then let $\ul{M}: \mc{C}_A \to \mathsf{Set}$ denote the functor $B \mapsto M \otimes_A B$. An $A$--\textit{polynomial law} $P:M\to N$ is a natural transformation $\ul{M} \to \ul{N}$, i.e. a collection of maps $M \otimes_A B \rightarrow N\otimes_A B$, compatible with scalar extensions via any morphism $\varphi: B \to B'$ in $\mc{C}_A$. An $A$--polynomial law $D$ is called \textit{homogeneous of degree} $d\geq0$ if $D(bx)=b^d D(x)$ for all $B \in \mathcal{C}_A$, all $b\in B$ and all $x \in M \otimes_A B$. If $M$ and $N$ are also $A$--algebras, then $D$ is called \textit{multiplicative} if it commutes with multiplication morphisms, i.e. for all $B \in \mathcal{C}_A$, all $x,y\in M\otimes_AB$, $D(xy)=D(x)D(y)$.
\begin{definition}[{\cite[p. 230]{Che14}}]
     Let $A$ be a (topological) commutative ring, $G$ a (topological) group and let $B$ be a (topological) commutative $A$--algebra. A $d$-dimensional \textit{determinant} with values in $B$ is a multiplicative and homogeneous of degree $d$ $A$--polynomial law $D:A[G]\rightarrow B$. For any $g \in G$ the function $D(X-g)\in B[X]$ is called a \textit{characteristic polynomial} of $g$. We say $D$ is \textit{continuous} if the map $G \rightarrow B[X]$, $g \mapsto D(X-g)$ is continuous (equivalently, if all coefficient functions are continuous).
\end{definition}
\subsection{Determinants valued in Hecke algebras}
\label{subsection: existence of Galois representations}
Let $L/\Q_p$ be a finite extension, $\mathcal{O}_L$ its ring of integers with a uniformizer $\varpi_L \in \mathcal{O}_L$.
We now record the following results about continuous determinants with values in Hecke algebras and the characteristic polynomials given by $\wt{P}_{\nu}$ and $P_\nu$. 
\begin{theorem}[Theorem 6.5.1 + proof of Theorem 6.5.3 of  \cite{caraiani2020shimura}]\label{determinant}
There exists a unique continuous $2n$--dimensional determinant $D: \mathcal{O}_L[{\Gal}_{F,S}] \rightarrow \wt{\T}^S(\wt{K}^p)$ such that for any prime $\nu \notin S$,
$$D(X-\on{Frob}_\nu)=\wt{P}_{\nu}(X) \ .
$$
\end{theorem}
\begin{proof}
    Uniqueness is immediate because of density of the elements $(\on{Frob}_\nu)_{\nu \notin S}$ in $\on{Gal}_{F,S}$ (as in \cite[Example 2.31]{Che14}).
    For the existence, arguing as in \cite[Example 2.32]{Che14}, it suffices to construct such continuous determinants into each finite discrete ring $\wt{\T}^S(\wt{K}_p\wt{K}^p, m)$. Using Remark \ref{coh vs hom}, it is enough to show the claim for $$\Im(\wt{K}_p\wt{K}^p,m):=\on{im}(\wt{\T}^S \rightarrow \on{End}_{D(\mathcal{O}_L/\varpi_{L}^m[K_0/\wt{K}_p])}(C^{\bullet}(\wt{K}_p\wt{K}^p, m))) \ .$$ By extending scalars, it suffices to consider the case $L=\mathbf Q_p$. This follows from the proof of \cite[Theorem 6.5.3]{caraiani2020shimura} (see the beginning of p. 1215), which says that the map
    \begin{equation}
     \T^S_{cl} \rightarrow \on{End}_{D(\mathbf{Z}_p/p^m[K_0/\wt{K}_p])}(C^{\bullet}(\wt{K}_p\wt{K}^p, m))
    \end{equation}
    is continuous, where the RHS is equipped with discrete topology. By Theorem 6.5.1 of \textit{loc.\,cit.} (which is a restatement of \cite[Corollary 5.1.11]{Sch15}), we get the claim for $\Im(\wt{K}_p\wt{K}^p,m)$. 
\end{proof}
\begin{remark}
    The reference \cite{caraiani2020shimura} works in the case that $p$ splits completely in $F$, but the arguments for Theorem 6.5.1 and Theorem 6.5.3 do not use this assumption.
\end{remark}
We recall $\wt{\T}^S(\wt{K}^p)_{\wt{\mf{m}}}=\varprojlim_n \wt{\T}^S(\wt{K}^p)_{\wt{\mf{m}}}/{\wt{\mf{m}}}^n$ has the ${\wt{\mf{m}}}$-adic topology and note that the localization $\wt{\T}^S(\wt{K}^p) \rightarrow \wt{\T}^S(\wt{K}^p)_{\wt{\mf{m}}}$ is continuous.
\begin{corollary}\label{loc at max det}
    There exists a unique continuous $2n$--dimensional determinant $\wt{D}: \mathcal{O}_L[\on{Gal}_{F,S}] \rightarrow \wt{\T}^S(\wt{K}^p)_{\wt{\mf{m}}}$ such that for any prime $\nu \notin S$,
$$\wt{D}(X-\on{Frob}_\nu)=\wt{P}_{\nu}(X) \ .
$$
\end{corollary}
\begin{proof}
    This follows from the previous theorem by restricting $\wt{\T}^S(\wt{K}^p)\rightarrow \wt{\T}^S(\wt{K}^p)_{\wt{\mf{m}}}$.
\end{proof}
Note that Theorem \ref{determinant} gives us a similar statement for compactly supported cohomology.
\begin{corollary}\label{comp. supp. det.}
  There exists a unique continuous $2n$--dimensional determinant $D: \mathcal{O}_L[{\Gal}_{F,S}] \rightarrow \wt{\T}_{\rm{BM}}^S(\wt{K}^p)$ such that for any prime $\nu \notin S$,
$$D(X-\on{Frob}_\nu)=\wt{P}_{\nu}(X) \ .
$$  
As a consequence, for any open maximal ideal $\wt{\mf{m}}\subset \wt{\T}^S_{\rm{BM}}(\wt{K}^p)$, there exists a unique continuous $2n$--dimensional determinant $\wt{D}: \mathcal{O}_L[\on{Gal}_{F,S}] \rightarrow \wt{\T}^S_{\rm{BM}}(\wt{K}^p)_{\wt{\mf{m}}}$ such that for any prime $\nu \notin S$,
$$\wt{D}(X-\on{Frob}_\nu)=\wt{P}_{\nu}(X) \ .
$$
\end{corollary}
\begin{proof}
    This follows from the proof of \cite[Theorem 6.5.3]{caraiani2020shimura} like Theorem \ref{determinant} did. However, let us sketch another argument as it is a useful way of constructing ``new'' determinants from the old ones. Base--changing the result of Theorem \ref{determinant} with respect to $\wt{i}: \wt{\T}^S(\wt{K}^p) \rightarrow \wt{\T}^S_{\rm{BM}}(\wt{K}^p)$, we get a determinant with characteristic polynomials for $\on{Frob}_\nu$ given by $\wt{i}(\wt{P}_\nu(X))=\wt{P}_{\nu^c}(X)= q_\nu^{2n(2n-1)}\wt{P}_\nu^{\vee}(q_\nu^{1-2n}X)$ (see \cite[Section 2.2.19]{10AuthorPaper}). Twisting by $\chi_{\on{cyc}}^{2n-1}$ gives us a determinant with characteristic polynomials corresponding to $\wt{P}_\nu^{\vee}(X)$.\footnote{We refer the reader to \cite[Section 3.4]{quast2026deformations} for some details behind twisting/dual of determinants.} Taking the dual determinant (this can be done for pseudocharacters, so for determinants as well) gives us the desired condition.
     The second statement follows just as Corollary \ref{loc at max det}.
\end{proof}
Furthermore, the results of \cite{caraiani2020shimura} also imply the following result.\footnote{We remove the assumption that $p$ splits completely in $F$ which was needed in \textit{loc.\,cit.} by imposing that $\mf{m}$ is decomposed generic. With this assumption, \cite{caraiani2020shimura} implies the statement of Theorem \ref{theorem: Galois representation for GLn} -- we include an argument for the sake of completeness.}
\begin{theorem}[Section 6.7 of \cite{caraiani2020shimura}]
\label{theorem: Galois representation for GLn}
    Let $\mf{m}$ be a non--Eisenstein decomposed generic maximal ideal of $\T^S(K^p)$. Then there exists a unique (up to isomorphism) continuous semisimple $n$--dimensional Galois representation $\rho: \on{Gal}_{F,S} \rightarrow \GL_n(\T^S(K^p)_{\mf{m}})$ such that for any prime $\nu \notin S$ 
    $$\det(X-\rho(\on{Frob}_{\nu}))= P_{\nu}(X) \ .
    $$
\end{theorem}
\begin{proof}
Since $\mf{m}$ is non--Eisenstein, by \cite[Theorem 2.22(i)]{Che14} it suffices to construct a continuous $n$--dimensional determinant $D: \mathcal{O}_L[\Gal_{F,S}]\rightarrow \T^S(K^p)_{\mf{m}}$ such that for all $\nu \notin S$
$$D(X-\on{Frob}_{\nu})= P_{\nu}(X) \ .
$$
It is again sufficient to construct a continuous determinant into each finite ring $$\T^S(K_pK^p, i,m)_{\mf{m}}:=\on{im}(\T^S_{\mf{m}} \rightarrow \End(\HH^i(X_{K_pK^p}^G, \mathcal{O}_L/\varpi_L^m)_{\mf{m}}))$$
for $K_p$ sufficiently small, and take limits (recall from \eqref{equation: big Hecke algebra for G} our definition of $\T^S(K^p)$). Such determinants are constructed in \cite[Section 6.7, Proposition 6.7.1]{caraiani2020shimura}, with an important wrinkle. In \textit{loc.\,cit.} they assume that $p$ splits completely in $F$ in order to appeal to their Theorem 6.6.6. In our setting, the role of Theorem 6.6.6 is played by the argument surrounding equation \eqref{eq: 10AP 2.4.4 on cohomology in degree i} in Section 2.1, which upon localizing at a non--Eisenstein \textit{decomposed generic} maximal ideal (and using \cite{CS24} and \cite{KoshikawaGeneric}) yields a natural morphism commuting with $\mc{S}: \wt{\T}^S_{\wt{\mf{m}}}\rightarrow \T^S_{\mf{m}}$:\footnote{We remind the reader that we always assume $K_pK^p=\wt{K}_p\wt{K}^p \cap G(\A_{F^+}^\infty)$, i.e. we consider only levels $\wt{K}$ which decompose well with respect to $P$.}
$$\wt{\T}^S_{\on{BM}}(\wt{K}_p\wt{K}^p, i,m)_{\wt{\mf{m}}}:=\on{im}(\wt{\T}^S_{\wt{\mf{m}}} \rightarrow \End(\HH^i_c(X_{\wt{K}_p\wt{K}^p}^{\wt{G}}, \mathcal{O}_L/\varpi_L^m)_{\wt{\mf{m}}})) \rightarrow \T^S(K_pK^p,i,m)_{\mf{m}} \ .
$$
Together with Corollary \ref{comp. supp. det.}, this is enough to replace Theorem 6.6.6 of \textit{loc.\,cit.} and proceed as they do in Section 6.7 to prove the result.
\end{proof}
We have the following consequence of the previous results.
\begin{corollary}
\label{corollary: big Hecke algebras are Noetherian}
    Both $\wt{\T}^S_{?}(\wt{K}^p)_{\wt{\mf{m}}}$ for $?=\emptyset, \on{BM}$ and $\T^S(K^p)_{\mf{m}}$ are complete local Noetherian rings.
\end{corollary}
\begin{proof}
    This follows using the (complete local Noetherian) pseudodeformation rings constructed in \cite[Proposition 3.3, Example 3.6, Proposition 3.7]{Che14}. Indeed, we get a continuous map from the appropriate pseudodeformation rings to the Hecke algebras. The images are dense since the respective subalgebras generated by coefficients of characteristic polynomials are dense by the definition of the Hecke algebras. Therefore, the surjectivity of the maps from pseudodeformation rings follows from compactness of the pseudodeformation rings. Hence all the Hecke algebras in question are quotients of Noetherian rings and so are themselves Noetherian. 
\end{proof}

The previous corollary allows us to consider rigid versions of the localized Hecke algebras for $\wt{G}$ and $G$ as defined in \cite[Section 7.1]{de1995crystalline} and used in \cite{DPS25} and \cite{PasQua25}. Namely, for any complete Noetherian local $\mO_L$-algebra $R$ with finite residue field, we let $R^{\rm{rig}} := \mO((\on{Spf} R)^{\on{rig}})$; note that there is a natural map $R[1/p] \to R^{\on{rig}}$ which is dense. 

The continuous action of $\wt{\T}^S(\wt{K}^p)_{\wt{\mf{m}}}$ on $\wt{\HH}^i(\wt{K}^p, L)_{\wt{\mf{m}}}^{\tG-\la}$ extends to a continuous action of $\wt{\T}^S(\wt{K}^p)_{\wt{\mf{m}}}^{\rm{rig}}$ (see for example \cite[3.1 + Proposition 3.8]{BHSAnnalen}, applied to the $p$-torsion-free quotient of $\wt{\T}^S(\wt{K}^p)_{\wt{\mf{m}}}$). Similarly, the continuous action of $\T^S(K^p)_{\mf{m}}$ on $\wt{\HH}^i(K^p, L)_{\mf{m}}^{G-\la}$ extends to a continuous action of $\T^S(K^p)_{\mf{m}}^{\rm{rig}}$. The morphisms $\wt{i}$ and $\mc{S}$ between Hecke algebras also extend to continuous morphisms between rigid--analytic versions of these rings. Furthermore, the Hecke equivariance of the Poincar\'e duality in Proposition \ref{prop:poincare for locally analytic} extends to the Hecke equivariance for rigid--analytic Hecke algebras via $\wt{i}$; similarly, the embedding given by Theorem \ref{theorem: G-la in G tilde la} is equivariant for the actions of the rigid--analytic Hecke algebras via $\mc{S}$. These last two claims follow from the density of $R[1/p] \to R^{\on{rig}}$ mentioned above.

We are highlighting all this because generalized Hodge-Tate-Sen characters constructed by \cite{DPS25} and \cite{PasQua25} are valued in rigid--analytic versions of Hecke algebras. Before proving that the infinitesimal action on $\wt{\HH}^i_{(c)}(\wt{K}^p, L)^{\wt{G}-\on{la}}_{\wt{\mf{m}}}$ is given by an appropriate character $Z(\wt{\mf{g}}) \rightarrow \wt{\T}^S_{(\on{BM})}(\wt{K}^p)_{\wt{\mf{m}}}^{\on{rig}}$, let us record the following consequences of Section \ref{section: Setup}.
\begin{corollary}\label{cor: transfer of inf act rigid}
    \begin{enumerate}[label=(\arabic*)]
        \item Suppose that, for all $j$, the action of $Z(\wt{\mf{g}})$ on $(\wt{\HH}^{j}_{\wt{i}(\wt{\mf{m}})})^{\wt{G}-\on{la}}$ is given by some character $\wt{\zeta}: Z(\wt{\mf{g}}) \rightarrow \wt{\T}^S(\wt{K}^p)_{\wt{i}(\wt{\mf{m}})}^{\on{rig}}$. Then, for all $j$, the action of $Z(\wt{\mf{g}})$ on $(\wt{\HH}^j_c)_{\wt{\mf{m}}}^{\wt{G}-\on{la}}$ is given by $\wt{i} \circ \wt{\zeta}\circ \iota:Z(\wt{\mf{g}}) \rightarrow \wt{\T}^S_{\on{BM}}(\wt{K}^p)_{\wt{\mf{m}}}^{\on{rig}}$, where $\iota:Z(\wt{\mf{g}})  \rightarrow Z(\wt{\mf{g}}) $ is the map induced by the anti--involution $\wt{\mf{g}}\rightarrow\wt{\mf{g}}$ mapping $x \mapsto-x$.
        \item Suppose that $Z(\wt{\mf{g}})$ acts on $\wt{\HH}^*_c(\wt{K}^p,L)_{\wt{\mf{m}}}^{\wt{G}-\on{la}}$ by some character $\wt{\zeta}: Z(\wt{\mf{g}}) \rightarrow \wt{\T}_{\on{BM}}^{S}(\wt{K}^p)_{\wt{\mf{m}}}^{\on{rig}}$. Let $\HC^{\overline{P}}: Z(\wt{\mf{g}}) \hookrightarrow Z(\mf{g})$ denote the Harish-Chandra morphism with respect to $\ol{\mf{p}} = \on{Lie} \ol{P}$. Then $\HC^{\overline{P}}(Z(\wt{\mf{g}}))$ acts on $\wt{\HH}^*(K^p,L)_{\mf{m}}^{G-\on{la}}$ by $\mc{S} \circ \wt{\zeta}\circ (\HC^{\overline{P}})^{-1}: \HC^{\overline{P}}(Z(\wt{\mf{g}})) \rightarrow \T^S(K^p)_{\mf{m}}^{\on{rig}}.$ 
    \end{enumerate}
\end{corollary}
\begin{proof}
    \begin{enumerate}[label=(\arabic*)]
        \item This follows from Proposition \ref{dualities} (2) and Proposition \ref{prop:poincare for locally analytic} (2). Indeed, we have that the right action of $Z(\wt{\mf{g}}) $ on $\wt{\HH}_{j,\wt{i}(\wt{\mf{m}})}\otimes_{L[[\wt{K}_p]]}D(\wt{K}_p)$ is given by $\wt{\zeta}$ and using part (2) of Proposition \ref{prop:poincare for locally analytic}, we get the right action of $Z(\wt{\mf{g}}) $ on $\wt{\HH}^{\rm{BM}}_{j,\wt{\mf{m}}}\otimes_{L[[\wt{K}_p]]}D(\wt{K}_p)$ is given by $\wt{i}\circ \wt{\zeta}\circ \iota$. Therefore, the left action on $\wt{\HH}^{j, \wt{G}-\on{la}}_{c,\wt{\mf{m}}}$ is as claimed.
        \item  This follows from \cref{theorem: G-la in G tilde la} and the definition of $\HC^{\overline{P}}$ (see also \cite[\textsection 1.3]{Emerton_Jacquet_I}): it is defined using a decomposition $U(\wt{\mf{g}})= \mf{n}U(\ol{\mf{p}}) \oplus U(\mf{g})\oplus U(\wt{\mf{g}})\ol{\mf{n}}$ where $\mf{n}, \ol{\mf{p}}, \ol{\mf{n}}$ are Lie algebras of $U(F_p^+), \ \overline{P}(F_p^+)$ and $\overline{U}(F_p^+)$, respectively, and noting that $Z(\wt{\mf{g}}) \subset U(\mf{g}) \oplus U(\wt{\mf{g}})\ol{\mf{n}}$, $\HC^{\overline{P}}$ denotes the projection to the $U(\mf{g})$-factor.
    \end{enumerate}
\end{proof}
In the next subsection, we show that the infinitesimal action on $\wt{\HH}^*_{c}(\wt{K}^p, L)^{\wt{G}-\on{la}}_{\wt{\mf{m}}}$ is given by an appropriate generalized Hodge-Tate-Sen character $Z(\wt{\mf{g}}) \rightarrow \wt{\T}^S_{\rm{BM}}(\wt{K}^p)_{\wt{\mf{m}}}^{\on{rig}}$ constructed by the recipe of \cite{PasQua25}.\footnote{As mentioned in footnote \ref{footnote: functorial}, this construction is naturally functorial in $L$.}

\subsection{Infinitesimal actions on $\wt{\HH}^d(\wt{K}^p)_{\wt{\mf{m}}}$ and $\wt{\HH}^i_c(\wt{K}^p)_{\wt{\mf{m}}}$}
\label{subsection: inf act on unitary}

We refer the reader to \cite[Sections 2--2.1]{DPS25} for a brief overview of an $L$--group and a $C$--group attached to a reductive group. Since $\wt{G}_{F_{\ol{\nu}}^+} \simeq \GL_{2n}/F_{\ol{\nu}}^+$ and $G_{F_{\ol{\nu}}^+} \simeq (\GL_{n} \times \GL_{n})/ F_{\ol{\nu}}^+$ for any $\ol{\nu} \mid p$ (note $F_{\ol{\nu}}^+ \simeq F_{\nu} \simeq F_{\nu^c}$ where $\nu \mid \ol{\nu}$ in $F$), we shall only consider $C$--groups for $\GL_m$ with $m=n, 2n$ so let us briefly define a $C$--group in this case.
Let $B$ be the standard Borel subgroup of $\GL_m$ and let $\wt{\delta}_m:=(1,t^{-1},\ldots,t^{1-m}): \mathbf G_m \rightarrow \GL_m$.\footnote{This is just one of the possible choices for a cocharacter of $T$: it is chosen so that on the Lie algebras, $(\wt{\delta}_m)_{\ast}(1)=\delta^++(a,\ldots,a)$ for some constant $a$ where $\delta^+=((m-1)/2,\ldots,(1-m)/2)$ is the half sum of positive roots viewed naturally as an element in $\mf{t}=\on{Lie}(T)$. See \cite[2.1]{DPS25}.} The $C$--group ${}^C\GL_m$ is defined to be $\GL_m \rtimes \mathbf G_m$ with multiplication given by $(g,t) \cdot (g', t')=(g \wt{\delta}_m(t)g'\wt{\delta}_m(t)^{-1}, tt')$. We note that this definition does not depend on $\wt{\delta}_m$ (any other choice would differ by a cocharacter valued in the center). Furthermore, we have an isomorphism of algebraic groups (for any choice of $\wt{\delta}_m$)
$$\on{tw}_{\wt{\delta}_{m}}: \ {}^C\GL_{m} \rightarrow \GL_{m} \times \mathbf G_m, \ \ \ (g,t) \mapsto (g\wt{\delta}_{m}(t),t)  \ .
$$
One can play a similar game for the opposite Borel subgroup $\overline{B}$ in which case, for example, we choose $\wt{\delta}_m^-=(t^{1-m},\ldots,t^{-1},1): \mathbf G_m \rightarrow \GL_m$. We shall denote the $C$--group we get for $\overline{B}$ by ${}^C\GL_m^-$. Both constructions will play a part for us.

Now we move onto constructing generalized Hodge-Tate-Sen characters.
Let $\wt{\Theta}$ be a $\GL_{2n}$--pseudocharacter of $\on{Gal}_{F,S}$ attached to $\wt{D}$ valued in $\wt{\T}^S(\wt{K}^p)_{\wt{\mf{m}}}^{\on{rig}}$ given by Corollary \ref{loc at max det} (by \cite{emerson2018comparison} we will not distinguish between pseudocharacters and determinants). As in \cite[Section 8]{PasQua25} we define a ${}^C\GL_{2n}$--pseudocharacter of $\on{Gal}_{F,S}$ valued in $\wt{\T}^S(\wt{K}^p)_{\wt{\mf{m}}}^{\on{rig}}$ $$\wt{\Theta}^C:= \on{tw}_{\wt{\delta}_{2n}}^{-1} \circ ( \wt{\Theta} \boxtimes \chi_{\on{cyc}}).$$
For each finite place $\ol{\nu}\mid p$ in $F^+$ we choose a place $\nu \mid \ol{\nu}$ of $F$ above it. Let $\wt{S}$ denote the set of all chosen places of $F$, so that $S_p(F) = \wt{S} \sqcup \wt{S}^c$. By specializing at each place in $\wt{S}$, the construction from \cite{PasQua25} yields a character
$$\wt{\zeta}^C: Z(\wt{\mf{g}}):= Z(\on{Res}_{F^+/\mathbf Q}\wt{\mf{g}})_{\mathbf Q_p} \cong \otimes_{\nu \in \wt{S}} Z(\on{Res}_{F_{\nu}/\mathbf Q_p}\mf{gl}_{2n})  \rightarrow \wt{\T}^S(\wt{K}^p)_{\wt{\mf{m}}}^{\on{rig}} \ .
$$
\begin{theorem}\label{unitary inf action}
   Let $\wt{\mf{m}}=\mc{S}^*(\mf{m})$ be the pullback under the Satake map of a decomposed generic maximal ideal of $\T^S(K^p)$. Then the action of $Z(\wt{\mf{g}})$ on $\wt{\HH}^d(\wt{K}^p, \mathbf Q_p)_{\wt{\mf{m}}}^{\on{la}}$ is given by $\wt{\zeta}^C$. The same statement applies to $\wt{\HH}^d(\wt{K}^p, \mathbf Q_p)_{\wt{i}(\wt{\mf{m}})}^{\on{la}}$.
\end{theorem}
\begin{proof}
    Let $L/\mathbf Q_p$ be a finite extension containing $[F:\mathbf Q]$ embeddings $F \hookrightarrow L$; therefore $\Res_{F^+/\mathbf Q}\wt{G}$ splits over $L$. It is sufficient to prove the theorem after changing the coefficients to $L$.
    
    The proof follows the same strategy as \cite[Theorem 9.2]{DPS25} and \cite[Theorem 8.2]{PasQua25}. We first note that if either of these spaces are zero there is nothing to prove. Otherwise, we can use Remark \ref{remark: loc at max} to appropriately view $\wt{\mf{m}}$ and $\wt{i}(\wt{\mf{m}})$ as (open) maximal ideals of $\wt{\T}^S(\wt{K}^p)$. We prove the statement for $\wt{\mf{m}}$ as the proof for $\wt{i}(\wt{\mf{m}})$ is the same. As mentioned after Corollary \ref{corollary: big Hecke algebras are Noetherian}, the action of $\wt{\T}^S(\wt{K}^p)_{\wt{\mf{m}}}$ on $\wt{\HH}^d(\wt{K}^p, L)_{\wt{\mf{m}}}$ induces a continuous action of $\wt{\T}^S(\wt{K}^p)_{\wt{\mf{m}}}^{\on{rig}}$ on $\wt{\HH}^d(\wt{K}^p, L)_{\wt{\mf{m}}}^{\on{la}}$.

    We want to show for each $z \in Z(\wt{\mf{g}})$ that $z-\wt{\zeta}^C(z)$ annihilates $\wt{\HH}^d(\wt{K}^p, L)_{\wt{\mf{m}}}^{\on{la}}$. It suffices to prove the vanishing of this difference on a dense subset.  Let $\wt{K}_p\subset \wt{G}(F^+_p)=(\Res_{F^+/\mathbf Q} \wt{G})(\mathbf Q_p)$ be a sufficiently small compact open subgroup. Let $\tau$ be smooth irreducible (finite dimensional) representation of $\wt{K}_p$ on an $L$--vector space which is a supercuspidal type (such a $\tau$ exists by \cite[Prop. 3.19]{emerton2020density}) and $\on{Irr}_{\wt{G}}(L)$ the set of isomorphism classes of irreducible algebraic representations of $(\on{Res}_{F^+/\mathbf Q} \wt{G})_{L}$ (which is a split group).
    By Corollary \ref{corollary: MiddleDegreeOnCompletedLevel} and since passing to locally analytic vectors is exact, we have a continuous surjection $\mc{C}^{\on{la}}(\wt{K}_p, L)^{\oplus k} \twoheadrightarrow \wt{\HH}^d(\wt{K}^p, L)_{\wt{\mf{m}}}^{\on{la}}$. The density statement from \cite[Corollary 7.7]{DPS25} for $\mc{C}^{\on{la}}(\wt{K}_p, L)$ implies that the natural map 
    \begin{equation}\label{eq: density}
    \bigoplus_{V \in \on{Irr}_{\wt{G}}(L)} \Hom_{\wt{K}_p}(V(\tau), \wt{\HH}^d(\wt{K}^p, L)_{\wt{\mf{m}}}^{\on{la}}) \otimes_{L} V(\tau) \rightarrow \wt{\HH}^d(\wt{K}^p, L)_{\wt{\mf{m}}}^{\on{la}}
    \end{equation}
    has dense image. Here $V(\tau):= V \otimes_{L} \tau$ endowed with the diagonal action of $\wt{K}_p$. The morphism (\ref{eq: density}) is equivariant for the $Z( \wt{\mf{g}}) \times \wt{\T}^S(\wt{K}^p)_{\wt{\mf{m}}}^{\on{rig}} $--action and we also note that $V(\tau)$ is a locally analytic $\wt{K}_p$--representation so $$\Hom_{\wt{K}_p}(V(\tau), \wt{\HH}^d(\wt{K}^p, L)_{\wt{\mf{m}}}^{\on{la}})=\Hom_{\wt{K}_p}(V(\tau), \wt{\HH}^d(\wt{K}^p, L)_{\wt{\mf{m}}}) \ .$$ Note that the action of $z \in Z( \wt{\mf{g}})$ on the LHS in (\ref{eq: density}) is via the $V(\tau)$ factor on the right of the tensor product. Since $V$ is an irreducible representation of $(\Res_{F^+/\mathbf Q}\wt{G})_L$, it is of some highest weight $\lambda \in (\mathbf Z^{2n}_+)^{\Hom(F^+, \overline{\mathbf{Q}}_p)}$ with respect to the standard Borel subgroup $(\Res_{F^+/\mathbf Q}B)_L$. Since $\tau$ is smooth, the action of $Z(\wt{\mf{g}})$ on the LHS factor of (\ref{eq: density}) is via the infinitesimal character $\chi_{\lambda}: Z(\wt{\mf{g}}) \rightarrow L$ given by composing the unnormalized Harish-Chandra map (with respect to $(\Res_{F^+/\mathbf Q}B)_L$) $Z(\wt{\mf{g}})_L \rightarrow U(\mf{t})_{L} \xrightarrow[]{\on{ev}_\lambda} L$. Therefore, we need to show that the action of $\wt{\zeta}^C(Z(\wt{\mf{g}}))$ on $\Hom_{\wt{K}_p}(V(\tau), \wt{\HH}^d(\wt{K}^p, L)_{\wt{\mf{m}}})$ is also via $\chi_{\lambda}$.

    We first show that the action of $ \wt{\T}^S(\wt{K}^p)_{\wt{\mf{m}}}^{\on{rig}}$ on $\Hom_{\wt{K}_p}(V(\tau), \wt{\HH}^d(\wt{K}^p, L)_{\wt{\mf{m}}}) \otimes_{L} \overline{\mathbf Q}_p$ is semisimple; that is, this space splits into eigenspaces on which  $ \wt{\T}^S(\wt{K}^p)_{\wt{\mf{m}}}^{\on{rig}}$ acts by a character valued in $\overline{\mathbf Q}_p$.
   The point is that when localized at $\wt{\mf{m}}$, $\Hom_{\wt{K}_p}(V(\tau), \wt{\HH}^d(\wt{K}^p, L)_{\wt{\mf{m}}})$ is related to classical cohomology of the spaces $X_{\wt{K}_p'\wt{K}^p}^{\tG}$ (and thus to automorphic forms). To see this, define $\HH^i(V^\ast) := \varinjlim_{K_p'} \HH^i(X_{K_p'\wt{K}^p}^{\wt{G}}, V^\ast)$. Now recall that $\wt{\HH}^i(\wt{K}^p,L)_{\wt{\mf{m}}} \neq 0$ iff $i =d$ by Theorem \ref{vanishing} (same holds for $\wt{i}(\wt{\mf{m}})$). Thus, Emerton's spectral sequence \cite[Theorem 2.1.5(ii)]{emerton2006interpolation} relating $\wt{\HH}^*(\wt{K}^p, V^*)_{\wt{\mf{m}}} \cong \wt{\HH}^*(\wt{K}^p, L)_{\wt{\mf{m}}} \otimes_L V^*$ and $\wt{\HH}^*(V^{\ast})_{\wt{\mf{m}}}$ degenerates, implying $\HH^d(V^*)_{\wt{\mf{m}}} = \wt{\HH}^d(\wt{K}^p, V^*)_{\wt{\mf{m}}}^{\on{sm}}$. Likewise, since $\tau$ is smooth we have $$\Hom_{\wt{K}_p}(V(\tau), \wt{\HH}^d(\wt{K}^p, L)_{\wt{\mf{m}}})=\Hom_{\wt{K}_p}(\tau, \wt{\HH}^d(\wt{K}^p, V^*)_{\wt{\mf{m}}})=\Hom_{\wt{K}_p}(\tau, \HH^d(V^*)_{\wt{\mf{m}}})$$ where we have again used that $\wt{\HH}^d(\wt{K}^p, V^*)_{\wt{\mf{m}}} \cong \wt{\HH}^d(\wt{K}^p, L)_{\wt{\mf{m}}} \otimes_L V^*$. 
   
   To show semisimplicity, first note $\HH^d(V^*)_{\wt{\mf{m}}}$ is a direct ($\wt{\T}^S(\wt{K}^p)$--equivariant) summand of $\HH^d(V^*)$ hence $\Hom_{\wt{K}_p}(V(\tau), \wt{\HH}^d(\wt{K}^p, L)_{\wt{\mf{m}}}) \otimes_{L} \overline{\mathbf Q}_p$ is a direct ($\wt{\T}^S(\wt{K}^p)$--equivariant) summand of $\Hom_{\wt{K}_p}(\tau, \HH^d(V^{\ast}_{\overline{\mathbf Q}_p}))$. 
   
   Then it suffices to show the Hecke action on $\Hom_{\wt{K}_p}(\tau, \HH^d(V^{\ast})_{\overline{\mathbf Q}_p})$ is semisimple, which we do by following \cite[Section 9.5]{DPS25}. Just as in \cite[Equations (61) and (62)]{DPS25}, by Franke's theorem (fixing an isomorphism $\iota: \overline{\mathbf Q}_p \simeq \mathbf C$):
  \begin{equation}\label{eq:Franke}
  \on{Hom}_{\wt{K}_p}(\tau, \HH^i(V^\ast)) \otimes_{\iota: L\hookrightarrow \C} \mathbf C \simeq \bigoplus_{\pi} \HH^i(\wt{\mf{g}}_{\R}, \mf{t}_{\R}, \pi_{\infty} \otimes_{\C} V_{\mathbf C}^\ast) \otimes_{\C} \Hom_{\wt{K}_p}(\tau_{\mathbf C}, \pi_p) \otimes_{\C} (\pi^{p, \infty})^{\wt{K}^p},
   \end{equation}
   where the sum is taken over all irreducible cuspidal subrepresentations $\pi$ of $\mathcal{A}_{\on{cusp}}$ counted with multiplicities and $\wt{\mf{g}}_{\R}$ denotes the Lie algebra of $(\on{Res}_{F^+/\mathbf Q} \wt{G})(\mathbf R)$ and $\mf{t}_{\R}$ the Lie algebra of a maximal compact subgroup $\wt{K}_{\infty}$ (note that Hecke algebra acts on the last term in the tensor product and it acts by a scalar!).\footnote{We note here that the reductive group $\on{Res}_{F^+/\mathbf Q} \wt{G}$ has no non--trivial rational character, which is why the whole Lie algebra of $\wt{G}(F^+\otimes_{\Q} \mathbf R)$ appears.} Therefore, we have the semisimplicity that we seek  by taking $i=d$ in the equation (\ref{eq:Franke}). 
    
   To prove that for all $z \in Z(\wt{\mf{g}})$ the action of $\wt{\zeta}^C(z)$ on $\Hom_{\wt{K}_p}(V(\tau), \wt{\HH}^d(\wt{K}^p, L)_{\wt{\mf{m}}})$ is via $\chi_{\lambda}(z)$, it suffices to show this after we base--change to $\overline{\mathbf Q}_p$, and so it suffices to work on each nonzero Hecke eigenspace $(\Hom_{\wt{K}_p}(V(\tau), \wt{\HH}^d(\wt{K}^p, L)_{\wt{\mf{m}}})\otimes_{L} \overline{\mathbf Q}_p)[\mf{m}_y]\neq 0$ for any $L$-algebra homomorphism $y: \wt{\T}^S(\wt{K}^p)_{\wt{\mf{m}}}^{\on{rig}} \rightarrow \overline{\mathbf Q}_p$. Recall that, by \cite{PasQua25}, specializing at $y$ gives rise to $\wt{\zeta}^C_y: Z(\wt{\mf{g}})\rightarrow \overline{\mathbf Q}_p$ with the following property: $\wt{\zeta}^C_y$ equals $\zeta^C_{\rho}$ for any admissible representation $\rho:\on{Gal}_{F,S} \rightarrow {}^C \GL_{2n}(\overline{\mathbf Q}_p)$ with pseudocharacter $\Theta_{\rho}=\wt{\Theta}^C_y$. We now argue as in \cite[Theorem 8.2]{PasQua25} that there is such a $\rho$ which also satisfies $\zeta_{\rho}=\chi_{\lambda}$.

   Using (\ref{eq:Franke}), to any such $y$ we can associate a cuspidal automorphic representation $\pi_y=\otimes_{w}' \pi_{y,w}$ of $\wt{G}(\mathbf A_{F^+})$ such that $\wt{\T}^S(\wt{K}^p)$ acts on $(\pi_y^S)^{\wt{K}^S}$ via $\iota \circ y$. Furthermore, $\pi_y$ is $\iota(V^*)$--cohomological as it appears in (\ref{eq:Franke}). Therefore, by \cite[Theorem 2.3.3]{10AuthorPaper}, there is a Galois representation $\rho_y: \on{Gal}_{F,S}\rightarrow \GL_{2n}(\overline{\mathbf Q}_p)$ such that $D_{\rho_y}= \wt{D}_y$ (since by \textit{loc.\,cit.} the determinants agree on a dense subset).\footnote{We denote by $\wt{D}_y$ the specialization of the determinant given by Corollary \ref{loc at max det} at $y$.} Hence $\Theta_{\rho_y}=\wt{\Theta}_y$ and so $\rho_y^C:= \on{tw}_{\wt{\delta}_{2n}}^{-1}\circ (\rho_y \boxtimes \chi_{\on{cyc}}) : \on{Gal}_{F,S}\rightarrow {}^C\GL_{2n}(\overline{\mathbf Q}_p)$ has $\Theta_{\rho^C_y}= \wt{\Theta}^C_y$.

   Finally, we note that $\zeta_{\rho_y^C}=\chi_{\lambda}$ where $\lambda \in (\Z_+^{2n})^{\Hom(F^+, \overline{\mathbf Q}_p)}$ is the highest weight of $V$, as desired. 
   This follows from \cite[Theorem 2.3.3 (b)]{10AuthorPaper}\footnote{There is a typo in \textit{loc.\,cit.} (compare with \cite[pg. 5]{BLGGT11}) and their $\wt{\lambda}$ should be the highest weight of $\iota^{-1}(\xi \otimes \xi),$ not  $\iota^{-1}(\xi \otimes \xi)^{\vee}$, with $\xi:=V^*$ in our notation. We also note that the fact that there are two copies of $\xi$ in Theorem 2.3.3 of \textit{loc.\,cit.} goes away by choosing $\wt{S}$ in our calculation.}. Indeed, recalling our convention that the Hodge--Tate weight $\on{wt}_{\nu,\sigma}(\chi_{\on{cyc}})=1$ (the opposite of \cite{10AuthorPaper}) for each $\sigma:F \hookrightarrow \overline{\mathbf Q}_p$ and a place $\nu \mid p$ of $F$, we have that the Hodge--Tate weights for $\rho_y$ at a place $\nu\mid \overline{\nu}\mid p$ and a choice $\sigma$, are given by $\{\lambda_{\overline{\sigma},1},\lambda_{\overline{\sigma},2}-1,\ldots,\lambda_{\overline{\sigma},2n}-(2n-1)\}$ where $\overline{\sigma}:F^+\hookrightarrow \overline{\mathbf Q}_p$ is the corresponding embedding of $F^+$. The result \cite[Proposition 5.5]{DPS25} then implies $\wt{\zeta}_{\rho^C_y}$ is precisely $\chi_{\lambda}$.
\end{proof}
\begin{remark}
\label{remark about base change and unconditional assumptions}
    Note that our appeal to \cite[Theorem 2.3.3]{10AuthorPaper} is the most essential place where we need the standing assumption that $F$ contains an imaginary quadratic field. Without this assumption, the result would be conditional on the twisted weighted fundamental lemma; however, this lemma might soon be proved, as mentioned in Remark \ref{remark: discussion on ass.}. 
\end{remark}
\begin{remark}
    We point the reader to \cite[Section 4.7]{DPS25}, in particular equation (34) for the core of the calculation appearing in the end of the previous proof and some calculations that will follow below. Even though not all our generalized Hodge-Tate-Sen characters come from Galois representations, but rather from determinants, this section is helpful to understand where the commutative diagrams below should come from.
\end{remark}
Using the previous results, we will deduce a similar statement for the infinitesimal action on the compactly supported completed cohomology groups.  We keep the assumption that $\wt{\mf{m}}=\mc{S}^*(\mf{m})$ is the pullback under the Satake map of a decomposed generic maximal ideal of $\T^S(K^p)$. We view it as a maximal open ideal of $\wt{\T}^S_{\on{BM}}(\wt{K}^p)$ (see Remark \ref{remark: loc at max}). Let $\wt{\Theta}_{\on{BM}}$ be a $\GL_{2n}$--pseudocharacter of $\on{Gal}_{F,S}$ attached to $\wt{D}$ valued in $\wt{\T}^S_{\on{BM}}(\wt{K}^p)_{\wt{\mf{m}}}^{\on{rig}}$ constructed in Corollary \ref{comp. supp. det.} As before, we define a ${}^C\GL_{2n}$--pseudocharacter of $\on{Gal}_{F,S}$ valued in $\wt{\T}^S_{\on{BM}}(\wt{K}^p)_{\wt{\mf{m}}}^{\on{rig}}$ $$\wt{\Theta}_{\on{BM}}^C:= \on{tw}_{\wt{\delta}_{2n}}^{-1} \circ ( \wt{\Theta}_{\on{BM}} \boxtimes \chi_{\on{cyc}}).$$
As we did for usual cohomology, we denote by $\wt{\zeta}^C_{\on{BM}}: Z(\wt{\mf{g}}) \rightarrow \wt{\T}^S_{\on{BM}}(\wt{K}^p)_{\wt{\mf{m}}}^{\on{rig}}$ the generalized Hodge-Tate-Sen character attached to $\wt{\Theta}^C_{\on{BM}}$.
\begin{theorem}\label{theorem: com. supp. inf. act.}
    For any $i\geq 0$, the action of $Z(\wt{\mf{g}})$ on $\wt{\HH}^i_c(\wt{K}^p, \mathbf Q_p)_{\wt{\mf{m}}}^{\on{la}}$ is given by $\wt{\zeta}^C_{\rm{BM}}$.
\end{theorem}
\begin{proof}
We will deduce this result from Theorem \ref{unitary inf action}, Corollary \ref{cor: transfer of inf act rigid} (1) and the two lemmas below. First, we apply Theorem \ref{unitary inf action} for $\wt{i}(\wt{\mf{m}})$ and note that the hypotheses in Corollary \ref{cor: transfer of inf act rigid}(1) are satisfied since $\wt{\HH}^i_{\wt{i}(\wt{\mf{m}})} \neq 0$ iff $i=d$. We conclude that $Z(\wt{\mf{g}})$ acts on $\wt{\HH}^i_c(\wt{K}^p, \mathbf Q_p)^{\on{la}}_{\wt{\mf{m}}}$ by $\wt{\zeta}^C_{\rm{BM},c} \circ \iota$, where $\iota: Z(\wt{\mf{g}}) \rightarrow Z(\wt{\mf{g}})$ is the involution coming from $x \mapsto -x$ and $\wt{\zeta}^C_{\rm{BM},c}$ is attached to the determinant $\on{tw}_{\wt{\delta}_{2n}}^{-1}\circ(\wt{D}_c \boxtimes \chi_{\on{cyc}})$ where $\wt{D}_c$ is the determinant as in Corollary \ref{comp. supp. det.} for characteristic polynomials for $\on{Frob}_{\nu}$ given by $\wt{P}_{\nu^c}$ instead of $\wt{P}_{\nu}$.
Therefore, it is left to note that $\wt{\zeta}^C_{\rm{BM}}=\wt{\zeta}^C_{\rm{BM},c} \circ \iota$.
This follows from the two lemmas below.
\end{proof}
\begin{lemma}\label{lemma: pos and neg}
    Let $m \geq 1$ and $\Theta$ a $\GL_m$--pseudocharacter of $\on{Gal}_{F,S}$ valued in $\mathcal{O}(Y)$ for $Y$ a rigid--analytic space over $\mathbf{Q}_p$. Let $\wt{\delta}_m:=(1,t^{-1},\ldots, t^{1-m}): \mathbf G_m \rightarrow \GL_m$ and $\wt{\delta}_m^-=(t^{1-m},\ldots,t^{-1},1): \mathbf{G}_m \rightarrow \GL_m$.  Let $\Theta^C:=\on{tw}_{\wt{\delta}_m}^{-1}\circ (\Theta \boxtimes \chi_{\on{cyc}})$ and $\Theta^{C,-}:=\on{tw}_{\wt{\delta}_m^-}^{-1}\circ (\Theta \boxtimes \chi_{\on{cyc}})$ be the corresponding pseudocharacters of ${}^C\GL_m$ and ${}^C\GL_m^-$ respectively. Then for each prime $\nu \mid p$ of $F$, we have $\zeta^C_{\nu}=\zeta^{C,-}_{\nu}$ where
    $$\zeta^C_{\nu}: Z(\Res_{F_{\nu}/\mathbf Q_p}\mf{gl}_m) \rightarrow \mathcal{O}(Y),  \ \ \zeta^{C,-}_{\nu}: Z(\Res_{F_{\nu}/\mathbf Q_p}\mf{gl}_m) \rightarrow \mathcal{O}(Y)
    $$
    are the generalized Hodge-Tate-Sen characters attached to $\Theta^C_{\nu}$ and $\Theta^{C,-}_{\nu}$ by \cite{PasQua25}, respectively.
\end{lemma}
\begin{proof}
    Let $w_0 \in W^{\GL_m}(T)$ be the  element of the Weyl group of the standard torus $T$ that switches the positive and the negative roots with respect to the standard Borel $B$, and let $\delta_m, \ \delta_m^{-}$ denote the half sums of positive/negative roots. Then $\delta_m^-=\on{ad}(w_0)(\delta_m)$ and $\wt{\delta}_m^-=w_0\wt{\delta}_mw_0^{-1}$, and hence $\wt{\delta}_m-\delta_m = ((1-m)/2,\ldots, (1-m)/2) = \wt{\delta}_m^- - \delta_m^-$. Therefore, the construction of \cite[Section 4.7]{DPS25} (see equation (34)), which applies to generalized Hodge-Tate-Sen characters constructed from Galois representations, is independent of this choice of positive/negative roots. The equality between $\zeta^C$ and $\zeta^{C,-}$ remains true even though they are constructed from pseudocharacters. To see this, first note that the construction of \cite[Theorem 7.11]{PasQua25} is the unique one satisfying the desired functoriality and compatibility with characters constructed from Galois representations, as in Theorem 7.11 in \textit{loc.\,cit.}. Secondly, we have the following commutative diagram and we note that pulling back along the lower horizontal arrow does not change any pseudocharacter attached to $\GL_m \times \mathbf{G}_m$ (in particular $\Theta \boxtimes \chi_{\on{cyc}}$) :
    \[
    \begin{tikzcd}
        {}^C\GL_m\ar[d, "\on{tw}_{\wt{\delta}_m}"]\ar[r, "{\phi:=(g,t) \mapsto (w_0gw_0^{-1},t)}"] & [3em] {}^C\GL_m^-\ar[d, "\on{tw}_{\wt{\delta}_m^-}"]\\
        \GL_m \times \mathbf G_m\ar[r, "{(g,t) \mapsto (w_0gw_0^{-1},t)}"]& [3em] \GL_m \times \mathbf G_m.
    \end{tikzcd}
    \]
    It is not hard to see that pulling back along $\phi$ the unique family of characters for ${}^C\GL_m^-$ given by \cite[Theorem 7.11]{PasQua25} agrees with the unique family for ${}^C\GL_m$, as it satisfies the required conditions by the previous lines. Therefore, we always have $\zeta^C=\zeta^{C,-}$.
\end{proof}
\begin{lemma}\label{lemma: anti + inf char}
    With the notation as above, we have that $\wt{\zeta}^C_{\rm{BM},c} \circ \iota = \wt{\zeta}^{C,-}_{\rm{BM}}$, where $\wt{\zeta}^{C,-}_{\rm{BM}}:= \otimes_{\nu \in \wt{S}}\wt{\zeta}^{C,-}_{\rm{BM},\nu}: Z(\wt{\mf{g}}) \cong \otimes_{\nu \in \wt{S}} Z(\Res_{F_{\nu}/\mathbf Q_p} \mf{gl}_{2n}) \rightarrow \wt{\T}_{\rm{BM}}^S(\wt{K}^p)_{\wt{\mf{m}}}^{\on{rig}}$ defined as in Lemma \ref{lemma: pos and neg}.
\end{lemma}
\begin{proof}
    First let us note that we have the following commutative diagram
    \[
    \begin{tikzcd}
        Z(\wt{\mf{g}})\ar[d, "\HC^B"]\ar[r, "\iota"] & Z(\wt{\mf{g}}) \ar[d, "\HC^{\overline{B}}"]\\
        U(\mf{t})\ar[r, "\iota"]& U(\mf{t})
    \end{tikzcd}
    \]
    where we use $\iota$ to denote involutions on both of $Z(\wt{\mf{g}})$ and $U(\mf{t})(=Z(\mf{t}))$.\footnote{We note that we have a similar diagram for any parabolic subgroup and its Levi subgroup in place of $(B,T)$, e.g. for $(P,G)$ instead of $(B,T)$.} We also note that $\iota: Z(\wt{\mf{g}}) \rightarrow Z(\wt{\mf{g}})$ is induced by the transpose--inverse morphism $\GL_{2n} \rightarrow \GL_{2n}$ since the above diagram is commutative in that case as well and the vertical maps are injections. This implies that $\wt{\zeta}^C_{\rm{BM},c} \circ \iota= \wt{\zeta}^{C,-}_{\rm{BM},1}$ where $\wt{\zeta}^{C,-}_{\rm{BM},1}$ is a generalized Hodge-Tate-Sen character attached to the ${}^C\GL_{2n}^-$--pseudocharacter $\Theta^{C,-}_{\rm{BM},1}:= \on{tw}_{\wt{\delta}^{-,1}}^{-1} \circ (\wt{D}_{c,\vee} \boxtimes \chi_{\on{cyc}})$ with $\wt{\delta}^{-,1}:=(1,t,\ldots,t^{2n-1})$ and $\wt{D}_{c, \vee}$ a determinant as in Corollary \ref{comp. supp. det.} for the polynomial $\wt{P}_{\nu^c}^{\vee}$.\footnote{Again, one can see this is true by first arguing for Galois representations using \cite[Section 4.7]{DPS25}. Then as in Lemma \ref{lemma: pos and neg}, we find an appropriate commutative diagram (lower horizontal arrow this time is given by $\GL_m \xrightarrow{(-)^{-1,T}} \GL_m$ and the identity on $\mathbf G_m$).} Since $\wt{P}_{\nu}(X)=q_{\nu}^{2n(2n-1)}\wt{P}_{\nu^c}^{\vee}(q_{\nu}^{1-2n}X)$ we have that $\wt{D}=\wt{D}_{c, \vee} \otimes \chi_{\on{cyc}}^{1-2n}$ where $\wt{D}$ is the determinant as in Corollary \ref{comp. supp. det.} attached to $\wt{P}_{\nu}$. Since $\wt{\zeta}^{C,-}_{\rm{BM}}$ is defined as a generalized Hodge-Tate-Sen character for the ${}^C\GL_{2n}^-$--pseudocharacter $\on{tw}_{\wt{\delta}_{2n}^-}^{-1}\circ(\wt{D} \boxtimes \chi_{\on{cyc}})$, the claim follows from noting that $(1-2n,\ldots,1-2n)=\wt{\delta}_{2n}^--\wt{\delta}^{-,1}$ (again, see \cite[Equation (34)]{DPS25} and argue as in Lemma \ref{lemma: pos and neg}).
\end{proof}
This completes the proof of Theorem \ref{theorem: com. supp. inf. act.}, so that the locally analytic vectors of compactly supported completed cohomology of $\wt{G}$ localized at $\wt{\mf{m}}$ have infinitesimal action \textit{in all degrees} given by the expected generalized Hodge-Tate-Sen character. In the next two sections, we use this result together with the content of Section \ref{section: facts about completed cohomology of unitary groups} to prove Theorem \ref{theorem: Main Theorem}.

\section{Compatibility of the characters $\zeta^C_{\wt{\rho}}$ and $\zeta^C_{\rho}$}

\label{section: Compatibility between rhoTilde and rho}

In this section we prove the compatibility of the generalized Hodge-Tate-Sen characters for $\wt{G}$ and $G$.
Recall that $\mf{m}$ is a non--Eisenstein decomposed generic maximal ideal of $\T^S(K^p)$ and that $\wt{\mf{m}}:=\mc{S}^{\ast}(\mf{m})$. From Lemma \ref{lemma: pos and neg} and Theorem \ref{theorem: com. supp. inf. act.} we have that the action of $Z(\wt{\mf{g}}):=Z(\Res_{F^+/\mathbf Q}\wt{\mf{g}})_{\mathbf Q_p}$ on $\wt{\HH}^i_c(\wt{K}^p, \mathbf Q_p)_{\wt{\mf{m}}}^{\wt{G}-\on{la}}$ is given by $\wt{\zeta}^{C,-}_{\rm{BM}}$. Let us now introduce the generalized Hodge-Tate-Sen character for $G$ that we consider.

Let $\rho$ be a Galois representation given by Theorem \ref{theorem: Galois representation for GLn} and consider $\rho^C: \on{Gal}_{F,S}\rightarrow {}^C\GL_n(\T^S(K^p)_{\mf{m}}^{\on{rig}})$ given by $\on{tw}_{\wt{\delta}_n}^{-1}\circ (\rho \boxtimes \chi_{\on{cyc}})$ where $\wt{\delta}_n:=(1,t^{-1},\ldots,t^{1-n})$. This gives a character $\zeta^C: Z(\mf{g}):=\otimes_{\nu\mid p} Z(\Res_{F_{\nu}/\mathbf Q_p}\mf{gl}_n) \rightarrow \T^S(K^p)_{\mf{m}}^{\on{rig}}$ (note, this time we take all places $\nu \mid p$ of $F$). As in Lemma \ref{lemma: pos and neg} we have the version of $\rho^{C,-}: \on{Gal}_{F,S} \rightarrow {}^C\GL_n^-(\T^S(K^p)_{\mf{m}}^{\on{rig}})$ given by $\on{tw}_{\wt{\delta}_n^{-}}^{-1}\circ (\rho \boxtimes \chi_{\on{cyc}})$ where $\wt{\delta}_n^{-}=(t^{1-n},\ldots,t^{-1},1)$. The two characters $\zeta^{C,-}$ and $\zeta^C$ agree. Since the opposite parabolic $\overline{P}$ appears in Corollary \ref{cor: transfer of inf act rigid}(2), we will use the constructions for $\zeta^{C,-}$ and $\zeta^{C,-}_{\rm{BM}}$. Recalling the already established compatibility of infinitesimal actions in Corollary \ref{cor: transfer of inf act rigid}(2), it is left to show the following theorem.
\begin{theorem}
\label{theorem: Harish Chandra + Inf Char compatibility}
    The following diagram commutes: 
    \[
    \begin{tikzcd}
        Z(\wt{\mf{g}})\ar[d, "{\wt{\zeta}^{C,-}_{\rm{BM}}}"]\ar[r, "{\HC^{\ol{P}}}"] & Z(\mf{g}) \ar[d, "{\zeta^{{C,-}}}"]\\
\wt{\T}_{\rm{BM}}^S(\wt{K}^p)_{\wt{\mf{m}}}^{\rm{rig}}\ar[r, "\mc{S}"]& \T^S(K^p)_{\mf{m}}^{\rm{rig}}.
    \end{tikzcd}
    \]
\end{theorem}
Intuitively, this result comes from the fact that $\mc{S}(\wt{P}_\nu(X))=P_\nu(X)q_\nu^{n(2n-1)}P_{\nu^c}^{\vee}(q_\nu^{1-2n}X)$ which follows from \cite[Proposition--Definition 5.3]{NT16}.
The proof of the commutativity consists of two steps, the first one being the following proposition.
\begin{proposition}
    There is a commutative diagram
        \[
    \begin{tikzcd}
        Z(\wt{\mf{g}})\ar[d, "\wt{\zeta}^{C}_{\rm{BM},c}"]\ar[r, "{\HC^{P}}"] & Z(\mf{g}) \ar[d, "\zeta^{C}_{\vee,1}"]\\
    \wt{\T}_{\rm{BM}}^S(\wt{K}^p)_{\wt{\mf{m}}}^{\rm{rig}}\ar[r, "\mc{S}"]& \T^S(K^p)_{\mf{m}}^{\rm{rig}},
    \end{tikzcd}
    \]
    where $\wt{\zeta}^C_{\rm{BM},c}$ is as in proof of Theorem \ref{theorem: com. supp. inf. act.} and $\zeta^C_{\vee,1}$ is a generalized Hodge-Tate-Sen character attached to ${}^C\GL_n$--pseudocharacter given by $\on{tw}_{\wt{\delta}^{1}}^{-1}\circ (\rho^{\vee} \boxtimes \chi_{\on{cyc}})$ where $\wt{\delta}^1=(t^{n-1},\ldots,t,1)$.
\end{proposition}
\begin{proof}
    Adopting the notation from the proof of Theorem \ref{theorem: com. supp. inf. act.}, we first remark that the generalized Hodge-Tate-Sen character $\wt{\zeta}^C_{\rm{BM},c}$ depends only on the determinant $\wt{D}_{c}$, therefore $\mc{S} \circ \wt{\zeta}^C_{\rm{BM},c}$ depends only on $\mc{S}(\wt{D}_c)$, which has characteristic polynomials for all $\on{Frob}_\nu$, $\nu \notin S$:
    $$\mc{S}(\wt{D}_c)(X-\on{Frob}_\nu)=P_{\nu^c}(X)\cdot q_\nu^{n(2n-1)}P_\nu^{\vee}(q_\nu^{1-2n}X).
    $$
    Furthermore, by appropriate versions of Theorem \ref{theorem: Galois representation for GLn} we have Galois representations $\rho^c, \ \rho^{\vee}: \on{Gal}_{F,S}\rightarrow \GL_n(\T^S(K^p)_{\mf{m}}^{\on{rig}})$ with respective characteristic polynomials given by $P_{\nu^c}(X)$ and $P_\nu^{\vee}(X)$.\footnote{Note that $\rho^c$ is given by composition of $\rho$ with $\on{Gal}_{F,S} \xrightarrow{\sigma \mapsto c\circ\sigma\circ c} \Gal_{F,S}$ which sends $\on{Frob}_\nu \rightarrow \on{Frob}_{\nu^c}$.} Also note that $\rho^{\vee} \otimes \chi_{\on{cyc}}^{1-2n}$ has characteristic polynomials given by $q_\nu^{n(2n-1)}P_\nu^{\vee}(q_\nu^{1-2n}X)$. Therefore, $\mc{S} \circ \wt{\zeta}^C_{\rm{BM},c}$ is constructed from 
    \begin{equation}\label{eq: rho tilde}
        \wt{\rho}^C:=\on{tw}_{\wt{\delta}_{2n}}^{-1}\circ ((\rho^c \times \rho^{\vee}\otimes \chi_{\on{cyc}}^{1-2n}) \boxtimes \chi_{\on{cyc}})
    \end{equation}
    where $\rho^c\times \rho^{\vee}\otimes\chi_{\on{cyc}}^{1-2n}: \on{Gal}_{F,S}\rightarrow \GL_n(\T^S(K^p)_{\mf{m}}^{\on{rig}}) \times \GL_n(\T^S(K^p)_{\mf{m}}^{\on{rig}}) \hookrightarrow \GL_{2n}(\T^S(K^p)_{\mf{m}}^{\on{rig}})$ is the corresponding block--diagonal representation. We also recall that $\mc{S} \circ \wt{\zeta}^C_{\rm{BM},c}$ is given by restricting the generalized Hodge-Tate-Sen for $\wt{\rho}^C$ only to places $\nu \in \wt{S}$ (therefore not all places in $S_p(F)$; on the other hand, for $G$ we will have both $\nu$ and $\nu^c$ appear, even though we look only at the generalized Hodge-Tate-Sen character for $\wt{\rho}^C$ restricted to $\nu\in \wt{S}$).

    To proceed further, let us recall how $G(F_p^+)\cong \prod_{\nu\mid p} \GL_n(F_{\nu})$ sits inside $\wt{G}(F_p^+)=\prod_{\nu \in \wt{S}} \GL_{2n}(F_{\nu})$, where we identify $F_{\nu} \simeq F_{\overline{\nu}}^+ \simeq F_{\nu^c}$. We have that $\GL_n(F_{\nu^c})\times \GL_n(F_{\nu})$ sits inside the block--diagonal $\GL_n \times \GL_n \subset \GL_{2n}(F_{\nu})$ via $(A,B)\mapsto (\Psi_nA^{-1,T}\Psi_n, B)$ where $\Psi_n$ is the anti--diagonal matrix (see \cite[Lemma 5.1]{NT16} and note that we omit $c$ from notation since we have identified $F_{\nu} \simeq F_{\nu^c}$). Also note that $\Psi_n A^{-1,T}\Psi_n$ is given by reflecting $A^{-1}$ with respect to the anti--diagonal.

    Let us now consider $\rho'=\rho^c \times \rho^{\vee}$ as a $\GL_n \times \GL_n$--representation and let $\zeta'$ be the generalized Hodge-Tate-Sen character attached to ${}^C(\GL_n \times \GL_n)$--Galois representation $\on{tw}_{\wt{\delta}'}^{-1}\circ(\rho' \boxtimes \chi_{\on{cyc}})$ where $\wt{\delta}'=(1,t^{-1},\ldots,t^{1-n}, t^{n-1},\ldots,t,1)$. The explicit formulae in the previous paragraph imply that the character
    \[
    \zeta: Z(\mf{g})=\bigotimes_{\nu \in \wt{S}}\left(Z(\Res_{F_{{\nu}^c}/\mathbf Q_p}\mf{gl}_n) \otimes_{\Q_p} Z(\Res_{F_{\nu}/\mathbf Q_p}\mf{gl}_n)\right) \xrightarrow[F_{\nu}\simeq F_{\nu^c}]{\otimes_{\nu \in \wt{S}}\zeta'_\nu} \T^S(K^p)_{\mf{m}}^{\on{rig}}
    \]
    equals $\zeta^C_{\vee,1}$ (where $G(F_{\nu^c})\times G(F_\nu) \rightarrow (\GL_n \times \GL_n) (F_\nu)$ via the map above). The only thing left to note is the commutativity of the following diagram:
     \[
    \begin{tikzcd}
        Z(\Res_{F_\nu/\mathbf Q_p}\mf{gl}_{2n})\ar[d, "{\zeta_{\wt{\rho}^C, \nu}}"]\ar[r, "{\HC^{P}}"] & Z(\Res_{F_\nu/\mathbf{Q}_p}(\mf{gl}_n \times \mf{gl}_n)) \ar[d, "\zeta^{'}_\nu"]\\
    \wt{\T}_{\rm{BM}}^S(\wt{K}^p)_{\wt{\mf{m}}}^{\rm{rig}}\ar[r, "\mc{S}"]& \T^S(K^p)_{\mf{m}}^{\rm{rig}},
    \end{tikzcd}
    \]
    This follows from equation \eqref{eq: rho tilde}, \cite[Equation (34)]{DPS25}, and the fact that
    $$\on{diag}(\underbrace{0,\ldots,0}_{n},\underbrace{1-2n,\ldots,1-2n}_{n})- \on{diag}(0,-1,\ldots,1-2n)=-\on{diag}(0,-1,\ldots,1-n,n-1,\ldots,1,0)
    $$
    keeping in mind that $\HC^{B\cap (\GL_n \times \GL_n)}\circ \HC^P=\HC^B$ where $B$ corresponds to the standard Borel subgroup of $\GL_{2n}$, $P$ corresponds to the parabolic subgroup of block $n\times n$ upper--triangular matrices, and all Harish-Chandra maps are the unnormalized ones (as usual in this paper).
\end{proof}
\begin{proof}[Proof of Theorem \ref{theorem: Harish Chandra + Inf Char compatibility}]
As in the proof of Lemma \ref{lemma: anti + inf char}, we have a commutative diagram  
    \[
    \begin{tikzcd}
        Z(\wt{\mf{g}})\ar[d, "\HC^{\overline{P}}"]\ar[r, "\iota"] & Z(\wt{\mf{g}}) \ar[d, "\HC^{P}"]\\
        Z(\mf{g})\ar[r, "\iota"]& Z(\mf{g})
    \end{tikzcd}
    \]
    where we again use $\iota$ to denote involutions on both $Z(\wt{\mf{g}})$ and $Z(\mf{g})$. The proof then follows from noting that $\wt{\zeta}^{C,-}_{\rm{BM}}=\wt{\zeta}^C_{\rm{BM},c}\circ \iota$ and $\zeta^{C,-}=\zeta^{C}_{\vee,1}\circ \iota$, the first by Lemma \ref{lemma: anti + inf char} and the second by noting again that $\iota:Z(\mf{g}) \rightarrow Z(\mf{g})$ is induced by transpose--inverse on $\GL_n$ and arguing as in the proof of Lemma \ref{lemma: anti + inf char}.
\end{proof}
The results so far imply that Theorem \ref{theorem: Main Theorem} is true for the subalgebra $\HC^{\overline{P}}(Z(\wt{\mf{g}})) \subset Z(\mf{g})$. In order to complete the proof of the main theorem, we need to cover the whole $Z(\mf{g})$. This is done in the next section by twisting with characters.
\section{Twisting and the infinitesimal action for $\GL_n$}

\label{section: Twisting and proof of the main theorem}

In this section, we will finish proving Theorem \ref{theorem: Main Theorem} (see Theorem \ref{theorem: main theorem v2} and Remark \ref{Remark: actually proving the main theorem}). From previous sections, we know that $\HC^{\ol{P}}(Z(\wt{\mf{g}}))\subset Z(\mf{g})$ acts on $\wt{\HH}^i(K^p,\Q_p)_{\mf{m}}^{\on{la}}$ through the expected generalized Hodge-Tate-Sen character $\zeta^{C,-}$. Since this is a proper subalgebra of $Z(\mf{g})$, we need to work more to deduce the main theorem on the whole of $Z(\mf{g})$. Let us briefly explain the strategy. We may base--change to a sufficiently large finite extension $L/\Q_p$, as usual. We will twist\footnote{We mention that the idea of twisting by a character appeared already in other works (starting with \cite{HLTT16}), with slightly different technical details.} the maximal ideal $\mf{m} \subset \T^S(K^p)$ by some continuous character $\chi$ of the Galois group $\Gal_F$, which gives us a new maximal ideal $\mf{m}(\chi)$ for which we can run the arguments of the previous sections to get that $\HC^{\ol{P}}(Z(\wt{\mf{g}})_L)$ acts on $\wt{\HH}^i(K^p,L)_{\mf{m}(\chi)}^{\on{la}}$ through the expected generalized Hodge-Tate-Sen character $\zeta^{C,-}_{\chi}$. 

It will be sufficient to consider $\chi:=\chi_{\on{cyc}}$, for which we have an explicit map $\phi:Z(\mf{g})_L\rightarrow Z(\mf{g})_L$ which intertwines the actions of $Z(\mf{g})_L$ on $\wt{\HH}^i(K^p,L)^{\on{la}}_{\mf{m}}$ and on $\wt{\HH}^i(K^p,L)^{\on{la}}_{\mf{m}(\chi)}$. In particular, this $\phi$ also intertwines the generalized Hodge-Tate-Sen characters $\zeta^{C,-}$ and $\zeta^{C,-}_{\chi}$. From here we deduce that $\phi(\HC^{\ol{P}}(Z(\wt{\mf{g}})_L))$ acts on $\wt{\HH}^i(K^p,L)_{\mf{m}}^{\on{la}}$ as expected and the final step is then to note, via explicit computations using Harish-Chandra isomorphisms, that $\phi(\HC^{\ol{P}}(Z(\wt{\mf{g}})_L))$ and $\HC^{\ol{P}}(Z(\wt{\mf{g}})_L)$ span (as an algebra) the whole $Z(\mf{g})_L$.

In the first subsection, we will do the twisting of $\mf{m}$, calculate $\phi$ explicitly and show that it intertwines $\zeta^{C,-}$ and $\zeta^{C,-}_{\chi}$. In the second subsection, using the results of the first and explicit calculation with Harish-Chandra isomorphisms for $\wt{G}$ and $G$, we will finish the proof of Theorem \ref{theorem: Main Theorem}.
\subsection{Twisting the maximal ideal}
Let $L$ be a finite extension of $\mathbf Q_p$ with $[F:\mathbf Q]$ embeddings $F\hookrightarrow L$ as usual.
Fix the tame level $K^p \subset \GL_n(\A_F^{p,\infty})$\footnote{Since we will consider only $\GL_n$ in this subsection (and not $\tG$), we may consider it over $F$, rather than $\Res_{F/F^+}\GL_n$ over $F^+$.}, and let $\chi: \Gal_F \to \mO_L^{\times}$ be a continuous character such that $\chi \circ \Art_{F_{\nu}}|_{\det(K_{\nu})} = 1$ for each place $\nu \notin S$ of $F$. This choice induces an isomorphism on the abstract Hecke algebra (see subsection \ref{subsection: big Hecke} for notation)
\[
f_{\chi}: \T^S \otimes_{\mO_L} L \to \T^S \otimes_{\mO_L} L
\]
given by the formula
\begin{equation}\label{eq: twist on Hecke}
    f_{\chi}(f)(g) = \chi(\Art_F(\det(g)))^{-1}f(g) \ .
\end{equation}
Given a maximal ideal $\mf{m}\subset \T^S$ we define $\mf{m}(\chi) := f_{\chi}(\mf{m})$. Likewise we let $\wt{\mf{m}}(\chi) := \mc{S}^{\ast}(\mf{m}(\chi))$.
Let $\chi_p:\GL_n(F_p) \to \mO_L^{\times}$ be defined by $(g_{\nu})_{\nu \mid p} \mapsto \prod_{\nu \mid p}\chi(\Art_{F_{\nu}}(\det(g_{\nu})))$. For us it will be sufficient to consider $\chi:=\chi_{\on{cyc}}$, although the arguments below hold more generally. 

First we set up some more notation. For each $\nu \mid p$, write $d\chi_{p,\nu}: U(\Res_{F_\nu/\Q_p}\mf{gl}_n) \rightarrow L$ for the action of the universal enveloping algebra given by differentiating $\chi_{p,\nu}: \GL_n(F_\nu) \rightarrow \GL_n(F_p) \xrightarrow{\chi_p} \mathcal{O}_L^\times$. In the case $\chi= \chi_{\on{cyc}}$, by local class field theory, we have that for any $x = (x_{i,j}) \in \Res_{F_\nu/\Q_p} \mf{gl}_n = \on{Mat}_{n\times n}(F_{\nu})$, $d\chi_{p,\nu}(x)= \on{Tr}_{F_\nu/\Q_p}(\sum_i x_{i,i})$, where $\on{Tr}_{F_\nu/\Q_p}: F_\nu \to \Q_p$ is the trace map. 

We also define the algebra homomorphism $\phi_{\chi}: U(\mf{g}) \rightarrow U(\mf{g})$ such that for $\nu \mid p$ and every $x_{\nu} \in \Res_{F_{\nu}/\Q_p} \mf{gl}_n \subset (\Res_{F/\Q} \mf{gl}_n)_{\Q_p}$ we have that $\phi_{\chi}(x_\nu)=x_{\nu}+d\chi_{p, \nu}(x_{\nu})$. We first record the following simple lemma.
\begin{lemma}\label{lemma: compatibility for inf action GLn}
    Let $V$ be an admissible $L$-Banach representation of $\GL_n(F_p)$. Denote by $\zeta$ and $\zeta_\chi$ the actions of $Z(\Res_{F/\Q }\mf{gl}_n)_{\Q_p}$ on $V^{\on{la}}$ and on $V^{\on{la}} \otimes_L \chi_p^{-1}$, respectively. Then the natural isomorphism $V^{\on{la}} \simeq V^{\on{la}}\otimes_L \chi_p^{-1}$ is equivariant for the action of $Z(\Res_{F/\Q} \mf{gl}_n)_{\Q_p}$ given by $\zeta$ on the left and by $\zeta_{\chi} \circ \phi_{\chi}|_{Z(\Res_{F/\Q }\mf{gl}_n)_{\Q_p}}$ on the right.
\end{lemma}
\begin{proof}
    The actions of $U(\Res_{F/\Q }\mf{gl}_n)_{\Q_p}$ (and also of $ Z(\Res_{F/\Q }\mf{gl}_n)_{\Q_p}$) are given by the universal property applied to the actions of $ (\Res_{F/\Q }\mf{gl}_n)_{\Q_p}$ on $V^{\on{la}}$ and $V^{\on{la}} \otimes_L \chi_p^{-1}$. Since $\chi_p^{-1}$ is locally analytic with infinitesimal action given by $-d\chi_{p, \nu}$ when restricted to $\Res_{F_\nu/\Q_p} \mf{gl}_n$, we get the claim.
\end{proof}
We now consider how twisting by $\chi$ changes the action of $\GL_n(F_p)$ on the completed cohomology.
\begin{lemma}
\label{lemma: twisting completed cohomology}
There is a $\GL_n(F_p)$-equivariant isomorphism
$t_{\chi}:\wt{\HH}^i(K^p, L)_{\mf{m}}^{(\la)} \simeq \wt{\HH}^i(K^p, L)_{\mf{m}(\chi)}^{(\la)}\otimes_L \chi_p^{-1}$, which intertwines the usual $\T^S \otimes_{\mathcal{O}_L}L$-action on the left and the usual action precomposed with $f_{\chi}$ on the right.
\end{lemma}
\begin{proof}
    This is essentially contained in the proof of \cite[Lemma 6.7]{mcdonald2025eigenvarieties}, whose reasoning we briefly recall. Namely, 
    \cite[Lemma 2.5]{Hevesi23} gives us an isomorphism $R\Gamma(X_{K^pK_p}, V) \to R\Gamma(X_{K^pK_p}, V\otimes_L \chi_p^{-1})$ for $V$ a smooth left $K_p$--module which is finite free as an $\mO_L/\varpi_L^m$--module, intertwining the usual Hecke action with the one precomposed with $f_{\chi}$. These maps are manifestly compatible with changing the level at $p$, and so plugging in $V = \mO_L/\varpi_L^m$,  localizing at $\mf{m}$, and passing to the limit yields an isomorphism $R\Gamma(K^p, L)_{\mf{m}}\simeq R\Gamma(K^p, L\otimes \chi_p^{-1})_{\mf{m}(\chi)}\simeq R\Gamma(K^p, L)_{\mf{m}(\chi)}\otimes_L\chi_p^{-1}$. Passing to locally analytic vectors yields the analogous result, noting that $\chi_p$ is a continuous character (and thus automatically locally analytic).
\end{proof}
Lemma \ref{lemma: twisting completed cohomology} and its proof imply that the map $f_{\chi}$ induces an isomorphism $f_{\chi}: \T^S(K^p)\rightarrow \T^S(K^p)$ on the Hecke algebra defined in subsection \ref{subsection: big Hecke} (its inverse is given by $f_{\chi^{-1}}$). In particular $\mf{m}(\chi)$ is an open maximal ideal of $\T^S(K^p)$ and we have an isomorphism $f_{\chi}: \T^S(K^p)_{\mf{m}}^{(\on{rig})}\simeq \T^S(K^p)_{\mf{m}(\chi)}^{(\on{rig})}$.
The following is a consequence of Lemma \ref{lemma: compatibility for inf action GLn} and Lemma \ref{lemma: twisting completed cohomology}.
\begin{corollary}[Twist on the automorphic side]\label{cor: twist on the automorphic side}
    The natural isomorphism $\wt{\HH}^i(K^p, L)^{\on{la}}_{\mf{m}(\chi)} \simeq \wt{\HH}^i(K^p,L)_{\mf{m}}^{\on{la}}$ given by Lemma \ref{lemma: twisting completed cohomology} intertwines the usual $\T^S(K^p)_{\mf{m}(\chi)}^{\on{rig}}$-action on the left and the $\T^S(K^p)_{\mf{m}(\chi)}^{\on{rig}}$-action on the right given by composing with $f_{\chi}^{-1}: \T^S(K^p)_{\mf{m}(\chi)}^{\on{rig}}\rightarrow \T^S(K^p)_{\mf{m}}^{\on{rig}}$. Furthermore, the natural action of $Z(\Res_{F/\Q} \mf{gl}_n)_{\Q_p}$ on the left intertwines with the natural action on the right precomposed with $\phi_{\chi}$, with the notation as in Lemma \ref{lemma: compatibility for inf action GLn}.
\end{corollary}
Next, we want to show that the twist on the automorphic side $\phi_{\chi}$ intertwines the expected actions of $Z(\mf{g})$ on $\wt{\HH}^i(K^p,L)_{\mf{m}(\chi)}^{\on{la}}$ and $\wt{\HH}^i(K^p, L)^{\on{la}}_{\mf{m}}$, i.e. that it agrees with the twist on the Galois side.

 We note that $\mf{m}(\chi)$ is also non--Eisenstein and decomposed generic (its residual representation equals a twist of the residual representation for $\mf{m}$ by $\chi$, see \eqref{eq: twist on Hecke}). Let $\zeta^{C,-}_{\chi}$ denote the generalized Hodge-Tate-Sen character attached to the Galois representation given by Theorem \ref{theorem: Galois representation for GLn} for $\mf{m}(\chi)$, by the same recipe as in Section \ref{section: Compatibility between rhoTilde and rho}.
\begin{remark}\label{remark: prev section for m(chi)}
From the previous sections applied to $\mf{m}(\chi)$ instead of $\mf{m}$, we know that the action of $\HC^{\overline{P}}(Z(\wt{\mf{g}})_L)$ on $\wt{\HH}^i(K^p,L)_{\mf{m}(\chi)}^{\on{la}}$ is given by $\zeta_{\chi}^{C,-}|_{\HC^{\overline{P}}(Z(\wt{\mf{g}})_L)}$.
\end{remark}
For simplicity and since it will turn out to be sufficient, let us now specialize our computation to $\chi:=\chi_{\on{cyc}}$. We want to prove that the characters $\zeta^{C,-}$ and $\zeta^{C,-}_{\chi}$ fit into the following diagram.
\begin{proposition}[Galois twist agrees with automorphic twist]
\label{proposition: Twists Galois inf char of GLn by chi}
    Let $\chi:=\chi_{\on{cyc}}$. The following diagram commutes:
        \[
   \begin{tikzcd} Z(\mf{g})_L\ar[d, "{\phi_{\chi}\otimes L}"] \ar[r, "{\zeta^{C,-}_{\chi}}"]& [5em]\T^S(K^p)_{\mf{m}(\chi)}^{\rm{rig}}\ar[d, "{f_{\chi}^{-1}}"]\\
   Z(\mf{g})_L\ar[r, "{\zeta^{C,-}}"] & [5em] \T^S(K^p)_{\mf{m}}^{\rm{rig}}.
   \end{tikzcd}
    \]
\end{proposition}
The proof involves an explicit calculation and careful bookkeeping of all $\nu \mid p$ and all $\tau: F_\nu \hookrightarrow L$. Before we start, let us fix the notation $x_1,\ldots,x_m$ for the standard basis of the Lie algebra of the diagonal torus of $\GL_m$ over any finite extension of $\Q_p$. We will add extra indices if we want to specify the extension of $\Q_p$ and its embedding into $L$.

For each $\nu \mid p$, $\Res_{F_\nu/\Q_p} \mf{gl}_n \otimes_{\Q_p}L= \bigoplus_{\tau} \mf{gl}_n \otimes_{\tau}L$ and under this equality, the standard $i$--th diagonal element $x_i^{\nu}$ in $\Res_{F_\nu/\Q_p} \mf{gl}_n$ maps to $\oplus_{\tau} x_i^{\nu,\tau}$. In particular, for any $\lambda \in F_\nu$, $\lambda x_i^{\nu} \mapsto \oplus_{\tau} \tau(\lambda) x_i^{\nu, \tau}$. Likewise we have an induced isomorphism $Z(\Res_{F_\nu/\Q_p} \mf{gl}_n) \otimes_{\Q_p}L \simeq \otimes_{\tau: F_\nu \hookrightarrow L} Z( \mf{gl}_{n, F_{\nu}}) \otimes_{\tau}L $. We will drop the index $F_\nu$ from $\mf{gl}_n$ when it is clear from the context. For each $\nu$ and $\tau$ we have the unnormalized Harish-Chandra morphism  $\HC^{\nu, \tau}: Z( \mf{gl}_{n, F_{\nu}}) \otimes_{\tau}L \hookrightarrow L[x_1^{\nu, \tau},\ldots,x^{\nu, \tau}_n]$ with respect to the standard Borel. 
\begin{remark}\label{remark: normalized and unnormalized}
    Note that any $L$--algebra morphism $\psi:L[x_1^{\nu, \tau},\ldots,x^{\nu, \tau}_n] \rightarrow L[x_1^{\nu, \tau},\ldots,x^{\nu, \tau}_n]$ sending $x_i^{\nu, \tau} \mapsto x_i^{\nu, \tau} +\lambda$  for all $i$ and a fixed $\lambda\in L$, automatically lifts to a unique $L$--algebra morphism $Z(\mf{gl}_n)\otimes_\tau L \rightarrow Z(\mf{gl}_n)\otimes_\tau L$, since $\psi$ commutes with the normalized version of $\HC^{\nu, \tau}$, which is an isomorphism onto its image (which is preserved under $\psi$).
\end{remark}

Now consider a map $\phi= \otimes_{\nu,\tau} \phi_{\nu, \tau}: Z(\mf{g})_L \rightarrow Z(\mf{g})_L$ where $\phi_{\nu, \tau}: Z(\mf{gl}_n) \otimes_{\tau:F_\nu \hookrightarrow L} L \rightarrow Z(\mf{gl}_n) \otimes_{\tau:F_\nu \hookrightarrow L} L$ is the $L$--algebra morphism which, after composing with $\HC^{\nu, \tau}$, commutes with the map on $L[x_1^{\nu, \tau},\ldots,x_n^{\nu, \tau}]$ given by $x_i^{\nu, \tau} \mapsto x_i^{\nu, \tau} + 1$ for all $i$. We first note that $\phi$ is the twist appearing on the Galois side (the intuition is that shifting $x_i^{\nu, \tau}$ by $1$ matches the fact that Hodge--Tate weights of $\chi_{\on{cyc}}$ are all 1).
\begin{lemma}\label{lemma: Galois twist}
    The following diagram commutes:
        \[
   \begin{tikzcd} Z(\mf{g})_L\ar[d, "{\phi}"] \ar[r, "{\zeta^{C,-}_{\chi}}"]& [5em]\T^S(K^p)_{\mf{m}(\chi)}^{\rm{rig}}\ar[d, "{f_{\chi}^{-1}}"]\\
   Z(\mf{g})_L\ar[r, "{\zeta^{C,-}}"] & [5em] \T^S(K^p)_{\mf{m}}^{\rm{rig}}.
   \end{tikzcd}
    \]
\end{lemma}
\begin{proof}
It suffices to show for each $\nu \mid p$ and each $\tau: F_\nu \hookrightarrow L$ that the right square of the following diagram commutes (the left is commutative by definition):
    \begin{equation*}
   \begin{tikzcd} 
   L[x_1^{\nu, \tau},\dots, x_{n}^{\nu, \tau}]\ar[d,"x_i^{\nu,\tau} \mapsto x_i^{\nu, \tau}+1"] &Z(\mf{gl}_n)\otimes_\tau L \ar[d,"\phi_{\nu, \tau}"] \ar[l, "\on{HC}^{\nu, \tau}"]\ar[r, "{\zeta^{C,-}_{\chi}}"]& \T^S(K^p)_{\mf{m}(\chi)}^{\rm{rig}}\ar[d, "f_{\chi}^{-1}"]\\
  L[x_1^{\nu, \tau},\dots, x_{n}^{\nu, \tau}]  & Z(\mf{gl}_n)\otimes_{\tau}L\ar[l, "\on{HC}^{\nu, \tau}"]\ar[r, "{\zeta^{C,-}}"] & \T^S(K^p)_{\mf{m}}^{\rm{rig}}.
   \end{tikzcd}
\end{equation*}
Now let $\rho_{\mf{m}}: \Gal_{F,S} \to \GL_n(\T^S(K^p)_{\mf{m}})$ and $\rho_{\mf{m}(\chi)}: \Gal_{F,S} \to \GL_n(\T^S(K^p)_{\mf{m}(\chi)})$ be the representations from Theorem \ref{theorem: Galois representation for GLn}. As in \cite[p.$\,$933]{10AuthorPaper}, for $\nu \notin S$, $f_{\chi}^{-1}(T_{\nu,i}) = \chi(\Frob_{\nu})^iT_{\nu,i}$, so by Chebotarev density we have $\rho_{\mf{m}}\otimes \chi\simeq f_{\chi}^{-1}(\rho_{\mf{m}(\chi)})$. Passing to the associated $C$-group representations as in the beginning of \cref{section: Compatibility between rhoTilde and rho} and following the recipe of \cite[Section 4.7]{DPS25} for generalized Hodge-Tate-Sen characters (the normalizations are the same for both representations) we get the commutativity of the diagram because the Hodge--Tate weights of $\chi_{\on{cyc}}$ are $1$ for all $\nu, \tau$.
\end{proof}
The remaining claim is that the automorphic twist agrees with $\phi$.
\begin{lemma}
\label{lemma: cyclo on nu tau just adds 1}
    The map $\phi_{\chi}\otimes L$ for $\chi:=\chi_{\on{cyc}}$ agrees with $\phi$.
\end{lemma}
\begin{proof}
    We may consider each $\nu \mid p$ individually. Note that by Poincar\'e--Birkhoff--Witt $U(\Res_{F_\nu/\Q_p} \mf{gl}_n)= (\Res_{F_\nu/\Q_p} \mf{n}^-) U(\Res_{F_\nu/\Q_p} \mf{b}^-) \oplus U(\Res_{F_\nu/\Q_p} \mf{t}) \oplus U(\Res_{F_\nu/\Q_p} \mf{gl}_n) (\Res_{F_\nu/\Q_p} \mf{n})$ where $\mf{n}^-, \mf{b}^-$ and $\mf{n}$ are Lie algebras of the opposite unipotent, the opposite of the standard Borel and the unipotent of the standard Borel, respectively. From the discussion before Lemma \ref{lemma: compatibility for inf action GLn}, we have that $d\chi_{p, \nu}(x)= 0$ for any $x \in\Res_{F_\nu/\Q_p} \mf{n}$ and any $x\in \Res_{F_\nu/\Q_p} \mf{n}^-$. Therefore, $\phi_{\chi}$ only changes $x \in \Res_{F_\nu/\Q_p} \mf{t}$. 

    Next, using that $\Res_{F_\nu/\Q_p} \mf{gl}_n \otimes_{\Q_p} L \simeq \oplus_\tau \mf{gl}_n \otimes_\tau L$, the previous paragraph and the definition of $\HC^{\nu, \tau}$ imply that it is enough to show that the following diagram commutes:
            \[
   \begin{tikzcd} U(\Res_{F_\nu/\Q_p} \mf{t}) \ar[d, "\psi"] \ar[r, "(\phi_\chi)|_{U(\Res_{F_\nu/\Q_p} \mf{t})}"]& [5em] U(\Res_{F_\nu/\Q_p} \mf{t}) \ar[d, "\psi"]\\
   \otimes_{\tau} L[x_1^{\nu, \tau},\ldots, x_n^{\nu, \tau}]\ar[r, "{\phi: \ x_i^{\nu, \tau} \mapsto x_i^{\nu, \tau}+1}"] & [5em] \otimes_{\tau} L[x_1^{\nu, \tau},\ldots, x_n^{\nu, \tau}].
   \end{tikzcd}
    \]
    Here, $\psi:U(\Res_{F_\nu/\Q_p} \mf{t}) \rightarrow U(\Res_{F_\nu/\Q_p} \mf{t}) \otimes_{\Q_p} L = \otimes_{\tau} L[x_1^{\nu, \tau},\ldots, x_n^{\nu, \tau}]$ is the natural inclusion. It suffices to check commutativity on every $x \in \Res_{F_\nu/\Q_p}\mf{t}$. Let $x_1,\ldots,x_n$ be the standard basis of $\mf{t}$ as always. It suffices to check commutativity on $\lambda x_i$, for $\lambda \in F_\nu$ and $i=1,\ldots,n$. This follows from noting that $\psi(\lambda x_i)= \oplus_{\tau} \tau(\lambda) x_i^{\nu, \tau}$, therefore $\phi \circ \psi (\lambda x_i)= \psi( \lambda x_i) + \on{Tr}_{F_\nu/\Q_p}(\lambda)= \psi( \lambda x_i) + d\chi_{p, \nu}(\lambda x_i)$, which is what we wanted.
\end{proof}
\begin{proof}[Proof of \cref{proposition: Twists Galois inf char of GLn by chi}] 
The proposition now follows from Lemma \ref{lemma: Galois twist} and Lemma \ref{lemma: cyclo on nu tau just adds 1}.
\end{proof}
In particular, we have the following corollary.
\begin{corollary}\label{corollary: phi(HC)}
    The action of $\phi(\HC^{\ol{P}}(Z(\wt{\mf{g}})_L))$ on $\wt{\HH}^i(K^p, L)_{\mf{m}}^{\on{la}}$ is given by $\zeta^{C,-}$.
\end{corollary}
\begin{proof}
    This is a consequence of Proposition \ref{proposition: Twists Galois inf char of GLn by chi} and Remark \ref{remark: prev section for m(chi)}.
\end{proof}
\subsection{Conclusion by explicit calculation}
In this subsection we show that $\phi(\HC^{\ol{P}}(Z(\wt{\mf{g}})_L))$ and $\HC^{\ol{P}}(Z(\wt{\mf{g}})_L)$ generate, as an algebra, the whole $Z(\mf{g})_L$, which, together with Theorem \ref{theorem: Harish Chandra + Inf Char compatibility} and Corollary \ref{corollary: phi(HC)}, will imply the main Theorem \ref{theorem: Main Theorem}. We do this by an explicit calculation, using the Harish-Chandra isomorphisms for $\tG(=U(n,n)/F^+)$ and $G(=\Res_{F/F^+} \GL_n)$.

Let $\gamma_{\tg}: Z(\tg) \simeq U(\mf{t})^{W_{\tG}}$ and $\gamma_{\mf{g}}: Z(\mf{g}) \simeq U(\mf{t})^{W_{G}}$ denote the normalized Harish-Chandra isomorphisms, \textit{with respect to the opposite standard Borels} $\ol{\mf{b}} \subset \tg, \ol{\mf{b}}_n \subset \mf{g}$. Let $\wt{S} \subset S_p(F)$ be a set of places such that $S_p(F) = \wt{S}\sqcup \wt{S}^c$, as chosen in subsection \ref{subsection: inf act on unitary} and implicitly used since then in the construction of generalized Hodge-Tate-Sen characters for $\wt{G}$. 

We view $\mf{t} = \on{Lie} T(F_p^+) \simeq \bigoplus_{\nu \in \wt{S}}\mf{t}_{\nu}$ and all other Lie algebras as Lie algebras over $\mathbf Q_p$. Then $\mf{t}_L:=\mf{t} \otimes_{\mathbf Q_p} L \simeq \bigoplus_{\nu\in \wt{S}, \tau: F_{\nu}\hookrightarrow L}\mf{t}_{\nu}\otimes_{F_{\nu},\tau}L\simeq \bigoplus_{\nu, \tau}L^{2n}$, and so $\gamma_{\tg}$ induces an isomorphism

\[
Z(\wt{\mf{g}})_L\simeq \bigotimes_{\nu \in \wt{S}, \tau: F_{\nu}\hookrightarrow L}Z(\mf{gl}_{2n})_L\simeq \bigotimes_{\nu \in \wt{S}, \tau: F_{\nu} \hookrightarrow L}L[x_1,\dots, x_{2n}]^{S_{2n}}.
\]

Similarly,  $\gamma_{\mf{g}}$ with the choice of $\wt{S}$ induces an isomorphism
\[
Z(\mf{g})_L\simeq \bigotimes_{\nu \in \wt{S}, \tau: F_{\nu} \hookrightarrow L}Z(\mf{gl}_n)_L\otimes_L Z(\mf{gl}_n)_L \simeq   \bigotimes_{\nu \in \wt{S}, \tau: F_{\nu}\hookrightarrow L}L[x_1,\dots, x_{2n}]^{S_n\times S_n},
\]
where the $S_n\times S_n$ comes from the fact that we group together the two split places $\nu^c, \nu \in S_p(F)$ ($\nu^c$ is for the first copy, with $\nu\in \wt{S}$). 
Very explicitly, the unnormalized Harish-Chandra map has the following description on the corresponding polynomial algebras:

\begin{lemma}
\label{lemma: HC Explicit Polynomial Description}
    There is a commutative diagram 
    \[
    \begin{tikzcd}
        Z(\wt{\mf{g}})_L\ar[d,"{\HC^{\overline{P}}}"]\ar[r, "{\gamma_{\tg}}"]& \bigotimes_{\nu, \tau}L[x_1,\dots,x_{2n}]^{S_{2n}}\ar[d, "E"]\\
        Z(\mf{g})_L\ar[r, "\gamma_{\mf{g}}"]& \bigotimes_{\nu, \tau}L[x_1,\dots, x_{2n}]^{S_n\times S_n},
    \end{tikzcd}
    \]
    where the map $E:= \bigotimes_{\nu, \tau}E_{\nu, \tau}$, with $E_{\nu, \tau}: L[x_1,\dots, x_{2n}]^{S_{2n}}\to L[x_1,\dots, x_{2n}]^{S_{n}\times S_n}$ induced by sending $x_i \mapsto -x_i - n/2$ and $x_{i+n} \mapsto x_{i+n} + n/2$ for $1 \le i \le n$.
\end{lemma}
\begin{proof}
    We recall that the embedding $G(F_p^+) \subset \tG(F_p^+)$ is given by the composition $\GL_n(F_{\nu^c}) \times \GL_n(F_{\nu}) \xrightarrow[\sim]{j_{\vee}}\GL_n(F_{\nu})\times \GL_n(F_{\nu}) \subset \GL_{2n}(F_{\nu})$, where $j_{\vee}: (A,B) \mapsto (\Psi_nA^{-1, T}\Psi_n, B)$ and $\Psi_n$ is the anti-diagonal $n \times n$-matrix (we have identified $F_{\nu^c} \simeq F_{\overline{\nu}}^+ \simeq F_{\nu}$). So to compute $\HC^{\ol{P}}$, we first consider the case of $\GL_n\times \GL_n$ inside $\GL_{2n}$, and then apply this involution $j_{\vee}$. 
    
   Since we chose $\gamma_{\tg}$ relative to the \textit{opposite Borel} $\ol{\mf{b}}$, it involves a shift by \[
    \delta_{\tg} = ((1 - 2n)/2, \dots, -1/2, 1/2, \dots, (2n-1)/2).
    \]

    Similarly, the Harish-Chandra map $\gamma_{\mf{gl}_n \times \mf{gl}_n}$ with respect to the opposite Borel involves a shift by 
    \[\delta_{\mf{gl}_n \times \mf{gl}_n} = ((1-n)/2, \dots, (n-1)/2, (1-n)/2, \dots, (n-1)/2).
    \]
    Since unnormalized Harish-Chandra maps compose well, we have the commutative diagram
        \[
    \begin{tikzcd}
        Z(\mf{gl}_{2n})\otimes_{\tau} L\ar[d,"{\HC^{\overline{P}}_{\nu, \tau}}"]\ar[r, "{\gamma_{\tg}}"]& L[x_1,\dots,x_{2n}]^{S_{2n}}\ar[d, "E'_{\nu, \tau}"]\\
        Z(\mf{gl}_n \times \mf{gl}_n)\otimes_{\tau}L\ar[r, "\gamma_{\mf{gl}_n \times \mf{gl}_n}"]& L[x_1,\dots, x_{2n}]^{S_n\times S_n},
    \end{tikzcd}
    \]
    where $E'_{\nu, \tau}: L[x_1,\dots, x_{2n}]^{S_{2n}}\to L[x_1,\dots, x_{2n}]^{S_{n}\times S_n}$ sends $x_i \mapsto x_i - n/2$ and $x_{i+n} \mapsto x_{i+n} + n/2$ for $1 \le i \le n$. Now note the isomorphism $j_{\vee}$ sends $x_i \mapsto -x_{n-i+1}$ and $x_{i+n} \mapsto x_{i+n}$ for $i\leq n$. Hence $E$ is given by $x_i \mapsto -x_{n-i+1}-n/2$ and $x_{i+n} \mapsto x_{i+n}+n/2$. Since the domain consists only of $S_{2n}$--invariant polynomials this is the same as the map in the statement.
\end{proof}

For any $\ul{a}= (a_{\nu,\tau, i})_{\nu,\tau} \in \bigoplus_{\nu \in \wt{S}, \tau: F_{\nu} \hookrightarrow L}L^{2n}$, we can define a variant of the Harish-Chandra homomorphism $\HC_{\ul{a}}^{\ol{P}}: Z(\wt{\mf{g}})_L\to Z(\mf{g})_L$ which fits into a diagram
    \[
    \begin{tikzcd}
        Z(\wt{\mf{g}})_L\ar[d,"{\HC_{\ul{a}}^{\overline{P}}}"]\ar[r, "{\gamma_{\tg}}"]& \bigotimes_{\nu, \tau}L[x_1,\dots,x_{2n}]^{S_{2n}}\ar[d, "E_{\ul{a}}"]\\
        Z(\mf{g})_L\ar[r, "\gamma_{\mf{g}}"]& \bigotimes_{\nu, \tau}L[x_1,\dots, x_{2n}]^{S_n\times S_n},
    \end{tikzcd}
    \]
where $E_{\ul{a}} = \bigotimes_{\nu,\tau}E_{\nu, \tau, \ul{a}}$ is defined to be the map $E_{\nu, \tau, \ul{a}}:L[x_1,\dots, x_{2n}]^{S_{2n}} \to L[x_1,\dots, x_{2n}]^{S_n\times S_n}$ sending $x_i \mapsto -x_i - n/2 + a_{\nu, \tau, i}$ and $x_{i+n} \mapsto x_{i+n} + n/2 + a_{\nu, \tau, i + n}$.

Let $\phi: Z(\mf{g})_L \rightarrow Z(\mf{g})_L$ be the map from the previous subsection (see around Lemma \ref{lemma: Galois twist}). It follows from the definitions and Lemma \ref{lemma: HC Explicit Polynomial Description} that $\phi(\HC^{\ol{P}}(Z(\wt{\mf{g}})_L))= \HC^{\ol{P}}_{\underline{a}}(Z(\wt{\mf{g}})_L)$, where $\ul{a}=(a_{\nu, \tau, i})$ with $a_{\nu, \tau ,i}= -1$ for all $\nu, \tau$ and $i\leq n$, and $a_{\nu, \tau ,i}= 1$ for all $\nu, \tau$ and $i> n$. Let us write $\on{HC}_{(\underbrace{-1,\ldots,-1}_{n}, \underbrace{1,\ldots,1}_{n})}^{\ol{P}}:= \on{HC}_{\ul{a}}^{\ol{P}}$.

This twisted Harish-Chandra map allows us to generate all of $Z(\mf{g})_L$.
\begin{lemma}
\label{lemma: JointSurjectivityHCMaps}
    The $L$--algebra map
    \[
    \HC^{\ol{P}}\otimes \HC_{(\underbrace{-1,\ldots,-1}_{n}, \underbrace{1,\ldots,1}_{n})}^{\ol{P}}: Z(\wt{\mf{g}})_L\otimes_{L} Z(\wt{\mf{g}})_L \to Z(\mf{g})_L
    \]
    is surjective.
\end{lemma}
\begin{proof}
        Let us write $a:=(\underbrace{-1,\ldots,-1}_{n}, \underbrace{1,\ldots,1}_{n})$. Using the Harish-Chandra isomorphisms $\gamma_{\tg}, \gamma_{\mf{g}}$, it suffices to show the map of polynomial algebras $E\otimes E_{\ul{a}}: \left(\bigotimes_{\nu, \tau}L[x_1,\dots, x_{2n}]^{S_{2n}}\right)^{\otimes 2} \to \bigotimes_{\nu, \tau} L[x_1,\dots, x_{2n}]^{S_n\times S_n}$ is surjective. It is enough to show that $E_{\nu, \tau}\otimes E_{\nu, \tau, \ul{a}}:\left(L[x_1,\dots, x_{2n}]^{S_{2n}}\right)^{\otimes 2} \to L[x_1,\dots, x_{2n}]^{S_n\times S_n}$ is surjective for each $\nu, \tau$. 

        Since $L$ is a field of characteristic $0$, $L[x_1,\dots, x_{2n}]^{S_{2n}}$ is generated as an algebra by polynomials $p_k(\ul{x})$ for $k\geq 1$ where $p_k(\ul{x}) = \sum_{i = 1}^{2n}x_i^k$ are the power basis elements. Similarly, $L[x_1,\dots, x_{2n}]^{S_{n}\times S_n}$ is generated by polynomials $q_k(\ul{x})$, $r_k(\ul{x})$ for $k\geq 0$, where $$q_k(\ul{x}) = \sum_{i = 1}^n(-x_i-n/2)^k= \sum_{i = 1}^n y_i^k, \ \ r_k(\ul{x}) = \sum_{i = 1}^{n}(x_{i+n} + n/2)^k= \sum_{i = 1}^ny_{i+n}^k$$ where $y_i:=-x_i-n/2$ and $y_{i+n}:=x_{i+n}+n/2$ for $i\leq n$. In terms of these bases, we are left to show that the polynomials $m_k(\ul{x}) := E(p_k(\ul{x})) = \sum_{i = 1}^{2n}y_i^k$ and $m_k^a(\ul{x}) := E_{a}(p_k(\ul{x})) = \sum_{i = 1}^n (y_i-1)^k + \sum_{i = 1}^n(y_{i + n} + 1)^k$ generate all $q_j(\ul{x}), r_j(\ul{x})$ for all $j\geq1$.

        We can prove this by induction on $j$. For $j = 1$, note that $m_2^a(\ul{x}) = m_2(\ul{x}) + 2\sum_{i = 1}^ny_{i+n} -2\sum_{i=1}^{n}y_i +2n = m_2(\ul{x})-2m_1(\ul{x})+2n+4r_1(\underline{x})$. Therefore $r_1(\underline{x})$ is in the image and so is $q_1(\underline{x})=m_1(\underline{x})-r_1(\ul{x})$. This finishes the base case.

        Now assuming that all $q_{\ell}, r_{\ell}$ for $1 \le \ell \le j-1$ lie in the image, we want to show this is true for $q_j$ and $r_j$ as well. Note
        \begin{equation}
        \label{eq: HCInductiveStep}
        m_{j+1}^a(\ul{x}) = m_{j+1}(\ul{x}) + (j+1)r_j(\ul{x})-(j+1)q_j(\ul{x})+r
        \end{equation}
        where $r$ is in the subalgebra generated by the $q_{\ell}, r_{\ell}$ for $1 \le \ell \le j-1$. Combined with $m_{j} = q_{j} + r_{j}$, we deduce that $q_j$ and $r_j$ lie in the image. This finishes the inductive step.
\end{proof}
We can now complete the proof of the main theorem.
\begin{theorem}[Main theorem \ref{theorem: Main Theorem}]\label{theorem: main theorem v2}
    Let $F$ be a CM field which contains an imaginary quadratic field in which $p$ splits.\footnote{See Remark \ref{remark: discussion on ass.} for further details on these assumptions and their (weaker) replacement.} Let $K^p \subset \GL_n(\mathbf A_{F}^{p,\infty})$ be a compact open and denote by $\wt{\HH}^i(K^p,\mathbf Q_p)$, $i \geq 0$, the corresponding completed cohomology groups for $\GL_n/F$. Let $\mf{m}\subset \T^S(K^p)$ be a non--Eisenstein decomposed generic maximal ideal. Then $Z(\mf{g}):=Z(\Res_{F/\Q}\mf{gl}_n)_{\Q_p}$ acts on $\wt{\HH}^i(K^p, \mathbf Q_p)^{\GL_n(F_p)-\on{la}}_{\mf{m}}$ via $\zeta^{C}:Z(\mf{g})\rightarrow \T^S(K^p)_{\mf{m}}^{\on{rig}}$ as defined in Section \ref{section: Compatibility between rhoTilde and rho}.
\end{theorem}
\begin{proof}
It suffices to prove the result after base--change to $L$, as usual. Using Theorem \ref{theorem: Harish Chandra + Inf Char compatibility} together with Corollary \ref{corollary: phi(HC)} and Lemma \ref{lemma: JointSurjectivityHCMaps}, we deduce that, for any $i$, $Z(\mf{g})_L$ acts on $\wt{\HH}^i(K^p,L)_{\mf{m}}^{\GL_n(F_p)-\on{la}}$ via $\zeta^C=\zeta^{C,-}$ as wanted.
\end{proof}
\begin{remark}[relation to Conjecture \ref{DPSconjecture}]
\label{Remark: actually proving the main theorem}
    Localizing further by any character $x:\T^S(K^p)_{\mf{m}}^{\on{rig}} \rightarrow \overline{\mathbf Q}_p$\footnote{Note that by \cite[Lem.\,7.1.9]{de1995crystalline} there is a bijection between $\on{m}-\Spec \T^S(K^p)_{\mf{m}}[1/p]$ and points of $\T^S(K^p)_{\mf{m}}^{\on{rig}}$.} proves the infinitesimal action of $Z(\mf{g})$ on the eigenspace $\wt{\HH}^i(K^p, \mathbf Q_p)_{\mf{m}}^{\on{la}}[\mf{m}_x]$ is given by $\zeta^C_{\rho_{\mf{m}_x}}$ attached to (usual $C$-group twist of) $\rho_{\mf{m}_x}: \Gal_{F,S}\rightarrow \GL_n(\overline{\mathbf Q}_p)$, as in \cite{DPS25}, where $\rho_{\mf{m}_x}$ is the Galois representation given by specializing Theorem \ref{theorem: Galois representation for GLn} with respect to $x$.
\end{remark}
\newpage

\bibliographystyle{alpha}
\bibliography{References}

@article {AC24,
    AUTHOR = {A'Campo, Lambert},
     TITLE = {Rigidity of automorphic {G}alois representations over {CM}
              fields},
   JOURNAL = {Int. Math. Res. Not. IMRN},
  FJOURNAL = {International Mathematics Research Notices. IMRN},
      YEAR = {2024},
    NUMBER = {6},
     PAGES = {4541--4623},
      ISSN = {1073-7928,1687-0247},
   MRCLASS = {11F80 (11F75)},
  MRNUMBER = {4721650},
       DOI = {10.1093/imrn/rnad087},
       URL = {https://doi.org/10.1093/imrn/rnad087},
}

@article {10AuthorPaper,
    AUTHOR = {Allen, Patrick B. and Calegari, Frank and Caraiani, Ana and
              Gee, Toby and Helm, David and Le Hung, Bao V. and Newton,
              James and Scholze, Peter and Taylor, Richard and Thorne, Jack
              A.},
     TITLE = {Potential automorphy over {CM} fields},
   JOURNAL = {Ann. of Math. (2)},
  FJOURNAL = {Annals of Mathematics. Second Series},
    VOLUME = {197},
      YEAR = {2023},
    NUMBER = {3},
     PAGES = {897--1113},
      ISSN = {0003-486X,1939-8980},
   MRCLASS = {11F80 (11F55 11F75 11G18)},
  MRNUMBER = {4564261},
       DOI = {10.4007/annals.2023.197.3.2},
       URL = {https://doi.org/10.4007/annals.2023.197.3.2},
}

@article {BLGGT11,
    AUTHOR = {Barnet-Lamb, Thomas and Gee, Toby and Geraghty, David and
              Taylor, Richard},
     TITLE = {Local-global compatibility for {$l=p$}, {II}},
   JOURNAL = {Ann. Sci. \'Ec. Norm. Sup\'er. (4)},
  FJOURNAL = {Annales Scientifiques de l'\'Ecole Normale Sup\'erieure.
              Quatri\`eme S\'erie},
    VOLUME = {47},
      YEAR = {2014},
    NUMBER = {1},
     PAGES = {165--179},
      ISSN = {0012-9593,1873-2151},
   MRCLASS = {11F70 (11R39 22Exx)},
  MRNUMBER = {3205603},
MRREVIEWER = {Nguy\cftil en Qu\^oc Th\'ang},
       DOI = {10.24033/asens.2212},
       URL = {https://doi.org/10.24033/asens.2212},
}

@article {BHSAnnalen,
    AUTHOR = {Breuil, Christophe and Hellmann, Eugen and Schraen, Benjamin},
     TITLE = {Une interpr\'{e}tation modulaire de la vari\'{e}t\'{e}
              trianguline},
   JOURNAL = {Math. Ann.},
  FJOURNAL = {Mathematische Annalen},
    VOLUME = {367},
      YEAR = {2017},
    NUMBER = {3-4},
     PAGES = {1587--1645},
      ISSN = {0025-5831,1432-1807},
   MRCLASS = {11F85 (14D24 14F30)},
  MRNUMBER = {3623233},
       DOI = {10.1007/s00208-016-1422-1},
       URL = {https://doi.org/10.1007/s00208-016-1422-1},
}

@article{buzzard2014conjectural,
  title={The conjectural connections between automorphic representations and {G}alois representations},
  author={Buzzard, Kevin and Gee, Toby},
  journal={Automorphic forms and Galois representations},
  volume={1},
  pages={135--187},
  year={2014}
}

@article {CalEm09,
    AUTHOR = {Calegari, Frank and Emerton, Matthew},
     TITLE = {Bounds for multiplicities of unitary representations of
              cohomological type in spaces of cusp forms},
   JOURNAL = {Ann. of Math. (2)},
  FJOURNAL = {Annals of Mathematics. Second Series},
    VOLUME = {170},
      YEAR = {2009},
    NUMBER = {3},
     PAGES = {1437--1446},
      ISSN = {0003-486X,1939-8980},
   MRCLASS = {22E55 (11F70 11F75)},
  MRNUMBER = {2600878},
MRREVIEWER = {Neven\ Grbac},
       DOI = {10.4007/annals.2009.170.1437},
       URL = {https://doi.org/10.4007/annals.2009.170.1437},
}

@incollection {CalEmSurvey,
    AUTHOR = {Calegari, Frank and Emerton, Matthew},
     TITLE = {Completed cohomology---a survey},
 BOOKTITLE = {Non-abelian fundamental groups and {I}wasawa theory},
    SERIES = {London Math. Soc. Lecture Note Ser.},
    VOLUME = {393},
     PAGES = {239--257},
 PUBLISHER = {Cambridge Univ. Press, Cambridge},
      YEAR = {2012},
      ISBN = {978-1-107-64885-2},
   MRCLASS = {11F85 (11F70 11F75)},
  MRNUMBER = {2905536},
MRREVIEWER = {Neven\ Grbac},
}

@article{caraiani2020shimura,
  title={Shimura varieties at level {$\Gamma_1(p^{\infty})$} and {G}alois representations},
  author={Caraiani, Ana and Gulotta, Daniel R and Hsu, Chi-Yun and Johansson, Christian and Mocz, Lucia and Reinecke, Emanuel and Shih, Sheng-Chi},
  journal={Compositio Mathematica},
  volume={156},
  number={6},
  pages={1152--1230},
  year={2020},
  publisher={London Mathematical Society}
}

@misc{CN2023,
      title={On the modularity of elliptic curves over imaginary quadratic fields},
      author={Ana Caraiani and James Newton},
      year={2023},
      note={\href{https://arxiv.org/abs/2301.10509}{arXiv:2301.10509}},
      eprint={2301.10509},
      archivePrefix={arXiv},
      primaryClass={math.NT}
}

@article{caraiani2023vanishing,
  title={Vanishing theorems for {S}himura varieties at unipotent level},
  author={Caraiani, Ana and Gulotta, Daniel R and Johansson, Christian},
  journal={Journal of the European Mathematical Society},
  volume={25},
  number={3},
  pages={869--911},
  year={2023},
  publisher={European Mathematical Society}
}

@article {CS24,
    AUTHOR = {Caraiani, Ana and Scholze, Peter},
     TITLE = {On the generic part of the cohomology of non-compact unitary
              {S}himura varieties},
   JOURNAL = {Ann. of Math. (2)},
  FJOURNAL = {Annals of Mathematics. Second Series},
    VOLUME = {199},
      YEAR = {2024},
    NUMBER = {2},
     PAGES = {483--590},
      ISSN = {0003-486X,1939-8980},
   MRCLASS = {11R39 (14G35 14G45)},
  MRNUMBER = {4713019},
MRREVIEWER = {Nguy\cftil en Qu\^oc Th\'ang},
       DOI = {10.4007/annals.2024.199.2.1},
       URL = {https://doi.org/10.4007/annals.2024.199.2.1},
}

@incollection {Che14,
    AUTHOR = {Chenevier, Ga\"etan},
     TITLE = {The {$p$}-adic analytic space of pseudocharacters of a
              profinite group and pseudorepresentations over arbitrary
              rings},
 BOOKTITLE = {Automorphic forms and {G}alois representations. {V}ol. 1},
    SERIES = {London Math. Soc. Lecture Note Ser.},
    VOLUME = {414},
     PAGES = {221--285},
 PUBLISHER = {Cambridge Univ. Press, Cambridge},
      YEAR = {2014},
      ISBN = {978-1-107-69192-6},
   MRCLASS = {11F70 (11F80 14G22 20E18)},
  MRNUMBER = {3444227},
MRREVIEWER = {Cameron\ Franc},
}

@article{de1995crystalline,
  title={Crystalline {D}ieudonn{\'e} module theory via formal and rigid geometry},
  author={De Jong, A. J.},
  journal={Publications Math{\'e}matiques de l'IH{\'E}S},
  volume={82},
  pages={5--96},
  year={1995}
}

@article {DPS23,
    AUTHOR = {Dospinescu, Gabriel and Pa\v{s}k\={u}nas, Vytautas and Schraen,
              Benjamin},
     TITLE = {Gelfand-{K}irillov dimension and the {$p$}-adic
              {J}acquet-{L}anglands correspondence},
   JOURNAL = {J. Reine Angew. Math.},
  FJOURNAL = {Journal f\"ur die Reine und Angewandte Mathematik. [Crelle's
              Journal]},
    VOLUME = {801},
      YEAR = {2023},
     PAGES = {57--114},
      ISSN = {0075-4102,1435-5345},
   MRCLASS = {22E50 (11S37)},
  MRNUMBER = {4621880},
MRREVIEWER = {Dongwen\ Liu},
       DOI = {10.1515/crelle-2023-0033},
       URL = {https://doi.org/10.1515/crelle-2023-0033},
}

@article {DPS25,
    AUTHOR = {Dospinescu, Gabriel and Pa\v{s}k\={u}nas, Vytautas and Schraen,
              Benjamin},
     TITLE = {Infinitesimal characters in arithmetic families},
   JOURNAL = {Selecta Math. (N.S.)},
  FJOURNAL = {Selecta Mathematica. New Series},
    VOLUME = {31},
      YEAR = {2025},
    NUMBER = {4},
     PAGES = {Paper No. 76, 77},
      ISSN = {1022-1824,1420-9020},
   MRCLASS = {22E50 (11F80 11F85)},
  MRNUMBER = {4940362},
       DOI = {10.1007/s00029-025-01045-6},
       URL = {https://doi.org/10.1007/s00029-025-01045-6},
}

@misc{emerson2018comparison,
  title={Comparison of different definitions of pseudocharacters},
  author={Emerson, Kathleen and Morel, Sophie},
  year={2023},
  note={\href{https://arxiv.org/abs/2310.03869}{arXiv:2310.03869}}
}

@article {Emerton_Jacquet_I,
    AUTHOR = {Emerton, Matthew},
     TITLE = {Jacquet modules of locally analytic representations of
              {$p$}-adic reductive groups. {I}. {C}onstruction and first
              properties},
   JOURNAL = {Ann. Sci. \'{E}cole Norm. Sup. (4)},
  FJOURNAL = {Annales Scientifiques de l'\'{E}cole Normale Sup\'{e}rieure.
              Quatri\`eme S\'{e}rie},
    VOLUME = {39},
      YEAR = {2006},
    NUMBER = {5},
     PAGES = {775--839},
      ISSN = {0012-9593},
   MRCLASS = {22E50},
  MRNUMBER = {2292633},
MRREVIEWER = {Bertrand\ Lemaire},
       DOI = {10.1016/j.ansens.2006.08.001},
       URL = {https://doi.org/10.1016/j.ansens.2006.08.001},
}

@article{emerton2006interpolation,
  title={On the interpolation of systems of eigenvalues attached to automorphic {H}ecke eigenforms},
  author={Emerton, Matthew},
  journal={Inventiones mathematicae},
  volume={164},
  number={1},
  pages={1--84},
  year={2006},
  publisher={Springer-Verlag, Berlin/Heidelberg}
}

@article {Emerton_Jacquet_II,
    AUTHOR = {Emerton, Matthew},
     TITLE = {Jacquet modules of locally analytic representations of
              {$p$}-adic reductive groups. {II}. {T}he relation to parabolic induction.},
   JOURNAL = {To appear in  J. Institut Math. Jussieu.},
  FJOURNAL = {To appear in  J. Institut Math. Jussieu.},
      YEAR = {2007},
}

@article{emerton2020density,
  title={On the density of supercuspidal points of fixed regular weight in local deformation rings and global {H}ecke algebras},
  author={Emerton, Matthew and Pa{\v{s}}k{\=u}nas, Vytautas},
  journal={Journal de l'{\'E}cole polytechnique Math{\'e}matiques},
  volume={7},
  pages={337--371},
  year={2020}
}

@misc{FuDerived,
      title={A derived construction of eigenvarieties},
      author={Weibo Fu},
      year={2022},
      note={\href{https://arxiv.org/abs/2110.04797}{arXiv:2110.04797}},
      eprint={2110.04797},
      archivePrefix={arXiv},
      primaryClass={math.NT}
}

@article {GeeNewton,
    AUTHOR = {Gee, Toby and Newton, James},
     TITLE = {Patching and the completed homology of locally symmetric
              spaces},
   JOURNAL = {J. Inst. Math. Jussieu},
  FJOURNAL = {Journal of the Institute of Mathematics of Jussieu. JIMJ.
              Journal de l'Institut de Math\'{e}matiques de Jussieu},
    VOLUME = {21},
      YEAR = {2022},
    NUMBER = {2},
     PAGES = {395--458},
      ISSN = {1474-7480,1475-3030},
   MRCLASS = {11F75 (11F80 22E50)},
  MRNUMBER = {4386819},
MRREVIEWER = {Atsushi\ Yamagami},
       DOI = {10.1017/S1474748020000158},
       URL = {https://doi.org/10.1017/S1474748020000158},
}

@article {HJ23,
    AUTHOR = {Hansen, David and Johansson, Christian},
     TITLE = {Perfectoid {S}himura varieties and the {C}alegari-{E}merton
              conjectures},
   JOURNAL = {J. Lond. Math. Soc. (2)},
  FJOURNAL = {Journal of the London Mathematical Society. Second Series},
    VOLUME = {108},
      YEAR = {2023},
    NUMBER = {5},
     PAGES = {1954--2000},
      ISSN = {0024-6107,1469-7750},
   MRCLASS = {11S37 (14G45)},
  MRNUMBER = {4668521},
       DOI = {10.1112/jlms.12799},
       URL = {https://doi.org/10.1112/jlms.12799},
}

@article {HLTT16,
    AUTHOR = {Harris, Michael and Lan, Kai-Wen and Taylor, Richard and
              Thorne, Jack},
     TITLE = {On the rigid cohomology of certain {S}himura varieties},
   JOURNAL = {Res. Math. Sci.},
  FJOURNAL = {Research in the Mathematical Sciences},
    VOLUME = {3},
      YEAR = {2016},
     PAGES = {Paper No. 37, 308},
      ISSN = {2522-0144,2197-9847},
   MRCLASS = {11G18 (11F75 14G35)},
  MRNUMBER = {3565594},
MRREVIEWER = {Fausto\ Jarqu\'in Z\'arate},
       DOI = {10.1186/s40687-016-0078-5},
       URL = {https://doi.org/10.1186/s40687-016-0078-5},
}

@misc{Hevesi23,
      title={Ordinary parts and local-global compatibility at $\ell=p$},
      author={Bence Hevesi},
      year={2023},
      note={\href{https://arxiv.org/abs/2311.13514}{arXiv:2311.13514}},
      eprint={2311.13514},
      archivePrefix={arXiv},
      primaryClass={math.NT},
      url={https://arxiv.org/abs/2311.13514},
}

@misc{JNWE24,
      title={Moduli stacks of {G}alois representations and the $p$-adic local {L}anglands correspondence for $\mathrm{GL}_2(\mathbb{Q}_p)$},
      author={Christian Johansson and James Newton and Carl Wang-Erickson},
      note={\href{https://arxiv.org/abs/2403.19565}{arXiv:2403.19565}, \textit{to appear in Compositio Mathematica}},
      year={2024},
      eprint={2403.19565},
      archivePrefix={arXiv},
      primaryClass={math.NT},
      url={https://arxiv.org/abs/2403.19565},
}

@article{khare2017potential,
  title={Potential automorphy and the {L}eopoldt conjecture},
  author={Khare, Chandrashekhar B and Thorne, Jack A},
  journal={American Journal of Mathematics},
  volume={139},
  number={5},
  pages={1205--1273},
  year={2017},
  publisher={Johns Hopkins University Press}
}

@misc{KoshikawaGeneric,
      title={On the generic part of the cohomology of local and global {S}himura varieties},
      author={Teruhisa Koshikawa},
      year={2021},
      note={\href{https://arxiv.org/abs/2106.10602}{arXiv:2106.10602}},
      eprint={2106.10602},
      archivePrefix={arXiv},
      primaryClass={math.NT}
}

@misc{mcdonald2025eigenvarieties,
  title={Eigenvarieties over {CM} fields and trianguline representations},
  author={McDonald, Vaughan},
  note={\href{https://arxiv.org/abs/2504.18319}{arXiv:2504.18319}},
  year={2025}
}

@article {NT16,
    AUTHOR = {Newton, James and Thorne, Jack A.},
     TITLE = {Torsion {G}alois representations over {CM} fields and {H}ecke
              algebras in the derived category},
   JOURNAL = {Forum Math. Sigma},
  FJOURNAL = {Forum of Mathematics. Sigma},
    VOLUME = {4},
      YEAR = {2016},
     PAGES = {Paper No. e21, 88},
      ISSN = {2050-5094},
   MRCLASS = {11F80 (11F75)},
  MRNUMBER = {3528275},
MRREVIEWER = {Alan\ Koch},
       DOI = {10.1017/fms.2016.16},
       URL = {https://doi.org/10.1017/fms.2016.16},
}

@misc{PasQua25,
      title={Infinitesimal characters and {L}afforgue's pseudocharacters},
      author={Vytautas Pa\v{s}k\={u}nas and Julian Quast},
      note={\href{https://arxiv.org/abs/2505.03544}{arXiv:2505.03544}},
      year={2025},
      eprint={2505.03544},
      archivePrefix={arXiv},
      primaryClass={math.NT},
      url={https://arxiv.org/abs/2505.03544},
}

@article{quast2026deformations,
  title={Deformations of {G}-valued pseudocharacters},
  author={Quast, Julian},
  journal={Peking Mathematical Journal},
  pages={1--59},
  year={2026},
  publisher={Springer}
}

@article{schneider2002banach,
  title={Banach space representations and {I}wasawa theory},
  author={Schneider, Peter and Teitelbaum, Jeremy},
  journal={Israel {J}ournal of {M}athematics},
  volume={127},
  number={1},
  pages={359--380},
  year={2002},
  publisher={Springer}
}

@article{schneider2002locally,
  title={Locally analytic distributions and $p$-adic representation theory, with applications to $\mathrm{GL}_2(\mathbf{Q}_p)$},
  author={Schneider, Peter and Teitelbaum, Jeremy},
  journal={Journal of the American Mathematical Society},
  volume={15},
  number={2},
  pages={443--468},
  year={2002}
}

@article {ST03,
    AUTHOR = {Schneider, Peter and Teitelbaum, Jeremy},
     TITLE = {Algebras of {$p$}-adic distributions and admissible
              representations},
   JOURNAL = {Invent. Math.},
  FJOURNAL = {Inventiones Mathematicae},
    VOLUME = {153},
      YEAR = {2003},
    NUMBER = {1},
     PAGES = {145--196},
      ISSN = {0020-9910,1432-1297},
   MRCLASS = {22E50},
  MRNUMBER = {1990669},
MRREVIEWER = {Marko\ Tadi\'c},
       DOI = {10.1007/s00222-002-0284-1},
       URL = {https://doi.org/10.1007/s00222-002-0284-1},
}

@article {Sch15,
    AUTHOR = {Scholze, Peter},
     TITLE = {On torsion in the cohomology of locally symmetric varieties},
   JOURNAL = {Ann. of Math. (2)},
  FJOURNAL = {Annals of Mathematics. Second Series},
    VOLUME = {182},
      YEAR = {2015},
    NUMBER = {3},
     PAGES = {945--1066},
      ISSN = {0003-486X,1939-8980},
   MRCLASS = {11S37},
  MRNUMBER = {3418533},
MRREVIEWER = {Kimball\ L.\ Martin},
       DOI = {10.4007/annals.2015.182.3.3},
       URL = {https://doi.org/10.4007/annals.2015.182.3.3},
}

@article {Shin14,
    AUTHOR = {Shin, Sug Woo},
     TITLE = {On the cohomological base change for unitary
similitude groups},
   JOURNAL = {Compos. Math.},
  FJOURNAL = {Compositio Mathematica},
    VOLUME = {150},
      YEAR = {2014},
    NUMBER = {2},
     PAGES = {191--228},
      ISSN = {0010-437X,1570-5846},
   MRCLASS = {11F33 (11F55 11F70)},
  MRNUMBER = {3177267},
MRREVIEWER = {Luis\ Alberto\ Lomel\'i},
       DOI = {10.1112/S0010437X13007355},
       URL = {https://doi.org/10.1112/S0010437X13007355},
      NOTE = {Appendix to \textit{Galois representations associated to holomorphic limits of discrete series,} by Wushi Goldring}
}

@misc{stacks-project,
	author = {The {Stacks Project Authors}},
	howpublished = {\url{https://stacks.math.columbia.edu}},
	shorthand = {Stacks},
	title = {\textit{Stacks Project}},
	year = {2018}}

\end{document}